\documentclass[12pt]{paper}
  \usepackage{geometry}
\usepackage{hyperref}
                \usepackage{amsthm,amsfonts,amsmath,amscd,amssymb,epsfig,verbatim,
}
\usepackage{xcolor}
\usepackage[all]{xy}

\usepackage{graphicx}
\usepackage{epstopdf}
 {\title{Existence and classification of overtwisted contact structures    in  all dimensions} 
\author{Matthew Strom  Borman\thanks{Partially supported by NSF grant   DMS-1304252} \\ Stanford University\and Yakov Eliashberg\thanks{Partially supported by  NSF grant DMS-1205349} \\ Stanford University  \and   Emmy Murphy
\\ MIT} 
\date{}  
 \setlength{\marginparwidth}{1.2in}
\let\oldmarginpar\marginpar
\renewcommand\marginpar[1]{\-\oldmarginpar[\raggedleft\footnotesize #1]%
{\raggedright\footnotesize #1}}

\parindent=0pt
\parskip=4pt
\hyphenation{ma-ni-fold ma-ni-folds sub-ma-ni-fold sub-ma-ni-folds}
\theoremstyle{plain}
\newtheorem{theorem}{Theorem}[section]

\newtheorem{corollary}[theorem]{Corollary}
\newtheorem{cor}[theorem]{Corollary}
\newtheorem{proposition}[theorem]{Proposition}
\newtheorem{prop}[theorem]{Proposition}
\newtheorem{lemma}[theorem]{Lemma}

\newtheorem{ex}[theorem]{Example}

\newtheorem{remark}[theorem]{Remark}
\newtheorem*{remark*}{Remark}

\theoremstyle{definition}
\newtheorem{definition}[theorem]{Definition}
\newcommand{\id}{{{\mathchoice {\rm 1\mskip-4mu l} {\rm 1\mskip-4mu l}
{\rm 1\mskip-4.5mu l} {\rm 1\mskip-5mu l}}}}

\newcommand{\wt}{\widetilde}
\newcommand{\wh}{\widehat}
\newcommand{\ol}{\overline}

\newcommand{\p}{\partial}

\newcommand{\eps}{\varepsilon}

\newcommand{\abs}[1]{\left| #1 \right|}

\newcommand{\Z}{{\mathbb{Z}}}
\newcommand{\R}{{\mathbb{R}}}

\newcommand{\A}{{\bf A}}
\newcommand{\B}{{\bf B}}

\newcommand{\st}{{\rm st}}

\newcommand{\cyl}{{\rm cyl}}

\newcommand{\Int}{{\rm Int\,}} 

\renewcommand{\min}{{\rm min}}
\renewcommand{\max}{{\rm max}}

\newcommand{\std}{{\rm std}}

\newcommand{\supp}{{\rm supp}}

\newcommand{\Id}{\mathrm {Id}}

\newcommand{\univ}{\mathrm{univ}}

\newcommand{\Diff}{\mathrm{Diff}}

\newcommand{\Supp}{\mathrm{Support}}

\newcommand{\stab}{\mathrm{stab}}

\newcommand{\Reeb}{\frak{R}}

\newcommand{\Cont}{\mathfrak{Cont}}
\newcommand{\cont}{\mathfrak{cont}}
\newcommand{\ot}{\mathrm{ot}}
\newcommand{\CCont}{\mathrm{Cont}}

\newcommand{\cD}{{\frak D}}

\newcommand{\CC}{\mathcal{C}}

\newcommand{\FF}{\mathcal{F}}

\newcommand{\TT}{\mathcal{T}}

\def\Op{{\mathcal O}{\it p}\,}
\numberwithin{figure}{section}

\begin{document}
\maketitle
\centerline{\small{\em To Misha Gromov with admiration}}

\begin{abstract}
We establish a parametric extension $h$-principle for overtwisted contact structures on manifolds of all dimensions,  which is the direct generalization of the $3$-dimensional result from \cite{Eli89}. It implies, in particular,  that any closed   manifold admits a contact structure in any given homotopy class of almost contact structures. 
\end{abstract}
\bigskip

\section{Introduction}
A {\em contact structure} on a  $(2n+1)$-dimensional manifold $M$  is a completely non-integrable hyperplane field $\xi\subset TM$. Defining $\xi$ by a Pfaffian equation $\{\alpha=0\}$ where $\alpha $ is a 1-form, possibly with coefficients in a local system for a non-coorientable $\xi$, then the complete non-integrability is equivalent to $\alpha\wedge d\alpha^n$ being non-vanishing on $M$. An equivalent definition of the contact condition is that the complement of the $0$-section of the total space of the  conormal bundle $L_\xi\subset T^*M$ is a symplectic submanifold of $T^*M$ with its canonical symplectic structure $d(pdq)$.

The corresponding formal homotopy counterpart of a contact structure is an {\it almost contact structure}, which  is   a defined up to a scalar factor  pair  $(\alpha,\omega)$ where $\alpha$ is non-vanishing 1-form on $M$, possibly with local coefficients in a non-trivial 1-bundle, and $\omega$ is a non-degenerate two-form on the hyperplane field $\xi=\{\alpha=0\}$  with coefficients in the same local system.  Thus, both  in the contact and almost contact cases, the hyperplane field is endowed with a conformal class of symplectic structures.  In the co-orientable case, i.e.\ when  $TM/\xi$ is trivialized by $\alpha$, the existence of an almost contact structure is equivalent to the existence of a {\it stable almost complex structure} on $M$, i.e.\ a complex structure on the bundle $TM\oplus\eps^1$ where $\eps^1$ is the trivial line bundle over $M$.

The current paper concerns with basic topological questions about contact structures: existence, extension and homotopy. 
This problem  has a long history. It was first explicitly formulated, probably, in S.S.~Chern's paper \cite{Chern66}.
In 1969 M.~Gromov \cite{Gro69} proved a parametric $h$-principle for contact structures on an {\em open} manifold $M$: {\it any almost contact structure is homotopic to a genuine one, and two contact structures are homotopic if they are homotopic as almost contact structures, }  see
Theorem \ref{thm:Gromov-open} below for a more precise formulation of Gromov's theorem.

For closed manifolds a  lot of progress  was achieved in the $3$-dimensional case beginning from the work of J. Martinet \cite{Martinet} and  R. Lutz \cite{Lutz} who  solved the non-parametric existence  problem for $3$-manifolds. 
D. Bennequin \cite{Bennequin} showed that the 1-parametric $h$-principle fails for contact structures on $S^3$ and    Y.~Eliashberg  in \cite{Eli89} introduced a dichotomy of $3$-dimensional contact manifolds into {\it tight} and  {\it overtwisted} 
and established a parametric $h$-principle for overtwisted ones:
{\em any almost contact homotopy class on a closed $3$-manifold contains a unique up to isotopy overtwisted contact structure.} Tight contact structures were also classified on several classes of 3--manifolds, see e.g.\ \cite{Eli92,Gi00, Ho00,Ho00-2}. V.~Colin, E.~Giroux and K.~Honda proved in   \cite{CoGiHo}  that  any atoroidal contact 3-manifold admits at most finitely many non-isotopic tight contact structures.

Significant progress in the problem of construction of contact structures on closed manifolds was achieved in the $5$-dimensional  case beginning from the work of H.~Geiges \cite{Ge91, Ge97} and H.~Geiges and  C.B.~Thomas  \cite{GeTh98, GeTh01},  and followed by the work of 
R.~Casals, D.M.~Pancholi and F.~Presas \cite{CaPaPr13} and J.~Etnyre \cite{Etn13}, where there was established existence of  contact structures on any $5$-manifolds  in any homotopy class of almost contact structures.
For manifolds of dimension $>5$ the results are more scarce.
The work \cite{Eli90} implied existence of contact structures on all closed
$(2n+1)$-dimensional manifolds that bound almost complex manifolds with the homotopy type of $(n+1)$-dimensional cell complexes, provided $n \geq 2$.
F.~Bourgeois \cite{Bourg02} proved that for any closed contact manifold $M$ and any surface $\Sigma$ with genus at least one, the product $M\times\Sigma$ admits a contact structure, using work of E.~Giroux~\cite{Gir02G}.
This positively answered a long standing problem about existence of contact structures on tori of dimension $2n+1>5$ (a contact structure on $T^5$ was first constructed by R. Lutz in \cite{Lutz-torus}).

Non-homotopic, but  formally homotopic contact structures were constructed   on higher dimensional manifolds as well, see e.g. \cite{Ust}. As far as we know,  before the current paper  there were no known  general results concerning extension of contact structures in dimension greater than three.  

    \begin{theorem}\label{thm:main-existence} Let $M$ be a $(2n+1)$-manifold, $A\subset M$ be a closed set, and
     $\xi$ be an almost contact structure on $M$.
    If $\xi$ is genuine on $\Op A\subset M$ then $\xi$ is homotopic relative to $A$ to a genuine contact structure.
    In particular, any almost contact structure on a closed manifold is homotopic to a genuine contact structure.
    \end{theorem}
    Here we are using Gromov's notation $\Op\,A$  for any unspecified open neighborhood of a closed subset $A\subset M$.

    In Section \ref{sec:main-existence} we will define the notion of an {\em overtwisted} contact structure
    for any odd dimensional manifold.   Deferring the definition until Section \ref{sec:ot-disc}, we will say here that a contact manifold $(M^{2n+1},\xi)$ is called overtwisted if it admits a contact embedding of a piecewise smooth $2n$-disc $D_\ot$ with a certain model germ $\zeta_\ot$ of a contact structure.
      In the $3$-dimensional case this notion is equivalent  to the standard notion introduced in \cite{Eli89}. 
See Section \ref{sec:discussion} for further discussion of the overtwisting property.

Given a $(2n+1)$-dimensional manifold $M$, let $A$ be a closed subset such that $M \setminus A$ is connected, and let $\xi_0$ be an almost contact structure $M$
that is a genuine contact structure on $\Op A$.  Define
$\Cont_\ot(M;A,\xi_0)$ to be the space of contact structures on $M$ that are overtwisted on $M\setminus A$ and coincide with 
$\xi_0$ on $\Op A$.
The notation   $\cont(M;A,\xi_0)$ stands for the space of {\em almost} contact structures
that agree with $\xi_0$ on $\Op A$. Let 
$$j:\Cont_\ot(M;A,\xi_0)\to\cont(M;A,\xi_0)$$  be  the inclusion map.
For an embedding $\phi:D_\ot\to M\setminus A$, let 
$\Cont_\ot(M;A,\xi_0,\phi)$ and $\cont_\ot(M;A,\xi_0,\phi)$ be the subspaces
   of  $\Cont_\ot(M;A,\xi_0)$ and $\cont_{\ot}(M;A,\xi_0)$ of contact and almost contact structures for which 
   $\phi:(D_\ot,\zeta_\ot)\to(M,\xi)$ is a contact embedding.
   
\begin{theorem}\label{thm:main}
 The inclusion map induces an isomorphism
 $$j_*:\pi_0(\Cont_\ot(M;A,\xi_0))\to\pi_0(\cont(M;A,\xi_0))$$
 and moreover the map
 $$j:\Cont_\ot(M;A,\xi_0,\phi)\to\cont_\ot(M;A,\xi_0,\phi)$$ is a (weak) homotopy  equivalence. 
 \end{theorem}
 
 As an immediate corollary we have the following
  
  \begin{corollary}\label{cor:main}
  On any closed manifold $M$ any almost contact structure is homotopic to an overtwisted contact structure which is unique up to isotopy. 
  \end{corollary}
  
  We also have the following corollary (see Section \ref{sec:proof-main} for the proof) concerning isocontact embeddings into an overtwisted contact manifold.
\begin{cor}\label{cor:isocontact}
  Let  $(M^{2n+1},\xi)$ be a connected overtwisted contact manifold and 
  let $(N^{2n+1},\zeta)$ be an open contact manifold of the same dimension.
 Let $f:N\to M$ be a smooth embedding covered by a   contact bundle homomorphism $\Phi:TN\to TM$, that is $\Phi(\zeta_x) = \xi|_{f(x)}$ and $\Phi$ preserves the conformal symplectic structures on $\zeta$ and $\xi$.  If $df$ and $\Phi$ are homotopic as injective bundle homomorphisms $TN\to TM$, then $f$ is isotopic to a contact embedding $\wt f:(N,\zeta)\to (M,\xi)$.  In particular, an open ball with any contact structure embeds into any overtwisted contact manifold of the same dimension.
\end{cor}
 
We note that there were many proposals for defining the overtwisting phenomenon in dimension greater than three.
We claim that our notion is stronger than any other possible notions, in the sense that any exotic phenomenon, e.g.\ a plastikstufe \cite{Ni06}, can be found in any overtwisted contact manifold.
Indeed suppose we are given some exotic model $(A, \zeta)$, which is an open  contact manifold, and assume it formally embeds into an equidimensional $(M, \xi_{ot})$, then by Corollary~\ref{cor:isocontact}  $(A, \zeta)$ admits a genuine   contact
embedding into $(M, \xi_{ot})$.

In particular, the known results about contact manifolds with a plastikstufe apply to overtwisted manifolds as well:
\begin{itemize}{\em
\item Overtwisted contact manifolds are not (semi-positively) symplectically fillable \cite{Ni06};
\item The Weinstein conjecture holds for any contact form defining an overtwisted contact structure on a closed manifold
\cite{AlHo09};
\item Any Legendrian submanifold whose complement is overtwisted is loose
 \cite{MuNiPlaSti12}. Conversely, any loose Legendrian in an overtwisted ambient manifold has  an overtwisted complement.}
\end{itemize}
   
     As it is customary in the $h$-principle type framework, a parametric $h$-principle yields results about leafwise structures on foliations,  see e.g.\ \cite{Gro69}. In  particular, in \cite{CPP14} the  parametric $h$-principle \cite{Eli89} for overtwisted contact structures on a $3$-manifold  was used for the construction of leafwise contact structures on  codimension one foliations on $4$-manifolds.

 Let $\FF$ be a  smooth $(2n+1)$-dimensional foliation  on a manifold $V$ of dimension $m=2n+1+q$. 
 \begin{theorem}\label{thm:existence-fol}
 Any leafwise almost contact structure on $\FF$ is homotopic to a genuine leafwise contact structure.
 \end{theorem} 

A leafwise contact structure $\xi$ on a codimension $q$ foliation $\FF$ on a manifold $V$ of dimension $2n+1+q$ is called {\em overtwisted} if there exist disjoint embeddings 
$$
\mbox{$h_i:T_i\times B \to V$ for $i=1,\dots, N$,}
$$
where $(B, \zeta)$ is a $(2n+1)$-dimensional overtwisted contact ball and each $T_i$ is a compact $q$-dimensional manifold with boundary,
such that 
\begin{itemize}
\item each leaf of $\FF$ is intersected by one of these embeddings, and
\item for each $i=1,\dots, N$ and $\tau\in T_i$ the restriction $h_i|_{\tau \times B}$ is a contact embedding
of $(B, \zeta)$ into some leaf of $\FF$ with its contact structure.
\end{itemize}
The set of embeddings $h_1,\dots, h_N$ is called an {\em overtwisted  basis} of the overtwisted leafwise contact structure $\xi$ on $\FF$.

For a closed subset $A \subset V$, let $\xi_0$ be a leafwise contact structure on $\FF|_{\Op A}$, and let
$h_i:T_i\times B \to V\setminus A$ for $i=1,\dots, N$ be a collection of disjoint embeddings.
Define $$\Cont_\ot(\FF;A,\xi_0,h_1,\dots, h_N)$$ to be the space of leafwise contact structures   $\FF$ that coincide with 
$\xi_0$ over $\Op A$ and such that $\{h_i\}_{i\in\{1,\dots, N\}}$ is an overtwisted basis for $\FF_{V\setminus A}$.  Define 
$$\cont_\ot(\FF;A,\xi_0,h_1,\dots, h_n))$$ to be the analogous space of leafwise almost contact structures on $\FF$.
 
 \begin{theorem}\label{thm:classif-fol}
The inclusion map  $$\Cont_\ot(\FF;A,\xi_0,h_1,\dots, h_N)\to \cont_\ot(\FF;A,\xi_0,h_1,\dots, h_N)$$ is a (weak) homotopy  equivalence. 
 \end{theorem}
  \begin{remark}\label{rem:foliations}{\rm 
   If $V$ is closed then an analog of Gray-Moser's theorem still holds even though the leaves could be non-compact. Indeed, the leafwise vector field produced by Moser's argument  is integrable because $V$ is compact, and hence it generates the flow realizing the prescribed deformation of the leafwise contact structure.  Therefore, a homotopical classification of leafwise contact structures coincides with their isotopical classification. 
}
\end{remark}
  
{\bf Plan of the paper.}  
 Because of Gromov's $h$-principle for contact structures on open manifolds, the entire problem can be reduced to a local extension problem of when a germ of a contact structure on the $2n$-sphere $\p B^{2n+1}$
can be extended to a contact structure on $B^{2n+1}$.
Our proof is based on the two main results: Proposition  \ref{prop:unique-model}, which reduces the extension problem to a unique model in every dimension, and Proposition \ref{p:ConnectSumDisc}, which provides an extension of the connected sum of this universal  model with a neighborhood of an overtwisted  $2n$-disc  $D_\ot$ defined in Section \ref{sec:ot-disc}. 
We formulate Propositions  \ref{prop:unique-model} and \ref{p:ConnectSumDisc}   in Section \ref{sec:main-existence}, and then deduce from them Theorem \ref{thm:main-existence}. We then continue   Section \ref{sec:main-existence} 
with Propositions \ref{prop:unique-model-param} and  \ref{p:ConnectSumDisc-param}   which are parametric analogs of the preceding propositions, and then prove Theorem \ref{thm:main} and Corollary \ref{cor:isocontact}. The proofs of Theorems \ref{thm:existence-fol} and
\ref{thm:classif-fol}, concerning leafwise contact structures on a foliation, are postponed till  Section \ref{sec:foliations-proofs}.

In Section \ref{sec:conjugation} we   study  the notion of domination between contact shells and prove Proposition \ref{prop:disorder} and its corollary  Proposition~\ref{prop:shallow}, which can be thought of as certain disorderability results for  the group of contactomorphisms of  a contact ball.   These results are used in an essential way   in the proof of Propositions \ref{prop:unique-model} and \ref{prop:unique-model-param} in Section \ref{sec:universal}.
We prove the main extension results,  Propositions \ref{p:ConnectSumDisc} and \ref{p:ConnectSumDisc-param},
 in Section~\ref{sec:tripling}.      
  
   Propositions \ref{prop:unique-model} and \ref{prop:unique-model-param} are proved in Section~\ref{sec:universal}. 
 This is done by gradually standardizing the extension problem in Sections~\ref{sec:holes}--\ref{sec:RedSaucers}. 
 First, in Section \ref{sec:holes} we reduce it to extension of   germs of contact structures induced by a certain family of immersions of $S^{2n}$ into the standard contact $\R^{2n+1}$. This part is fairly  standard, and  the proof uses the traditional  $h$-principle type techniques going back to Gromov's papers \cite{Gro69, Gro72}  and Eliashberg-Mishachev's  paper  \cite{EliMi97}.  
 In Section~\ref{sec:further} we show how the extension problem of Section~\ref{sec:holes} can be reduced to the extension of some special models determined by contact Hamiltonians.   Finally,  to complete the proof of 
 Propositions \ref{prop:unique-model} and \ref{prop:unique-model-param} we introduce in
 Section \ref{sec:universal}  equivariant coverings  and use them   to further reduce  the problem to just one   universal  extension model in any given dimension. 

The final Section \ref{sec:discussion} is devoted to further comments regarding the overtwisting property. We also provide an explicit classification of overtwisted contact structures on spheres.
 
\begin{displaymath}
\xymatrix{
\text{Thm } \ref{thm:main-existence}  && \text{Prop } \ref{prop:neighb-ot} \ar[ll] \ar[r] & \text{Thm }\ref{thm:main-existence-T} \ar[r] & \text{Thm } \ref{thm:main} \\ 
&\text{Prop } \ref{p:ConnectSumDisc} \ar[ul] & \text{Lemma } \ref{lem:Scale} \ar[u] & \text{Prop } \ref{p:ConnectSumDisc-param}\ar[u] & \text{Lemma } \ref{lm:choosing-paths} \ar[ul] \\
\text{Prop } \ref{prop:unique-model} \ar[uu] && \text{Prop } \ref{p:ContEmb} \ar[ul] \ar[r] & \text{Prop } \ref{p:LocalConnectSumDisc-param} \ar[u] & \text{Prop } \ref{prop:unique-model-param} \ar[uul]\\
&& \S \ref{sec:equi-cov} \ar[ull] \ar[r] & \S \ref{sec:standard-param} \ar[ur] \\
\text{Prop } \ref{prop:to saucers} \ar[uu] & \text{Prop } \ref{prop:saucer-to-circle} \ar@2{->}[uul] & \S \ref{ssec:disorder} \ar[u] \ar[ur] & \text{Prop } \ref{prop:saucer-to-circle-param} \ar[u] \ar[uur] & \text{Prop } \ref{prop:to-saucers-param} \ar[uu] \\
\text{Lemma }\ref{lm:annuli-to-saucers} \ar[u] && \text{Thm } \ref{thm:Gromov-open} \ar[ull] \ar[urr]
}
\end{displaymath}

 
The above diagram outlines the logical dependency of the major propositions in the paper. Notice that the left three columns together give the proof of Theorem~\ref{thm:main-existence}, whereas the right three columns together prove Theorem~\ref{thm:main}.  The double arrow between Propositions~\ref{prop:saucer-to-circle} and \ref{prop:unique-model} indicates that \ref{prop:saucer-to-circle} is used in the proof of \ref{prop:unique-model}  twice in an essential way. The diagram is symmetrical about the central column, in the sense that any two propositions which are opposite of each other are parametric/non-parametric versions of the same result.

 
 {\bf Acknowledgements.}
 After the first version of this paper was posted on arXiv many mathematicians send us their comments and corrections. We are very grateful to all of them, and  especially  to Kai Cieliebak, Urs Fuchs and Janko Latschev.
 
\section{Basic notions}\label{sec:basic}
 \subsection{Notation and conventions}\label{sec:notation}
Throughout the paper, we will often refer to discs of dimension $2n-1$, $2n$, and $2n+1$. For the sake of clarity, we will always use the convention $\dim B=2n+1$, $\dim D = 2n$, and $\dim \Delta = 2n-1$. When we occasionally refer to discs of other dimensions we will explicitly write their dimension as a superscript, e.\,g. $D^{m}$. All discs will be assumed      diffeomorphic to closed balls, \emph{with possibly piecewise smooth boundary}.  

Functions, contact structures, etc, on a subset $A$ of a manifold $M$ will always be assumed given on a neighborhood $\Op\,A\subset M$.
Throughout the paper, the notation $I$ stands for the  interval $I= [0,1]$ and $S^1$ for   the circle $S^1 = \R/\Z$. The notation $A \Subset B$ stands for compact inclusion, meaning $\ol A \subset \Int B$.

As the standard model  of contact structures in $\R^{2n-1}=\R\times(\R^2)^{n-1}$ we choose
$$\xi_\st:=\left\{\lambda_\st^{2n-1}:=dz+\sum_{i=1}^{n-1} u_id\varphi_i=0\right\}$$ where
$(r_i,\varphi_i)$ are polar coordinates in copies of $\R^2$ with $\varphi_i \in S^1$ and $u_i:=r_i^2$. 
We always use the contact form $\lambda_\st^{2n-1}$ throughout the paper. 
On $\R^{2n+1}$ we will use two equivalent  contact structures, both defined by 
$$
	\xi_\st:= \{\lambda_\st^{2n-1} + vdt = 0\}
$$  
where the coordinates $(v,t)$ have two possible meanings.  For $\R^{2n-1} \times \R^2$ we will take
$v := r^2$ and $t \in S^1$ where $(r,t)$ are polar coordinates on $\R^2$, while for
$\R^{2n-1} \times T^*\R$ we will take $v := -y_n$ and $t:= x_n$.
In each case it will be explicitly clarified which model contact structure is  considered.

A compact domain in $(\R^{2n-1}, \xi_\std)$ will be called {\em star-shaped} if its boundary is transverse to the contact vector field $Z
 = z\frac{\p}{\p z}+\sum_{i=1}^{n-1}u_i\frac{\p}{\p u_i}$. An abstract contact $(2n-1)$-dimensional closed ball will be called star-shaped if it is contactomorphic to a star-shaped domain in $(\R^{2n-1}, \xi_\st)$. 

 A hypersurface $\Sigma \subset (M, \xi = \ker \lambda)$ in a contact manifold
  has a singular $1$-dimensional
  {\em characteristic distribution} $\ell\subset T\Sigma\cap\xi$, defined to be the kernel of the $2$-form 
  $d\lambda|_{T\Sigma\cap\xi}$, with singularities where $\xi = T\Sigma$.
  The distribution $\ell$ integrates to a singular characteristic foliation $\FF$ with a transverse contact structure, that is contact structures on hypersufaces $Y \subset \Sigma$ transverse to 
   $\FF$,
   which is invariant with respect to monodromy along the leaves of $\FF$.
  The characteristic foliation $\FF$ and its transverse contact structure determines the germ of $\xi$ along $\Sigma$ up to a diffeomorphism fixed on $\Sigma$.

\subsection{Shells}\label{sec:shells}
We will need below some specific models for germs of contact structures along the boundary sphere  of a
$(2n+1)$-dimensional ball $B$ with {\em piecewise smooth} (i.e.\ stratified by smooth submanifolds) boundary, extended to $B$ as   {\em
 almost contact} structures.\footnote{We always view these balls as domains in a larger manifolds so the germs of contact structures along $\p B$  are assumed to be slightly  extended outside of $B$.} 
   
 A \emph{contact shell} will be an almost contact structure $\xi$ on a ball $B$ such that $\xi$ is genuine
near $\p B$.   A contact shell  $(B,\xi)$ is called {\em solid} if $\xi$ is  a genuine contact structure. 
An {\em equivalence} between two contact shells  $(B,\xi)$ and $(B',\xi')$ is a diffeomorphism $g:B\to B'$ such that $g_*\xi$ coincides with $\xi'$ on $\Op\p B'$ and $g_*\xi$ is homotopic   to $\xi'$ through almost contact structures fixed on $\Op\p B'$.  

 \begin{figure}\begin{center}
\includegraphics[scale=.5]{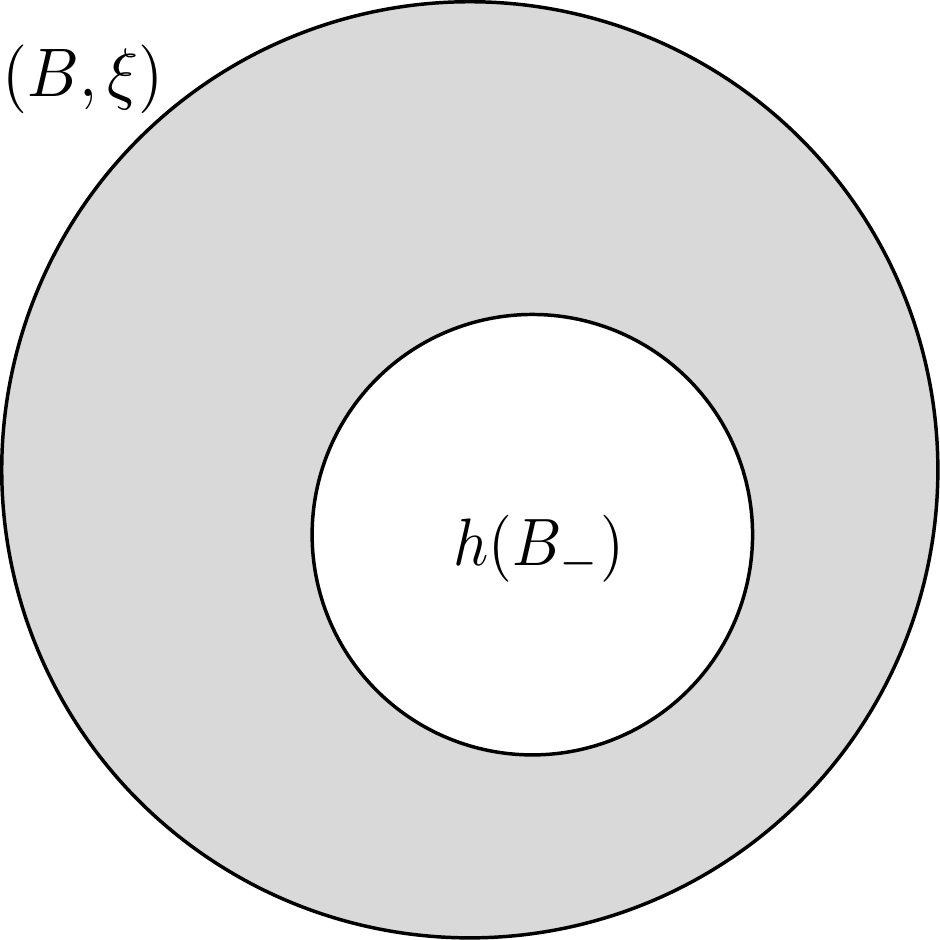}
\caption{Domination of contact shells, where $\xi$ is genuine in a neighborhood of the gray region and $\xi|_{h(B_-)} \cong \xi_-$
as almost contact structures.}\label{fig: domination}
\end{center}\end{figure}
  
Given two shells $\zeta_+=(B_+,\xi_+)$ and $\zeta_-=(B_-,\xi_-)$ we say that 
$\zeta_+$ {\em dominates} $\zeta_-$ if there exists both
\begin{itemize}
\item a shell $\wt\zeta=\left( B, \xi\right)$ with an equivalence $g:(B,\xi)\to (B_+,\xi_+)$ of contact shells,
\item an embedding $h:  B_- \to  B$ such that $h^*\xi=\xi_-$ and
   $ \xi$ is a genuine  contact structure on $ B \setminus \Int h(B_-).$
  \end{itemize}
   We will refer to the composition 
   $g\circ h:(B_-,\xi_-)\to (B_+,\xi_+)$ as a {\em subordination} map.
Notice that, if $ (B_+,\xi_+)$ dominates $(B_-,\xi_-)$ and $(B_-,\xi_-)$ is solid, then $ (B_+,\xi_+)$ is equivalent to a solid shell. If  both  shells $(B_-,\xi_-)$ and $ (B_+,\xi_+)$ are solid, then the subordination map is called {\em solid} if it is a contact embedding.

  
     A \emph{gluing place} on a contact shell $(B, \xi)$ is a smooth point $p \in \p B$ where  $T_{p}\p B = \xi|_{p}$.
     Given two gluing places $p_i \in (B_i, \xi_i)$ on contact shells,
     the standard topological boundary connected sum construction can be performed in a  straightforward way
     at the points $p_i$ to produce a contact shell $(B_0\#B_1,\xi_0\#\xi_1)$, which we will call the {\em boundary connected sum}
of the shells  $(B_i,\xi_i)$ at the boundary  points $p_i$.
%
%
We refer the reader to Section \ref{sec:BCS} for precise definitions, and only say here that we can make the shells 
$(B_i,\xi_i)$ isomorphic near $p_i$ via an orientation reversing diffeomorphism by a $C^1$-perturbation of the shells that fix the contact planes $\xi_i|_{p_i}$.

\subsection{Circular model shells}\label{sec:ContactHamShell}

Here we will describe a contact shell model associated to contact Hamiltonians, which will play a key role in this paper
for it is these models that we will  use to define overtwisted discs.   .

Let $\Delta \subset \R^{2n-1}$ be a compact star-shaped domain and consider a smooth function
\begin{equation}\label{e:KBoundary}
	K: \Delta \times S^1 \to \R \quad\mbox{with}\quad K|_{\partial \Delta \times S^1} > 0.
\end{equation}
Throughout the paper we will use the notation $(K, \Delta)$ to refer to a such a contact Hamiltonian 
on a star-shaped domain.
For a constant $C \in \R$, we can define a piecewise smooth $(2n+1)$-dimensional ball associated to $(K, \Delta)$ by
\begin{equation}\label{e:S1ball}
	B_{K,C} := \{(x,v,t) \in \Delta \times \R^2 : v \leq K(x,t)+C\} \subset \R^{2n-1} \times \R^2
\end{equation}
provided $C + \min(K) > 0$.
\begin{figure}[h]
   \centering
   \def\svgwidth{400pt}
   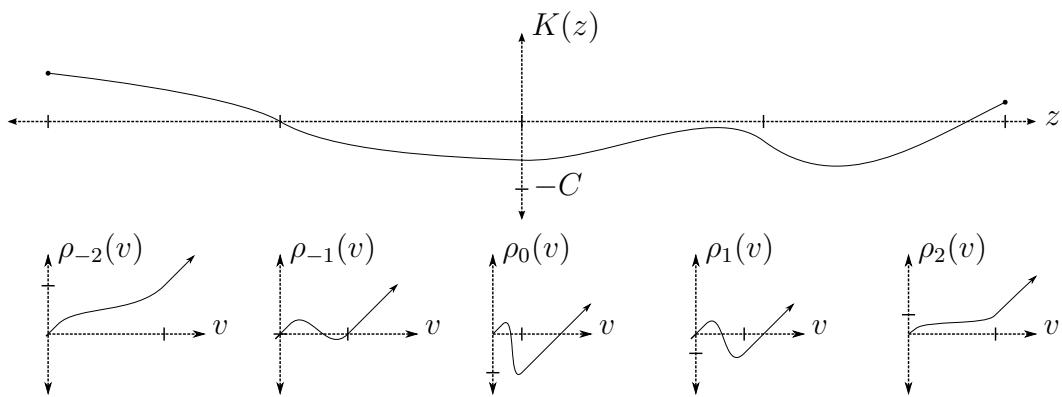
   \caption{A family of functions $\rho_z$ for the Hamiltonian $K: [-2,2] \to \R$.
   The hash mark on the vertical axis is at $\rho_z = K(z)$ and the hash mark on the horizontal axis is at $v = K(z) + C$.}
   \label{f:Rho}
  \label{f:embedHamiltonian2a}
\end{figure}
Now consider a smooth family of functions
\begin{equation}\label{e:FamilyRhoS1}
	\rho_{(x,t)}: \R_{\geq 0} \to \R \quad\mbox{for}\quad (x,t) \in \Delta \times S^1 \quad\mbox{such that}
\end{equation}
\begin{enumerate}
\item $\rho_{(x,t)}(0) = 0$ for all $(x,t) \in \Delta \times S^1$\,,
\item $\rho_{(x,t)}(v) = v - C$ when $(x,v,t) \in \Op\{v = K(x,t) +C\}$\,, and
\item $\p_v \rho_{(x,t)}(v) > 0$ when $(x,v,t) \in \Op\{v \leq K(x,t) + C, x \in \p\Delta\}$\,.
\end{enumerate}
See Figure~\ref{f:Rho} for a schematic picture of such a family of functions.
Such a family, which exists by \eqref{e:KBoundary}, defines an almost contact structure on $B_{K, C}$ given by
\begin{equation}\label{S1ContactForm}
	\eta_{K, \rho} := (\alpha_\rho, \omega) \quad\mbox{with}\quad \alpha_\rho := \lambda_{\st} + \rho\,dt
	\mbox{ and } \omega := d\lambda_{\st} + dv\wedge dt
\end{equation}
where smoothness and condition (i) ensures $\alpha_\rho$ is a well defined $1$-form.
\begin{figure}
   \centering
   \def\svgwidth{150pt}
   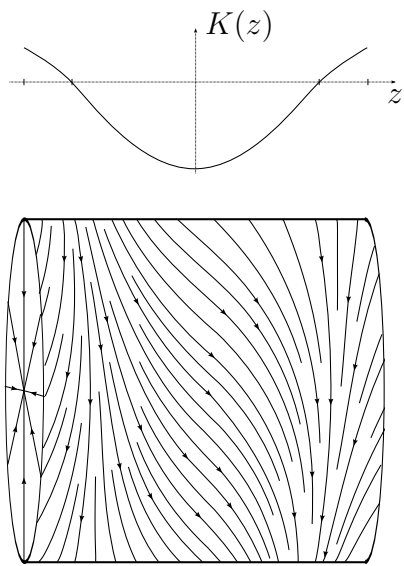
   \caption{The characteristic foliation on the piecewise smooth sphere $\partial B_K$.}
  \label{f:poly}
\end{figure} 
\begin{lemma}\label{l:CircleShell}
	The pair $(B_{K,C}, \eta_{K,\rho})$ is a contact shell, which up to 
	equivalence is independent of the choice of $\rho$ and $C$.
\end{lemma}
\begin{proof}
	One checks $\alpha_\rho$ satisfies the contact condition
	\begin{equation}\label{e:CC}
	\mbox{$\alpha_\rho \wedge (d\alpha_\rho)^n > 0$ whenever $\p_v \rho > 0$}
	\end{equation}
	and hence $(B_{K,C}, \eta_{K,\rho})$ is a contact shell by conditions (ii) and (iii).

	Consider the special case of two choices $\rho_0$ and $\rho_1$
	as in \eqref{e:FamilyRhoS1} for the same $C$.  We can pick a family of diffeomorphisms
	$\phi_{(x,t)}: \R_{\geq 0} \to \R_{\geq 0}$ such that
	$$
		\phi_{(x,t)}(v) = (\rho_1^{-1}\circ\rho_0)_{(x,t)}(v) \mbox{ on }
		\Op \{v = K(x,t) + C\} \cup \Op \partial\Delta
	$$
	and this family induces a diffeomorphism $\Phi: B_{K,C} \to B_{K,C}$ such that
	$\Phi^*\alpha_{\rho_1} = \alpha_{\rho_1 \circ \phi}$ and $\alpha_{\rho_1 \circ \phi} = \alpha_{\rho_0}$
	on $\Op\partial B_{K,C}$.  It follows that $\Phi$ is an equivalence between $\eta_{K, \rho_0}$ and $\eta_{K, \rho_1}$
	since we have the straight line homotopy from $\alpha_{\rho_1 \circ \phi}$ to $\alpha_{\rho_0}$.
	
	
	Given two choices $(C_0, \rho_0)$ and $(C_1, \rho_1)$, we can 
	pick a  family of diffeomorphisms  $\psi_{(x,t)}: \R_{\geq 0} \to \R_{\geq 0}$
	such that
	$$
	\psi_{(x,t)}(v) = 
	 v + (C_1 - C_0) 
	 \quad \mbox{ on $\Op \{v = C_0+K(x,t)\}$}.
	 $$  
	 Now to see $(B_{K, C_0}, \eta_{K, \rho_0})$ and $(B_{K, C_1}, \eta_{K, \rho_1})$
	 are equivalent, just note that $(B_{K, C_0}, \eta_{K, \rho_0})$ and $(B_{K, C_0}, \eta_{K, \rho_1 \circ \psi})$ are 
	 equivalent via the previous special case and $\psi$ induces
	 an isomorphism $\Psi: (B_{K, C_0}, \alpha_{\rho_1\circ\psi}) \to (B_{K, C_1}, \alpha_{\rho_1})$.
\end{proof}
	
We will use the notation $(B_{K, C}, \eta_{K, \rho})$ throughout the paper for this specific construction, though we will usually  drop
$C$ and $\rho$ from the notation and write $(B_K, \eta_{K})$ when the particular choice will be irrelevant.  We will refer to this contact shell as the \emph{circle model associated to $(\Delta, K)$}.

\begin{remark} {\rm This construction can produce a genuine contact structure on $B_K$ if and only if $K > 0$.
If $K > 0$, then from \eqref{e:CC} it follows $(B_{K, 0}, \eta_{K, \rho})$ for $\rho(v) = v$ is a contact ball.
Conversely if $K(x_0,t_0) \leq 0$ for some $(x_0, t_0)$, then $(B_{K,C}, \eta_{K, \rho})$ is
never contact by \eqref{e:CC} since 
	since conditions (i) and (ii) on $\rho$ force
	$\p_v\rho_{(x_0,t_0)}(v_0) = 0$ for some $v_0$.}
\end{remark}


The contact germ $(\p B_K, \eta_K)$ without its almost contact extension can be described more directly in the following way.
Consider the contact germs on the hypersurfaces
\begin{align*}
	\tilde\Sigma_{1,K} &= \{(x,v,t): v = K(x,t)\} \subset (\Delta \times T^*S^1, \ker(\lambda_{\st} + v\,dt)) \quad\mbox{and}\\
	\tilde\Sigma_{2,K} &= \{(x,v,t): 0 \leq v \leq K(x, t)\,,\,\, x \in \p \Delta\} \subset (\Delta \times \R^2, \ker(\lambda_{\st} + v\,dt)).
\end{align*}
These germs can be glued together via the natural identification between neighborhoods of their boundaries,
to form a contact germ $\tilde{\eta}_K$ on $\tilde{\Sigma}_K := \tilde\Sigma_{1,K} \cup \tilde\Sigma_{2,K}$.
\begin{lemma}\label{l:EasyCircleGerm}
	The contact germs $(\p B_K, \eta_K)$ and $(\tilde{\Sigma}_K, \tilde{\eta}_K)$ are contactomorphic.
\end{lemma}
\begin{proof}
	We have that the boundary $\p B_{K, C} = \Sigma_{1,K, C} \cup \Sigma_{2,K, C}$ where
\begin{align*}
		\Sigma_{1,K,C} &:= \{(x,v,t) \in \Delta \times \R^2: v = K(x,t)+C\} \quad\mbox{and} \\
	         \Sigma_{2,K,C} &:= \{(x,v,t) \in \Delta \times \R^2: 0 \leq v \leq K(x, t)+C\,,\,\, x \in \p \Delta\}\,.
\end{align*}
	Recalling the $1$-form $\alpha_\rho = \lambda_{\st} + \rho\,dt$ is a contact form 
	near $\p B_{K,C} \subset \Delta \times \R^2$, just note that
	$\rho$ induces contactomorphisms of neighborhoods
	$$
		(\Op \Sigma_{j,K,C}\,,\,\, \ker \alpha_\rho) \to (\Op \tilde\Sigma_{j,K}\,,\,\, \ker(\lambda_{\st} + v\,dt))
	$$
	for $j=0,1$ by construction.
\end{proof}

\subsection{The cylindrical domain}

Throughout the paper we will often use the following star-shaped cylindrical domain
$$
	 \Delta_\cyl := D^{2n-2}\times[-1,1] = \{\abs{z} \leq 1\,,\,\, u \leq 1\} \subset (\R^{2n-1}, \xi_{\st}).
$$
where $D^{2n-2}:= \{u=\sum_i^{n-1}u_i\leq 1\}\subset\R^{2n-2}$ is the unit ball.  

Also observe for any contact Hamiltonian $(K, \Delta_{\cyl})$ the \emph{north pole} and \emph{south pole}
$$
	P_{\pm 1} := (0, \pm 1, 0) \in (\p B_K, \eta_K)
$$
in the coordinates $(u, z, v) \in \R^{2n-1} \times \R^{2}$ are gluing places in the sense of Section~\ref{sec:shells}.
When performing a boundary connected sum of such models $(B_{K}\# B_{K'}, \eta_{K}\#\eta_{K'})$ we will always use the north
pole of $B_K$ and the south pole of $B_{K'}$.  See Section \ref{sec:BCS} for more details on the gluing construction.

 \section{Proof of Theorems
  \ref{thm:main-existence} and \ref{thm:main}\label{sec:main-existence}}
\subsection{Construction of contact structures with universal holes}\label{sec:univ-holes}

Proposition \ref{prop:unique-model}, which we prove in  Section \ref{sec:equi-cov}, and which represents one half
of the proof of Theorem \ref{thm:main-existence}, constructs from an almost contact structure a contact structure in the complement  of a finite number of disjoint  $(2n+1)$-balls where the germ of the contact structure on the boundaries of the balls has a 
 {\em unique universal} form:
 \begin{proposition}\label{prop:unique-model}
 For fixed dimension $2n+1$ there exists a contact Hamiltonian $(K_{\univ}, \Delta_{\cyl})$,
 specified in Lemma \ref{lm:universal-eps}, such that the following holds. 
 For any almost contact manifold $(M, \xi)$ as in Theorem \ref{thm:main-existence}
  there exists an almost contact structure $\xi'$ on $M$, which is homotopic   to $\xi$ relative $A$ through almost contact structures,
and a finite collection of
 disjoint balls $B_i\subset M\setminus A$ for $i=1,\dots, L$, with piecewise smooth boundary such that
 \begin{itemize}
 \item $\xi'$ is a genuine contact structure on $M\setminus \bigcup_1^L\Int B_i$
 \item the contact shells $\xi'|_{B_i}$ are equivalent to
 $(B_{K_{\univ}},\eta_{K_\univ})$
 for $i=1, \dots, L$.
 \end{itemize}
 \end{proposition}

 \begin{remark} {\rm  If $(B_K, \eta_K)$ is dominated by $(B_{K_{\univ}},\eta_{K_\univ})$, then we can $K$
 in the statement of Proposition~\ref{prop:unique-model} in lieu of $K_\univ$.
  In particular by Lemma~\ref{lm:3-domin}, in {\em the $3$-dimensional case}
  we can take $K_\univ: [-1, 1] \to \R$ to be {\em any somewhere negative function.}
  Our proof in higher dimension is not constructive, and   we do not know an effective criterion which would allow one to verify whether a particular function $K_\univ$  satisfies  \ref{prop:unique-model}. Of course, it is easy to construct a $1$-parameter family of Hamiltonians $K^\eps$ so that any Hamiltonian $K$ is less than $K^\eps$ for sufficiently small $\eps>0$ (see Example \ref{ex:special}). We can then take $K_\univ = K^\eps$ for sufficiently small $\eps$. It would be interesting to find such a general criterion for which Hamiltonians can be taken as $K_\univ$.}
\end{remark}

\subsection{Overtwisted discs and filling of universal holes}\label{sec:ot-disc}
Proposition \ref{p:ConnectSumDisc}, which we formulate in this section and prove in Section \ref{sec:filling},
will combine with Proposition \ref{prop:unique-model} to prove Theorem \ref{thm:main-existence} in Section \ref{sec:proof-existence}.
 
A smooth function $k: \R_{\geq 0} \to \R$  is called  \emph{special} if $k(1)>0$ and 
\begin{equation}\label{e:subscale}
	a\,k\left(\tfrac{u}{a}\right) < k(u) \quad\mbox{for all $a > 1$ and $u \geq 0$.}
\end{equation}
This implies that   $k(0)<0$, and hence $k(u)$ has a zero in $(0, 1)$.
By differentiating \eqref{e:subscale} with respect to $a$ we conclude that
\begin{equation}\label{e:specialCon}
	\quad k(u) - u\,k'(u) < 0 \quad\mbox{for all $u \geq 0$} 
\end{equation}
which means that  the $y$-intercept of all  tangent lines  to the graph of $k$ is negative.

We call  a function $K:\Delta_\cyl\to\R$ {\em spherically symmetric} if it depends only on the coordinates $(u, z)$ where
$u=\sum_{i=1}^{n-1}u_i$, in this case by a slight abuse of notation  will write $K(u,z)$, rather than   $\wt K(u,z)$ for some function 
$\wt K:[0,1]\times[-1,1]\to \R$.

\begin{definition}\label{def:SpecialHam}
 
A spherically symmetric piecewise smooth contact Hamiltonian $K: \Delta_\cyl \to \R$
is  called {\em special} if  
\begin{description}
\item{(i)}  $K > 0$ on $\p\Delta_{\cyl}$;
\item{(ii)} $K(u,-\delta)=K(u,\delta)$ for $\delta \in \Op \{1\}$ and $u \in [0,1]$;
\item{(iii)} If $\delta \in \Op\{1\}$, then $K(u,z)\leq K(u,\delta)$ for  $u\in[0,1]$ and $\abs{z} \leq \delta$;
\end{description} and
   there exist $z_D \in (-1, 1)$ and a  special function $k: \R_{\geq 0} \to \R$ such that
\begin{description}
 \item{(iv)} $K$ is non-increasing in the coordinate $z$ for $z\in\Op [-1,\,z_D]\subset[-1,1]$;
\item{(v)} $K(u,z) \geq k (u)$ {\em with equality} if $z \in \Op\{ z_D\}$. 
\end{description}
 
When $n=1$, where $\Delta_\cyl = [-1,1]$, condition (v) can be replaced with $K(z_D) < 0$.
\end{definition}
 
As the following example shows special contact Hamiltonians exist and furthermore for any particular contact Hamiltonian $(K',\Delta_\cyl)$  positive on $\p\Delta_\cyl\times S^1$ there is a special contact Hamiltonian $K: \Delta_\cyl \to \R$ such that $K < K'$.
\begin{ex}\label{ex:special}{\rm
	For positive constants $a, b, \lambda$ with $b < 1$ and $\lambda > \tfrac{a}{1-b}$, define the 
	special piecewise smooth function
	$$
		k(u) = \begin{cases}
		\lambda(u-b) - a &   \mbox{if $u \geq b$}\\
		-a & \mbox{if $u \leq b$}
		\end{cases}
	$$
	and the special piecewise smooth contact Hamiltonian
	$$
		K(u,z) = \max\{k(u),\,k(\abs{z})\}.
	$$
	By a perturbation of $K$ near its singular set, we may construct a smooth
	special contact Hamiltonian $\tilde{K}$ that is $C^0$-close to $K$, though smoothness of $K$ will not be needed in the proof.}
\end{ex}

Let $K:\Delta_\cyl\to\R$ be a special contact Hamiltonian
and define $(D_{K}, \eta_{K})$ to be the contact germ on the 
$2n$-dimensional disc
\begin{equation}\label{eq:ot-disc}
	D_{K} := \{(x, v, t) \in \p B_{K}: z(x) \in   [-1, z_D]\} \subset (B_{K}, \eta_{K})
\end{equation}
where $z_D$ is the constant in Definition~\ref{def:SpecialHam}.
Notice that $D_{K}$ inherits the south pole of the corresponding circular
model and the coorientation of $\p B_{K}$ as a boundary. 

\begin{figure}
   \centering
   \def\svgwidth{50pt}
   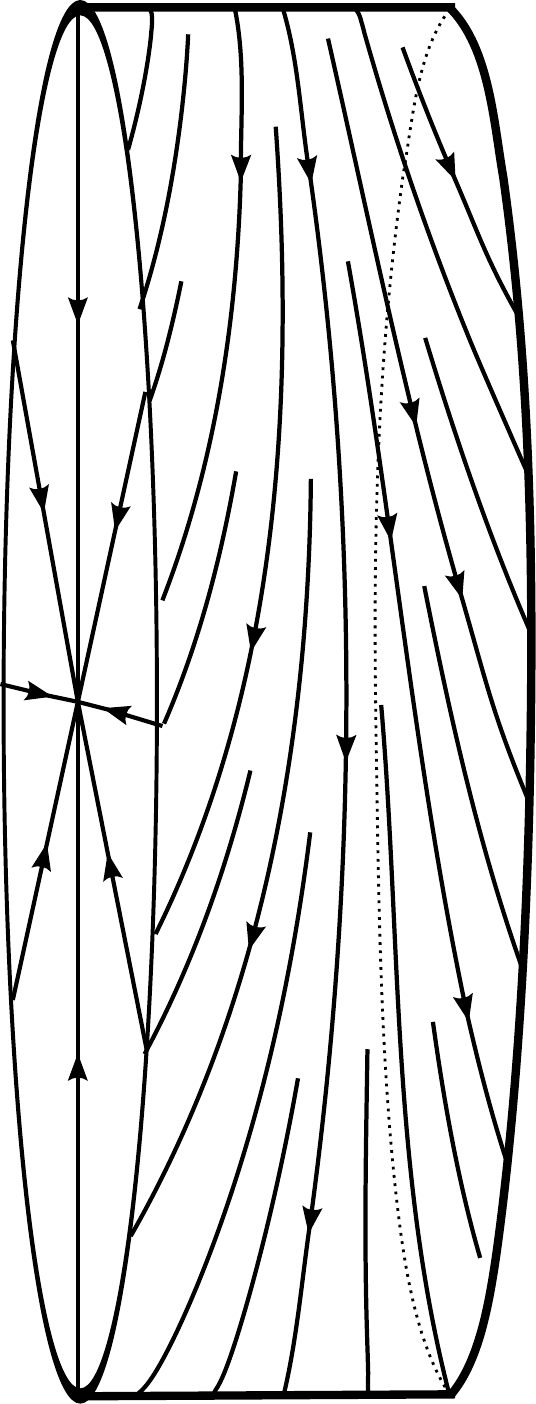
   \caption{A $2$-dimensional overtwisted disc $(D_{\ot}, \eta_{\ot})$ with its characteristic foliation.}
  \label{f:Disk}
\end{figure} 

\begin{definition}\label{def:OTD}
	Let $K_{\univ}$ be as in Proposition \ref{prop:unique-model}.
	An \emph{overtwisted disc} $(D_\ot,\eta_\ot)$ is a $2n$-dimensional disc with a germ of a contact structure
	such that there is a contactomorphism
	$$
		(D_\ot,\eta_\ot) \cong \left(D_{K}, \eta_{K}\right)
	$$
	where $K$ is some special contact Hamiltonian such that $K < K_\univ$.
	A contact manifold $(M^{2n+1}, \xi)$ is {\em overtwisted} if it admits a contact embedding 
	$(D_\ot,\eta_\ot)\to (M,\xi)$ of some overtwisted disc.
\end{definition}

\begin{ex}
{\rm In the $3$-dimensional case, it follows from Lemma~\ref{lm:3-domin} that the disc
$$
	D_{K} := \{(z, v, t) \in \p B_{K}: z \in   [-1, z_D]\} \subset (B_{K}, \eta_{K})
$$
is overtwisted in the sense of Definition~\ref{def:OTD} for {\em any} special contact Hamiltonian, i.e a somewhere negative function on the interval $[-1,1]$, positive  near the end-points $\pm 1$. }
\end{ex}
\begin{remark}\label{rem:various-ot-discs}{\rm
The definition of the overtwisted disc $(D_K,\eta_K)$ depends on the choice of a special Hamiltonian $K<K_\univ$, and the germs $\eta_K$  need not to be contactomorphic when we vary $K$. However, as Corollary~\ref{cor:isocontact}  shows,  for any two  special Hamiltonians  $K,K'<K_\univ$  any neighborhood of $(D_K,\eta_K)$ contains $(D_{K'},\eta_{K'})$.}
\end{remark}

As the following proposition shows any overtwisted contact manifold contains infinitely many disjoint overtwisted discs.

\begin{prop}\label{prop:neighb-ot}
Every neighborhood of an overtwisted disc in a contact manifold contains a foliation by overtwisted discs.
\end{prop}
We prove Proposition \ref{prop:neighb-ot}    at the end of Section~\ref{sec:effect}.

Given a special contact Hamiltonian $K: \Delta_\cyl \to \R$, 
the contact germ $(D_{K}, \eta_{K})$ has the following remarkable property, which we will
prove in Section \ref{sec:final-filling}.
Suppose one has a $(2n+1)$-dimensional contact ball $(B,\xi)$ with piecewise smooth boundary such that
$(D_{K}, \eta_{K}) \subset (\p B, \xi)$ where the co-orientation of $D_{K}$ coincides with the 
outward co-orientation of $\p B$.

\begin{prop}\label{p:ConnectSumDisc}
Let $K_0, K, \;K_0\geq K,$ be two contact Hamiltonians and $K$ is special. Then
the contact shell $(B_{K_0}\# B, \eta_{K_0} \# \xi)$
given by performing a boundary connected sum at the north pole of $B_{K_0}$ and the south pole of $D_{K} \subset \p B$
is equivalent to a genuine contact structure.
\end{prop}


\subsection{Proof of Theorem \ref{thm:main-existence}}\label{sec:proof-existence}
 
Choose a ball $B\subset M\setminus A$ with piecewise smooth boundary and deform the almost contact structure $\xi$ to make it a contact structure on $B$ with an overtwisted disc $(D_\ot, \eta_{\ot}) \subset (\p B, \xi)$ on its boundary.
This can be done since any two almost contact structures on the ball are homotopic if
we do not require the homotopy to be fixed on $\p B$.

 Using Proposition \ref{prop:unique-model} we deform the almost  contact  structure $\xi$  relative to $A\cup B$ to an almost contact structure $\xi$ on $M$, which is genuine in the complement of finitely many disjoint balls
$B_1,\dots, B_N\subset M\setminus (A\cup B)$ where each $(B_i, \xi|_{B_i})$ is isomorphic to $(B_{K_\univ}, \eta_{K_\univ})$
as almost contact structures.
    
    According to Proposition \ref{prop:neighb-ot} we can pick disjoint balls $B_i' \subset \Int(B)$, $i=1,\dots, N$, each with an overtwisted disc on their boundary $(D^i_{\ot}, \eta_{\ot}^i) \subset (\p B_i', \xi)$.
As we will describe in Section~\ref{sec:BCS} we can perform
    an ambient boundary connected sum $B_i \# B_i' \subset M \setminus A$ such that the sets $B_i \# B_i'$ are disjoint for $i=1, \dots, N$ and there are isomorphisms of almost contact structures
    $$
    	(B_i \# B_i', \xi|_{B_i \# B_i'}) \cong (B_i \# B_i', \xi|_{B_i} \# \xi|_{B_i'}) \cong (B_K \# B_i', \eta_K \# \xi|_{B_i'}).
    $$
Now for $i=1, \dots, L$ by definition we have $(D^i_{\ot}, \eta_{\ot}^i) = (D_{K_i}, \eta_{K_i})$ for
special contact Hamiltonians $K_i$ such that $K_i < K_\univ$.  Therefore we can
apply Proposition \ref{p:ConnectSumDisc} to homotope $\xi|_{B_i \# B_i'}$ relative to the boundary to a genuine contact structure on $B_i \# B_i'$ for each $i=1,\dots, N$.  The result will be a contact structure on $M$ that is homotopic relative to $A$ to the original almost contact structure.
 


  \subsection{Fibered structures}
  To prove the parametric version of the Theorem  \ref{thm:main-existence},
  we need to discuss  the  parametric form of the introduced above notions.  The parameter space, always denoted by $T$, will be assumed to be a compact manifold of dimension $q$, possibly with boundary, and 
  we will use the letter $\tau$ for points in $T$.

A family of (almost) contact    structures $\{\xi^\tau\}_{\tau\in T}$
 on a manifold $M$ can be  equivalently viewed as   a fiberwise, or as we also say {\em fibered} (almost) contact  structure ${}^T\!\xi$ on the total space of the trivial fibration ${}^TM:=T\times M\to T$, which coincides with $\xi_\tau$ on the fiber $M^\tau:=\tau\times M$,
   $\tau\in T$.  
  
 
A \emph{fibered contact shell} $({}^T M, {}^T\xi)$ is a fibered almost contact structure that is genuine on $\Op \p({}^T M)$,
by which we mean $(M^{\tau}, \xi^\tau)$ is genuine for all $\tau \in \Op \p T$ and $(\Op M^{\tau}, \xi^{\tau})$ is genuine for all 
$\tau \in T$.
 An {\em equivalence} between fibered contact shells 
  $$
  G: ({}^{T_1}B_1 ,{}^{T_1}\xi_1) \to ({}^{T_2}B_2 ,{}^{T_2}\xi_2)
  $$ 
  is a diffeomorphism covering a diffeomorphism $g:T_1\to T_2$ such that
  $G^*\left({}^{T_2}\xi_2\right)$ and ${}^{T_1}\xi_1$ are homotopic relative to $\Op \p({}^{T_1}B_1)$
   through fibered almost contact structures on ${}^{T_1}B_1$. 
  In particular this requires $G: (B_1^{\tau}, \xi_1^{\tau}) \to (B_2^{g(\tau)}, \xi_2^{g(\tau)})$ to be an equivalence of contact shells for all 
  $\tau \in T_1$
  and to be a contactomorphism when $\tau \in \Op \p T_1$.
  
    Given fibered contact shells 
       ${}^{T_\pm}\zeta_\pm=\left({}^{T_\pm}B_\pm,{}^{T_\pm}\xi_\pm\right)$ 
       we say ${}^{T_+}\zeta_+$ \emph{dominates} ${}^{T_-}\zeta_-$
       if there is a third fibered contact shell $\zeta=\left({}^T B,{}^T\xi\right)$ such that
\begin{itemize}
\item there is a fibered equivalence   $G:{}^{T}\zeta \to {}^{T_+} \zeta_+$ and
\item a fiberwise embedding $ H: {}^{T_-}B_-\to{}^{T}  B$ covering an embedding $h:  T_-\to T$
such that $H^*({}^{T}\zeta)= {}^{T_-}\zeta_{-}$ and ${}^T \xi$ is genuine on $^TB\setminus H(\Int ^{T_-}B_-)$.  
  \end{itemize}
               
    We will refer to the embedding $G\circ H:\left({}^{  T_-}   B_-,{}^{T_-}  \xi_-\right)\to   \left({}^{T_+}B_+,{}^{T_+}\xi_+\right)$ as a {\em subordination map}.
    
    Finally we note that the  boundary connected sum construction can be performed in the fibered set-up to define a fibered connected sum 
    $$
    ({}^{T}B_1 \# {}^{T}B_2,\, {}^{T}\xi_1\# {}^{T}\xi_2) \quad\mbox{with fibers}\quad
     (B_1^\tau \# B_2^\tau, \xi_1^{\tau} \# \xi_2^{\tau})
     $$
    provided that we are given a family of boundary points 
    $p_{1}^\tau \in\p B_1^\tau$ and $p_{2}^{\tau} \in \p B_2^\tau$ as in the non-parametric case.
     
  \subsection{Parametric contact structures with universal holes}\label{sec:univ-holes-param}
  Given a special contact Hamiltonian $K:\Delta_\cyl \to\R$ we define a function $E\to\R$ by the formula 
    $E(u, z) := K(u,1)$. By assumption, we have  $K \leq E$ on $\Delta_{\cyl}$.
We further define a family of contact Hamiltonians $K^{(s)}:\Delta_\cyl\to\R$ by
 \begin{equation}\label{e:ParametricHam}
	K^{(s)} :=sK+ (1-s)E \quad\mbox{for $s \in [0,1]$.}
 \end{equation}

   Given a disc $T := D^q \subset \R^q$ pick a bump function $\delta: T \to [0,1]$ with support in the interior of $T$
   and consider the   family of contact Hamiltonians $K^{(\delta(\tau))}: \Delta_{\cyl} \to \R$ parametrized by  $\tau \in T$ 
 and  the fibered circular model shell over $T$
   \begin{equation}\label{e:FiberHamShell}
   	\Big({}^T B_{K},\, {}^T\eta_{K}\Big) \quad\mbox{where}\quad  {}^T B_{K} = \bigcup_{\tau \in T} \{\tau\} \times B_{K^{\delta(\tau)}}
   \end{equation}
 and  the fiber over $\tau \in T$ is given by $(B_{K^{(\delta(\tau))}}, \eta_{K^{(\delta(\tau))}})$.

Recall Proposition \ref{prop:unique-model} and its contact Hamiltonian $K_{\univ}: \Delta_{\cyl} \to \R$.   
The next proposition, which we prove in Section \ref{sec:standard-param},
is the parametric generalization of Proposition \ref{prop:unique-model} and says that any fibered almost contact structure is equivalent to a fibered almost contact structure that is genuine away from holes modeled on $({}^T B_{K_{\univ}},\, {}^T\eta_{K_\univ})$.

 \begin{proposition}\label{prop:unique-model-param}
Let $T = D^q$ and let $A \subset M$ be a closed subset.
Every fibered almost contact structure ${}^T\xi_0$ on ${}^TM = T \times M$
that is genuine on $(T\times\Op A)\cup (\Op\p T \times M)$ is homotopic relative to 
$(T\times A) \cup (\p T \times M)$ through fibered almost contact structures on ${}^TM$ 
to some structure ${}^T\xi$ with the following property:

There is a collection of disjoint  embedded fibered shells ${}^{T_i}B_i \subset {}^TM$ 
 over (not necessarily disjoint) $q$-dimensional discs $T_i \subset T$ for $i = 1, \dots, L$ such that 
\begin{enumerate}
	\item 
	the fibers of ${}^T\xi$ are genuine contact structures away from $\bigcup_{i=1}^L \Int({}^{T_i} B_i)$ and
	\item the fibered contact shells
	$({}^{T_i}B_i, {}^{T_i}\xi)$ and $({}^{T_i} B_{K_{\univ}}, {}^{T_i}\eta_{K_{\univ}})$
	are equivalent.
\end{enumerate}
Furthermore for every $C \subset \{1, \dots L\}$ the intersection $\bigcap_{i \in C} T_i$ either empty or a disc. 
\end{proposition}

Recall the setting of Proposition~\ref{p:ConnectSumDisc}, where 
$(B,\xi)$ is a $(2n+1)$-dimensional contact ball for which there is a 
special contact Hamiltonian $K: \Delta_\cyl \to \R$ such that
$(D_{K}, \eta_{K}) \subset (\p B, \xi)$ where the co-orientation of $D_{K}$ coincides with the 
outward co-orientation of $\p B$. 
The following proposition, which we prove 
in Section \ref{sec:final-filling}, is the parametric generalization of Proposition~\ref{p:ConnectSumDisc}
where $({}^T B,{}^T\xi)$  is the fibered contact structure $T\times(B,\xi)$.

\begin{prop}\label{p:ConnectSumDisc-param}
	Let $(K_0, \Delta_{\cyl})$ be a contact Hamiltonian and consider the fibered contact shell 
	$$({}^TB_{K_0} \# {}^TB, {}^T\eta_{K_0} \# {}^T\xi)$$
	given by performing a boundary connected sum on each fiber over $\tau \in T$ at the north pole of 
	$B_{K^{(\delta(\tau))}_0}$ and the south pole of $D_{K} \subset \p B$.  If $K \leq K_0$ is special, then
	$({}^TB_{K_0} \# {}^TB, {}^T\eta_{K_0} \# {}^T\xi)$ is fibered equivalent to a  genuine fibered contact structure.
\end{prop}
	 
\subsection{Proof of Theorem \ref{thm:main} and Corollary \ref{cor:isocontact}}\label{sec:proof-main}
Theorem \ref{thm:main} is an immediate corollary of the following theorem, which is a fibered version of Theorem~\ref{thm:main-existence}.  In particular for each $q \geq 0$, we see that
$$
 j_*: \pi_q\big(\Cont_\ot(M;A,\xi_0,\phi)\big) \to \pi_q\big(\cont_\ot(M;A,\xi_0,\phi)\big)
$$
is an isomorphism by applying Theorem~\ref{thm:main-existence-T} in the cases of 
$D^q$ and $D^{q+1}$.

  \begin{theorem}\label{thm:main-existence-T} 
  Let $T = D^q$ and $A \subset M$ be a closed subset such that $M \setminus A$ is connected, and let 
  ${}^T\xi$ be a fibered almost contact structure on ${}^T M$ which is genuine on 
  $(T\times \Op A) \cup (\p T\times M)$.
  If there exists a fixed overtwisted disc $(D_\ot, \eta_{\ot}) \subset M\setminus A$ such that for all
  $\tau \in T$ the inclusion $(D_\ot, \eta_{\ot}) \subset (M \setminus A, \xi^\tau)$ is a contact embedding,
  then ${}^T\xi$ is homotopic   to a fibered genuine contact structure   through fibered almost contact structures
fixed on $\big(T\times (A\cup D_\ot)\big)\cup\big(\p T\times M\big)$
 \end{theorem}

\begin{proof}[Proof of Theorem \ref{thm:main-existence-T}]
By assumption there is a piecewise smooth disc $D_\ot\subset M\setminus A$ such that all almost  contact structures 
$\xi^\tau$ for $\tau\in T$ are genuine on $\Op D_\ot$ and restrict to $D_\ot$ as $\eta_{\ot}$.  
Since $(D_\ot, \eta_{\ot})$ determines the germ of the contact structure we may pick a ball $B \subset \Op D_\ot$
with $D_\ot \subset \p B$ and assume $({}^{T} B, {}^T \xi) = T \times (B, \xi)$.

By applying Proposition \ref{prop:unique-model-param} we may assume
there is a collection of disjoint  fibered balls ${}^{T_i}B_i \subset M \setminus (A \cup B)$ over a collection of discs $T_i \subset T$ 
for $i = 1, \dots, L$ such that 
\begin{enumerate}
\item ${}^T\xi$ is genuine away from $\bigcup_{i=1}^L \Int({}^{T_i}B_i)$ and
\item the fibered shells
$({}^{T_i}B_i, {}^{T_i}\xi)$ and $({}^{T_i}B_{K_{\univ}}, {}^{T_i}\eta_{K_\univ})$ are equivalent.
\end{enumerate}
Apply Proposition~\ref{prop:neighb-ot} to get $L$ disjoint balls 
$B_i' \subset \Int(B\setminus (D_\ot\cup A))$  with an overtwisted disc 
 $(D^i_{\ot}, \eta_{\ot}) \subset (\p B_i', \xi)$ in each of them. 
 
 It follows from Lemma \ref{lm:choosing-paths}, proven in Section \ref{sec:foliations-proofs} below, that for each $j$ we can find
 a parametric family of embedded paths $^{T_j}\gamma_j$ connecting $^{T_j}B_j$ to $^{T_j}B'_j$ in $T\times(M\setminus A\cup D_\ot)$.    Moreover, using Gromov's parametric $h$-principle for transverse paths, see \cite{Gr-PDR}, we can assume the constructed paths are transverse.

As we explain in Section~\ref{sec:BCS}, with these parametric paths we can form disjoint parametric 
ambient boundary connected sums ${}^{T_j}C_j \subset {}^{T_j}(M \setminus (A \cup D_\ot))$ for each $j=1, \dots, L$,
between the fibered shells $^{T_j}B_j$ and $^{T_j}B'_j$.  Furthermore, by Section~\ref{sec:BCS}
and property (ii) above we have isomorphisms of fibered almost contact structures
$$
	({}^{T_i}C_i, {}^{T_i}\xi) \cong ({}^{T_i}B_{K_\univ} \# {}^{T_i}B_i',\, {}^{T_i}\eta_{K_\univ} \# {}^{T_i}\xi).
$$
Applying Proposition~\ref{p:ConnectSumDisc-param} inductively for $j=1,\dots, L$  we  deform ${}^T\xi$ on these connected sums relative to their boundary
to get a fibered genuine contact structure on ${}^{T}M$.
\end{proof}
\medskip

\begin{proof}[Proof of Corollary \ref{cor:isocontact}] By an isotopy of $f$ we can arrange that the complement $M\setminus f(N)$ is overtwisted and the closure $f(N)$ is compact. Then, slightly reducing  if necessary  the manifold  $M$, we can assume it non-compact and overtwisted at infinity. 	 Let us exhaust $N$ by compact subsets: $N=\bigcup\limits_1^\infty C_j$,  such that $C_j\Subset\Int C_{j+1}$ and $V\setminus  C_j$ is connected, $ j=1,\dots $.  Set $C_0:=\varnothing$.  The result follows by induction from the following claim:

{\em Suppose we are given  an embedding    $f^{j-1}:N\to M$ which  is  contact on $\Op C_{j-1}$ and   a  homotopy  of bundle isomorphisms  $\Phi^{j-1}_{t} :TN\to  TM$   covering $f^{j-1}$ such that the following property  ${\rm P}^{j-1}$ is satisfied:
\begin{description}
\item{${\rm P}^{j-1}:$} The homotopy $\Phi^{j-1}_t$ is contact on $T(N)|_{\Op C_{j-1}}$ for all $t\in[0,1]$,  
 $\Phi^{j-1}_0$ is   contact everywhere, and 
   $\Phi^{j-1}_1=df^{j-1}$.
 \end{description}
Then  there exists a pair $(f^j,\Phi_t^j)$ which satisfies ${\rm P}^{j}$ and such that $f^{j-1}$ and $f^j$ are   isotopic via an isotopy fixed on $C_{j-1}$.}

 Let $\xi_t$ be a family of almost  contact structures on $M$ such that $\xi_t=(\Phi^{j-1}_t)_*\zeta$ on $f^{j-1}(C_j)$ and $\xi_t=\xi$ outside $f^{j-1}(C_{j+1})$. We note that $\xi_0=\xi$ on $f^{j-1}(C_j)$,  and $\xi_t=\xi$ on $f^{j-1}(C_{j-1})$ for all $t\in[0,1]$.
 Theorem \ref{thm:main} allows us to construct  a  compactly supported homotopy $\wt \xi_{t}$ of {\em genuine} contact structures on $M$,  $t\in[0,1]$,  connecting $\wt \xi_0=\xi$ and  a contact structure $\wt\xi_1$ which coincides with $\xi_1$ on $f^{j-1}(C_j)$. Moreover, this can be done to ensure  existence of a homotopy $\Psi_t:TM\to TM$ of bundle isomorphisms such that $\Psi_0=\Id$, $\Psi_t^*\wt\xi_t=\xi_t$,  and $\Psi_t|_{f^{j-1}(C_{j-1})}=\Id$, $t\in[0,1]$. Then Gray's  theorem \cite{Gray59} provides us with   a compactly supported diffeotopy $\phi_t:M\to M$, $t\in[0,1]$, such that $\phi_0=\Id$, $\phi_t^*\xi =\wt \xi_{t}$ and $\phi_t|_{f^{j-1}(C_{j-1})}=\Id$. 
   Set $f^j:=\phi_1\circ f^{j-1}$ and $\Phi^j_{t}:= d\phi_t\circ \Psi_t^*\circ\Phi^{j-1}$, $t\in[0,1]$.
 Then  $\Phi^j_1=df^j$, $(\Phi^j_t)^*\xi= (\Phi^{j-1}_t)^*\circ(\Psi_{t}) \circ (d\phi_t)^*\xi= (\Phi^{j-1}_t)^*\circ\Psi_t^*\wt\xi_{t}= (\Phi^{j-1}_t)^*\xi_t$.
 Hence,  $(\Phi^j_t)^*\xi|_{C_j}=\zeta$ for all $t\in[0,1]$. We also have
 $(\Psi^j_0)^*\xi=\zeta$ everywhere. Thus, the pair
 $(f^j,\Phi^j_t)$ satisfies ${\rm P}^{j}$, and the claim follows by induction.
   \end{proof}


\section{Domination and conjugation for Hamiltonian contact shells}\label{sec:conjugation}

Recall the notation $(K, \Delta)$ to refer to a contact Hamiltonian $K$ on a star-shaped domain
$\Delta \subset (\R^{2n-1}, \xi_\st)$ such that $K|_{\p \Delta \times S^1} > 0$ as in \eqref{e:KBoundary}.

In this section we will develop two properties of Hamiltonian contact shells that make them well-suited for the purposes of this paper.  Namely in Section~\ref{sec:PartialOrder} we show that a natural partial order $(K, \Delta) \leq (K', \Delta')$ is compatible with the partial order on contact shells given by domination, and in Section~\ref{sec:effect} we show the action of
$\CCont(\Delta)$ on a contact Hamiltonian $(K, \Delta)$ by conjugation preserves the equivalence class of the associated contact shell.  

A simple, but very important  observation is then made in Section~\ref{sec:domin} where  we show how  conjugation can be used to make  some
 contact Hamiltonians $(K, \Delta)$  much smaller with respect to the partial order.  For instance in the $3$-dimensional case
 where $\Delta \subset \R$ is an interval, we prove that up to conjugation $K: \Delta \to \R$ is a minimal element for the partial order if $K$ is somewhere negative.  In higher dimensions, the existence of a minimal element up to conjugation seems to no longer be true, but the weaker Propositions~\ref{prop:disorder} and \ref{prop:shallow} hold in general and they suffice for our purposes.


\subsection{A partial order on contact Hamiltonians with domains}\label{sec:PartialOrder}

Let us introduce a partial order on contact Hamiltonians with domains where 
$$(K, \Delta) \leq (K', \Delta')$$
is defined to mean $\Delta \subset \Delta'$ together with
\begin{align}
	K(x,t) &\leq K'(x,t) \quad\mbox{for all $x \in \Delta$ and} \label{e:DirectDomConditon}\\
	0 &< K'(x,t) \quad\mbox{for all $x \in \Delta' \setminus \Delta$.} \label{e:HamDomBasicCondition}
\end{align}
This partial order is compatible with domination of contact shells.
 \begin{lemma}\label{l:HamiltonianShellsAnnulus}
	If $(K, \Delta) \leq (K', \Delta')$, then $(B_{K}, \eta_{K})$ is dominated by
	$(B_{K'}, \eta_{K'})$.  More specifically, given a contact shell $(B_{K,C}, \eta_{K, \rho})$ there is a shell
	$(B_{K'\!,C'}, \eta_{K'\!, \rho'})$ such that 
	$$(B_{K,C}, \eta_{K, \rho}) \subset (B_{K'\!,C'}, \eta_{K'\!, \rho'})$$ 
	the inclusion is a subordination map.
\end{lemma}
\begin{proof}
	If $C' \geq C$, then by \eqref{e:DirectDomConditon} we have
	$(B_{K,C}, \eta_{K,\rho}) \subset (B_{K'\!,C'}, \eta_{K'\!,\rho'})$
	and it will be an embedding of almost contact structures whenever
	$$
	\mbox{$\rho' = \rho$ \quad on \quad $\Op B_{K,C} \subset B_{K',C'}$.}
	$$
	If we pick the extension so that
$$		
		\mbox{$\p_v \rho'_{(x,t)}(v) > 0$ on $\Op\{x \in \Delta, v \geq K(x,t) +C\} \cup \Op\{x \in \Delta' \setminus \Int \Delta\}$\,,}
$$
	which is possible on the latter region by \eqref{e:HamDomBasicCondition}, it follows that 
	$\eta_{K'\!,\rho'}$ is contact on $\Op (B_{K'\!,C'} \setminus \Int B_{K,C})$
	and hence the inclusion is a subordination map.
\end{proof}

\subsection{Conjugation of contact Hamiltonians}\label{sec:effect}

Given a contact manifold $(M, \alpha)$ and a contact Hamiltonian 
$K: M \times S^1 \to \R$ let $\{\phi^t_K\}_{t \in [0,1]}$ be the unique contact isotopy with $\phi^0_K = \id$ and
$$\alpha(\p_t \phi^t_K(x)) = K(\phi^t_K(x), t).$$
For a contactomorphism $\Phi : (M, \alpha) \to (M', \alpha')$ define the push-forward Hamiltonian
\begin{equation}\label{e:ContPushK}
	\Phi_*K: M' \times S^1 \to \R \quad\mbox{by}\quad (\Phi_*K)(\Phi(x),t) = c_{\Phi}(x)\,K(x,t)\,. \quad
\end{equation}
where $c_\Phi: M \to \R_{> 0}$ satisfies $\Phi^*\alpha' = c_\Phi\, \alpha$.
One can verify
$$
	\{\Phi \phi^t_K \Phi^{-1}\}_{t \in [0,1]} = \{\phi^t_{\Phi_*K}\}_{t \in [0,1]}
$$
so $\Phi_*$ corresponds with conjugating by $\Phi$.

In this paper we will primarily be concerned with contactomorphisms $\Phi: \Delta \to \Delta'$ between star-shaped domains
in $(\R^{2n-1}, \xi_{\st})$ where $c_\Phi: \Delta \to \R_{>0}$ is defined by
$$
	\Phi^*\lambda_{\st} = c_\Phi\,\lambda_{\st}.
$$
It is clear that if $(K, \Delta)$ satisfies \eqref{e:KBoundary}, then $(\Phi_*K, \Delta')$ does as well.  As the next lemma shows
the push-forward operation induces an equivalence of contact shells.

\begin{lemma}\label{l:PhiToTildePhi}
A contactomorphism between star-shaped domains $\Phi: \Delta \to \Delta'$ in $(\R^{2n-1}, \xi_{\st})$
induces an equivalence of the contact shells
$$
	\wh\Phi: (B_{K}, \eta_{K}) \to (B_{\Phi_*K}, \eta_{\Phi_*K})
$$
defined by $(K, \Delta)$ and $(\Phi_*K, \Delta')$.
\end{lemma}
\begin{proof}
	For a given model $(B_{K,C}, \eta_{K, \rho})$ we will build a model $(B_{\Phi_*K,\tilde{C}}, \eta_{\Phi_*K, \tilde{\rho}})$
	such that the two models are isomorphic as almost contact structures.
	
	For $\tilde{C} + \min(\Phi_*K) > 0$, pick a family of diffeomorphisms for $(x,t) \in \Delta \times S^1$
	$$
		\phi_{(x,t)}: [0, K(x,t)+C] \to [0, c_{\Phi}(x)K(x,t) + \tilde{C}] 
	$$
	and define a smooth family of functions for $(x, t) \in \Delta \times S^1$ 
	$$
		\tilde{\rho}_{(\Phi(x), t)}: [0, c_{\Phi}(x)K(x,t) + \tilde{C}] \to \R \quad
		\mbox{by}\quad
		\tilde{\rho}_{(\Phi(x), t)}(v) = c_\Phi(x) \rho_{(x,t)}(\phi_{(x,t)}^{-1}(v)).
	$$
	One see one $\tilde{\rho}$ satisfies the conditions in \eqref{e:FamilyRhoS1} to define
	$(B_{\Phi_*K, \tilde{C}}, \eta_{\Phi_*K, \tilde{\rho}})$ provided
	$$
	\mbox{$\phi_{(x,t)}(v) = c_{\Phi}(x)(v - C)+ \tilde{C}$ \quad for \quad $(x,t,v) \in \Op\{v = K(x,t) +C\}$.}
	$$
	It follows by construction that the diffeomorphism
	$$
	\wh\Phi: (B_{K,C}, \eta_{K,\rho}) \to (B_{\Phi_*K, \tilde{C}}, \eta_{\Phi_*K,\tilde{\rho}})
	\quad\mbox{defined by}\quad
	\wh\Phi(x,v,t) = (\Phi(x), \phi_{(x,t)}(v), t)
	$$
	is an isomorphism of almost contact structures.
\end{proof}

\subsubsection{Foliations of overtwisted discs}

For a first example of this push-forward procedure, we will prove Proposition~\ref{prop:neighb-ot}
as a corollary of Lemma~\ref{l:PhiToTildePhi} above and Lemma~\ref{lem:Scale} below. 
For $\delta \in \Op \{1\}$ observe the contactomorphism $C_\delta: \R^{2n-1} \to \R^{2n-1}$ given by 
$$
	C_\delta(u_1, \dots, u_{n-1},\phi_1, \dots, \phi_{n-1}, z) = 
		\left(\frac{u_1}{\delta}, \dots, \frac{u_{n-1}}{\delta}, \phi_1, \dots, \phi_{n-1}, \frac{z}{\delta}\right)
$$
maps $C_{\delta}(\Delta_\delta) = \Delta_{\cyl}$ where $\Delta_\delta := \{\abs{z} \leq \delta,\, u \leq \delta\}$.
\begin{lemma}\label{lem:Scale}
	Let $K:\Delta_\cyl \to \R$ a special contact Hamiltonian and define
	$$
		K_{\delta}: \Delta_\delta \to \R \quad\mbox{by}\quad K_{\delta}:= K + (\delta-1).
	$$
	If $\delta < 1$ is sufficiently close to $1$, then 
	$\tilde{K}_\delta:= (C_\delta)_* K_{\delta} : \Delta_{\cyl} \to \R$ is also a special contact Hamiltonian.
\end{lemma}
\begin{proof}
	Provided $\delta$ is sufficiently close to $1$, it is clear that
	$$
		\tilde{K}_\delta := (C_\delta)_* K_{\delta} = \frac{K \circ C_\delta^{-1}}{\delta} + \frac{\delta -1}{\delta}
	$$
	satisfies items (i)-(iv) in Definition~\ref{def:SpecialHam} with 
	$\tilde{z}_D = z_D / \delta$.
	If $k: \R_{\geq 0} \to \R$ is the special function for $K$, then let 
	$\tilde{k}_\delta(u) := \tfrac{k(\delta u)}{\delta} + \tfrac{\delta-1}{\delta}$.
	Using that $k$ is special, computing for $a > 1$ and recalling $\delta < 1$, we see
	$$
		a\,\tilde{k}_\delta(u/a) -\tilde{k}_{\delta}(u)
		< (a-1)\tfrac{\delta - 1}{\delta} < 0
	$$
	so therefore $\tilde{k}_\delta$ is special provided $\delta$ is close enough to $1$ so that $\tilde{k}_\delta(1) > 0$.
	Therefore item (v)
	in Definition~\ref{def:SpecialHam} holds for $\tilde{K}_\delta$ with respect to $\tilde{k}_\delta$.
%
%
\end{proof}

\begin{proof}[Proof of Proposition~\ref{prop:neighb-ot}]
	Consider an overtwisted disc $(D_{K}, \eta_{K})$ defined by a special contact Hamiltonian $K: \Delta_{\cyl} \to \R$.  
	For $\delta \in [1-\eps, 1]$, let $\Delta_\delta = \{\abs{z} \leq \delta,\, u \leq \delta\}$
	and consider the family of contact Hamiltonians
	$$
		K_{\delta}: \Delta_\delta \to \R \quad\mbox{where}\quad K_{\delta}:= K + (\delta - 1).
	$$	
	Observe any neighborhood of $(\p B_{K}, \eta_{K})$ contains
	a foliation $\{(\p B_{K_{\delta}}, \eta_{K_{\delta}})\}_{\delta \in [1-\eps, 1]}$ provided
	$\eps > 0$ is small enough.
	
	Furthermore, when $\eps > 0$ is sufficiently small, Lemmas~\ref{l:PhiToTildePhi} and \ref{lem:Scale} give us
	a family of special contact Hamiltonians
	$\{\tilde{K}_\delta: \Delta_{\cyl} \to \R\}_{\delta \in [1-\eps, 1]}$ such that
	$\tilde{K}_\delta < K_\univ$ together with contactomorphisms
	$$
		(\p B_{K_{\delta}}, \eta_{K_{\delta}}) \cong (\p B_{\tilde{K}_{\delta}}, \eta_{\tilde{K}_{\delta}})\,.
	$$
	Therefore every neighborhood of $(D_{K}, \eta_{K})$ contains a foliation
	$\{(D_{K_{\delta}}, \eta_{K_{\delta}})\}_{\delta \in [1-\eps, 1]}$ of overtwisted discs.
\end{proof}

\subsubsection{Embeddings of contact Hamiltonian shells}

As a second application of the push-forward procedure, we have the following lemma about embeddings of contact Hamiltonian shells.

\begin{lemma}\label{lem:contain}
	Let $(B_{K, C}, \eta_{K, \rho})$ be a contact shell structure for $(K, \Delta)$.  For any other
	$(K', \Delta')$ there exists a contact shell structure $(B_{K', C'}, \eta_{K', \rho'})$ together with
	an embedding of almost contact structures
	$$
		(B_{K, C}, \eta_{K, \rho}) \to (B_{K'\!, C'}, \eta_{K'\!, \rho'}).
	$$
	If $\Delta \subset \Int \Delta'$, then the embedding can be taken to be an inclusion map.
\end{lemma}
\begin{proof}
	Since $\Delta'$ is star-shaped, there is a contactomorphism $\Phi \in \CCont_0^c(\R^{2n-1})$ such that 
	$\Delta \subset \Int \Phi(\Delta')$ and therefore by Lemma~\ref{l:PhiToTildePhi} without loss of generality we may
	assume $\Delta \subset \Int \Delta'$.
	
	Given the contact shell structure $(B_{K, C}, \eta_{K, \rho})$, pick any contact shell $(B_{K'\!,C'}, \eta_{K'\!,\rho'})$
	subject to the additional conditions that
	\begin{equation}\label{e:containCon}
		K'(x,t) + C' > K(x,t) + C \mbox{ for all $(x,t) \in \Delta \times S^1$}
	\end{equation}
	and the smooth family of functions 
	$\rho'_{(x,t)} : \R_{\geq 0} \to \R$ for $(x,t) \in \Delta' \times S^1$ satisfy
	\begin{equation}\label{e:embed}
		\rho' = \rho \quad\mbox{on\quad $\Op B_{K,C} \subset B_{K',C'}$},
	\end{equation}
	where the latter is always possible since $\Delta \subset \Int(\Delta')$.
	By \eqref{e:containCon} we have an inclusion 
	\begin{equation}\label{e:inclusion}
		(B_{K,C}, \eta_{K,\rho}) \subset (B_{K'\!,C'}, \eta_{K'\!,\rho'})
	\end{equation}
	and by \eqref{e:embed} it is an embedding of almost contact structures.
\end{proof}

\begin{remark}
{\rm If the inclusion \eqref{e:inclusion} was a subordination map, then
$$
	\p_v \rho'_{(x,t)}(v) > 0 \quad\mbox{ for all $(x,v,t) \in \Op\{x \in \Delta\,, v \geq K(x,t) + C\}$}
$$
which together with \eqref{e:containCon} and \eqref{e:embed} imply $K'(x,t) > K(x,t)$ for all $x \in \Delta$ since
$$
	K'(x,t) - K(x,t) = \rho'_{(x,t)}(K'(x, t)+C') - \rho'_{(x,t)}(K(x, t)+C) > 0 \,.
$$
A similar argument shows why assuming $\Delta \subset \Delta'$ is not sufficient, since the conditions \eqref{e:containCon}, 
\eqref{e:embed}, and $\p_v\rho' > 0$ on $\p\Delta'$ imply $K'(x,t) > K(x,t)$ for all $x \in \partial \Delta \cap \p\Delta'$}
\end{remark}

\subsubsection{Changing the contactomorphism type of the domain}

Recall that star-shaped domains $\Delta \subset (\R^{2n-1}, \xi_\st)$ are ones for which the contact vector field
$Z = z \tfrac{\p}{\p z} + u \tfrac{\p}{\p u}$ is transverse to $\p\Delta$ and we denote the flow of a vector field $X$ by $X^t$.
While not all star-shaped domains are contactomorphic, up to mutual domination of contact shells $(B_K, \eta_K)$
the choice of domain does not matter.

 \begin{lemma}\label{lm:changing-B}
 For any contact Hamiltonian $(K, \Delta)$ and star-shaped domain $\Delta'$ there is a
 contact Hamiltonian $(K', \Delta')$ such that $(B_K, \eta_K)$ dominates $(B_{K'}, \eta_{K'})$.
 \end{lemma}
 \begin{proof}
 For any neighborhood $U\supset \p \Delta$ there is a contactomorphism $\Phi \in \CCont_0^c(\R^{2n-1})$
 such that $\Phi(\Delta') \subset \Delta$ and $\Phi(\p \Delta')\subset U$.  
 To see this first note without loss of generality we may assume $\Delta' \subset \Delta$ by replacing $\Delta'$
 with $Z^{-N}(\Delta')$ for $N$ sufficiently large.
 After this reduction, the required contactomorphism is given by $\tilde{Z}^T$ for $T$ sufficiently large,
 where $\tilde{Z}$ is a contact vector field with $\supp(\tilde{Z}) \subset \Int \Delta$
 and $\tilde{Z} = Z$ on $\Op Z^{-\eps}(\Delta)$ where $\Delta \setminus Z^{-\eps}(\Delta) \subset U$.
 
 Now pick $U \supset \p \Delta$ to be such that $K|_{U \times S^1} > 0$, take the constructed contactomorphism $\Phi$ above,
 and consider the contact Hamiltonian $K' = \Phi^{-1}_*(K|_{\Phi(\Delta')})$ on $\Delta'$.  It follows from 
 Lemmas~\ref{l:HamiltonianShellsAnnulus} and \ref{l:PhiToTildePhi} that $(B_K, \eta_K)$ dominates $(B_{K'}, \eta_{K'})$.
 \end{proof}

\subsection{Domination up to conjugation}\label{sec:domin}


If we want to prove the contact shell $(B_K, \eta_K)$ is dominated by the shell $(B_{K'}, \eta_{K'})$
then Lemmas~\ref{l:HamiltonianShellsAnnulus} and \ref{l:PhiToTildePhi} instruct us to care about the partial order
from Section~\ref{sec:PartialOrder} up to conjugation.  In particular it is enough to find a contact embedding $\Phi: \Delta \to \Delta'$
such that $(\Phi_*K, \Phi(\Delta)) \leq (K', \Delta')$ to prove $(B_K, \eta_K)$ is dominated $(B_{K'}, \eta_{K'})$.

\subsubsection{Minimal elements up to conjugation in the $3$-dimensional case}

In the $3$-dimensional case where $\Delta \subset \R$ is always a closed interval, 
up to conjugation any somewhere negative Hamiltonian $(K, \Delta)$ is minimal with respect to the partial order from
Section~\ref{sec:PartialOrder}.
 
 \begin{lemma} \label{lm:3-domin}
 Let $(K, \Delta)$ be somewhere negative where $\Delta = [-1, 1]$.  For any other contact Hamiltonian
 $(\wt{K}, \Delta)$ there is a contactomorphism $\Phi \in \CCont_0(\Delta)$ such that
  $(\Phi_*K,\Delta) \leq (\wt K,\Delta)$, and hence $(B_K,\eta_K)$ is dominated by $(B_{\wt K},\eta_{\wt K})$.
%
%
 \end{lemma}
\begin{proof}
Without loss of generality assume $K(0) < 0$ and pick a $\delta > 0$ such that
\begin{align*}
K(z) &< -\delta  \mbox{ if $\abs{z} \in [0,a]$ for some $a \in (0,1)$, and}\\
\tilde{K}(z) &> \delta  \mbox{ if $\abs{z} \in [1-\eps, 1]$ for some $\eps \in (0,1)$.}
\end{align*}
For $0 < \sigma \ll 1$, pick a diffeomorphism $\Phi: [-1,1] \to [-1, 1]$ such that linearly
$$
\mbox{$[-\sigma,\sigma]$ maps onto $[-1+2\sigma,1-2\sigma]$\quad and\quad $\pm [2\sigma, 1]$ maps onto $\pm [1-\sigma, 1]$.}
$$
Since $(\Phi_*K)(\Phi(z)) = \Phi'(z) K(z)$ we can pick $\sigma$ sufficiently small so that
\begin{align*}
(\Phi_*K)(z) &< -\tfrac{1-2\sigma}{\sigma}\, \delta < \tilde{K}(z) & \mbox{ if $\abs{z} \in [0, 1-2\sigma]$}\\
(\Phi_*K)(z) &< 0 < \tilde{K}(z) & \mbox{ if $\abs{z} \in [1-2\sigma, 1-\sigma]$}\\
(\Phi_*K)(z) &\leq \tfrac{\sigma}{1-2\sigma}\, \max(K) < \delta < \tilde{K}(z) &\mbox{ if $\abs{z} \in [1-\sigma, 1]$}
\end{align*}
and hence get that $\Phi_*K < \tilde{K}$.
\end{proof}

As a consequence of this lemma, the $3$-dimensional case simplifies by making Section~\ref{sec:universal} unnecessary
and allows us to give an effective description of an overtwisted disc.  It seems unlikely to us (though we do not have a proof) disc that
the generalization of Lemma~\ref{lm:3-domin} holds when $\dim(\Delta_{\cyl}) \geq 3$.  The immediate obstacle to adapting the proof 
is essentially that  $\Op\p\Delta_{\cyl}$  is not a star-shaped domain in higher dimensions, while for $\Delta_\cyl = [-1, 1]$ we get two intervals which are star-shaped.

\subsubsection{Remnants of the $3$-dimensional case}\label{ssec:disorder}

Proposition~\ref{prop:disorder} and its corollary Proposition~\ref{prop:shallow} below,
represent the remnants of the $3$-dimensional Lemma~\ref{lm:3-domin} that survive to higher dimensions.

Proposition~\ref{prop:shallow} essentially says that up to conjugation the only part of $(K, \Delta)$ that is relevant for the partial order is $K|_{\{K \geq 0\}}$ whereas for instance $\min(K)$ is irrelevant if $K < 0$ somewhere.
It will play a key role in Section~\ref{sec:universal}  
where we prove the existence universal contact shells.

Given a domain $\Delta\subset\R^{2n-1}_\st$ let 
$$
	F_+(\Delta) := \{K \in C^0(\Delta): \supp(K) \subset \Int \Delta,\, K \geq 0\,, \mbox{ and } K \not= 0\}
$$ 
and consider the action of $\cD_0(\Delta):=\mbox{Cont}^c_0(\Int\Delta)$ on $F_+(\Delta)$ given by
$$
	\Phi_*K :=(c_{\Phi} \cdot K) \circ \Phi^{-1} \quad\mbox{for $K\in F_+(\Delta)$ and $\Phi\in\cD_0(\Delta)$}
$$
i.e.\ the push-forward operation from \eqref{e:ContPushK}.

 \begin{prop}\label{prop:disorder}
 If $\Delta \subset (\R^{2n-1}, \xi_{\st})$ is star-shaped, then for any two $K, H\in F_+(\Delta)$ there is a contactomorphism 
 $\Phi\in\cD_0(\Delta)$ such that $\Phi_*K\geq H$.
 \end{prop}
 \begin{proof}
 Without loss of generalize assume $\Delta$ is star-shaped with respect to the radial vector field 
 $Z$ and that $K(0) > 0$.  Pick a sufficiently small neighborhood $U \ni 0$ so that for some $T >0$:
 $$
 	\mbox{$\inf_U(K) > 0$\,, \quad $\supp(H) \subset Z^T(U)\subset\Int \Delta$\,, \quad and \quad $e^T \inf_U(K) > \max(H)$}
$$
where $Z^t:\R^{2n-1}\to\R^{2n-1}$ is the flow of $Z$ and satisfies $(Z^t)^*\lambda_{\st} = e^t \lambda_{\st}$.
Let $\tilde{Z}$ be another contact vector field supported in $\Int(\Delta)$ and equal to $Z$ on $Z^T(U)$.
It follows the contactomorphism $\Phi := \tilde{Z}^T \in \cD_0(\Delta)$ satisfies $\Phi_* K \geq H$ since
$$
	(\Phi_*K)(x) = (c_\Phi \cdot K)(\Phi^{-1}(x)) \geq e^T \inf_U(K) \geq H(x) \quad\mbox{if $x \in \supp(H)$}
$$
and $\Phi_*K \geq 0$ otherwise.
 \end{proof}
 
 Note that Proposition \ref{prop:disorder} shows that on the conjugacy classes of elements of the positive cone
 $\CC:=\{f\in\cD_0;\; f\geq\Id,\;f\neq\Id\}$ partial order from \cite{EliPolt00} is trivial
 and it would be to understand for which contact manifolds the analog of Proposition \ref{prop:disorder} holds.
 As pointed out to us by L.~Polterovich, a non-trivial bi-invariant metric on $\CCont_0^c$ compatible with the notion of order
 on $\CCont_0^c$ from \cite{EliPolt00} provides an obstruction to Proposition \ref{prop:disorder}.  For instance
 Sandon's metric \cite{Sandon} shows Proposition \ref{prop:disorder} does not hold for $D^{2n}_R\times S^1$ with contact form $dz+\sum_{i=1}^n u_id\phi_i$ where $D^{2n}_R$ is a $2n$-disc of a sufficiently large radius $R$.  
%

As an application of  Proposition \ref{prop:disorder} we show in this next proposition
that condition \eqref{e:DirectDomConditon} in the definition of
the partial order $(K, \Delta) \leq (K', \Delta')$ from Section~\ref{sec:PartialOrder} can be weaken so that there is still
domination of the contact shells.

		\begin{figure}[h]
   \centering
   \def\svgwidth{400pt}
   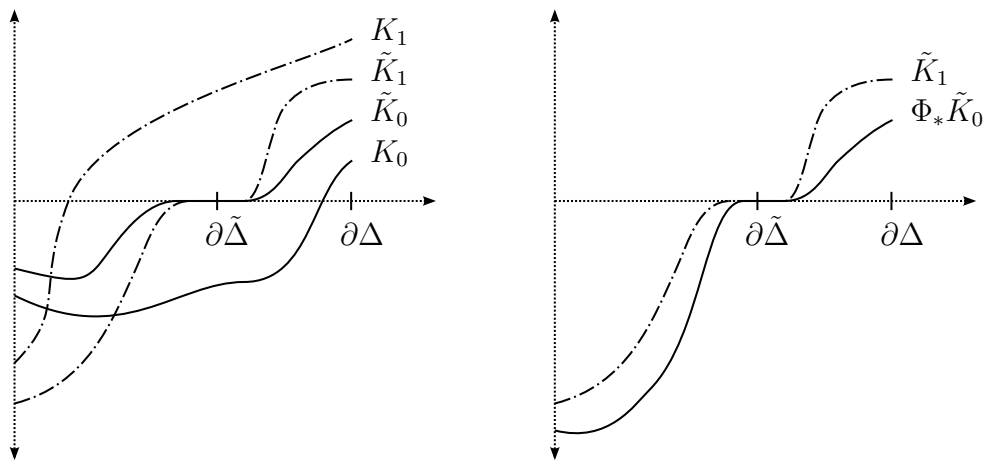
   \caption{Schematic representation of the proof of Proposition~\ref{prop:shallow}}
  \label{f:shallow}
\end{figure}

\begin{prop}\label{prop:shallow}
Consider contact Hamiltonians $K_i: \Delta \to\R$ defining contact shells $(B_{K_i}, \eta_{K_i})$ for $i=1,2$.
If there is a star-shaped domain $\tilde\Delta \subset \Int \Delta$ such that
$$
	K_0 \leq K_1 \mbox{ on $\Op(\Delta \setminus \Int \tilde\Delta)$,} \quad
	0 \leq K_1 \mbox{ on $\Op \partial \tilde\Delta$,} \quad \mbox{and}\quad
	K_0 \leq 0 \mbox{ on $\Op \tilde\Delta$}
$$
with $K_0|_{\Int \Delta} \not\equiv 0$,
then the contact shell $(B_{K_0}, \eta_{K_0})$ is dominated by 	$(B_{K_1}, \eta_{K_1})$.  
\end{prop}
\begin{proof} 
The assumptions ensure we can pick contact Hamiltonians $\tilde{K}_i: \Delta \to \R$ defining contact shells $(B_{\tilde{K}_i}, \eta_{\tilde{K}_i})$ so that
\begin{enumerate}
\item $K_0 \leq \tilde{K}_0$ and $\tilde{K}_1 \leq K_1$, 
\item $\tilde{K}_0 \leq \tilde{K}_1$ on $\Delta \setminus \tilde{\Delta}$, and
\item $-\tilde{K}_i|_{\tilde{\Delta}} \in F_+(\tilde{\Delta})$ for $i=1,2$.
\end{enumerate}
By item (i) and Lemma~\ref{l:HamiltonianShellsAnnulus} it suffices to show $(B_{\tilde{K}_0}, \eta_{\tilde{K}_0})$
is dominated by $(B_{\tilde{K}_1}, \eta_{\tilde{K}_1})$.

Applying Proposition~\ref{prop:disorder} to item (iii) gives a $\Phi \in \CCont_0^c(\Int \tilde{\Delta})$ such that
$$
	\Phi_*(\tilde{K}_0|_{\tilde{\Delta}}) \leq \tilde{K}_1|_{\tilde{\Delta}}.
$$
Together with item (ii) this means $\Phi_*\tilde{K}_0 \leq \tilde{K_1}$ where we think of $\Phi \in \CCont_0^c(\Int \Delta)$ and therefore 
$(B_{\tilde{K}_0}, \eta_{\tilde{K}_0})$
is dominated by $(B_{\tilde{K}_1}, \eta_{\tilde{K}_1})$ by Lemmas~\ref{l:PhiToTildePhi} and \ref{l:HamiltonianShellsAnnulus}.
\end{proof}

 We also have the following parametric version of Proposition \ref{prop:shallow}.
 
 \begin{prop}\label{prop:shallow-param}
Assume that $\Delta \subset \R^{2n-1}$ is a star-shaped domain.  	Let $\Delta' \subset \Delta$ be a smooth star-shaped subdomain  and let $K^\tau: \Delta \to \R$, $\tau\in T$,  be a family  time-independent functions satisfying
	$K^\tau|_{\Delta \setminus \Int \Delta'} > 0$.  Suppose that  $K^\tau>0$ for $\tau$ in a closed subset $A\subset T$.
	 Then for any $\delta > 0$, there exists a family  $\wt K^\tau$ such that
	 \begin{itemize}
	 \item $\wt K^\tau=K^\tau$ on 
	 $\Delta \setminus \Int \Delta'$ and  $\wt{K^\tau} > -\delta$,  $\tau\in T$;
	 \item $\wt K^\tau= K^\tau$ for $\tau\in  A$
\item there exists a family of subordination maps $h^\tau: \eta_{{\wt K}^\tau}\to \eta_{ {K^\tau}}$
which are identity maps 	for $\tau\in A$.
\end{itemize} 
\end{prop}

\section{Filling of the universal circular models}\label{sec:tripling}
We prove in this section Propositions \ref{p:ConnectSumDisc} and \ref{p:ConnectSumDisc-param}.
In this section, we will always take 
$$
	\Delta =\Delta_\cyl= \{ u \leq 1,\, \abs{z} \leq 1\} \subset (\R^{2n-1}, \xi_{\st})
	\quad\mbox{where}\quad u = u_1 + \dots + u_{n-1}.
$$ 
  All contact Hamiltonians $(K, \Delta)$ will be assumed time indepedent and spherically symmetric, i.e.
  functions $K(u,z)$ of only the $u$ and $z$ variables.

The contactomorphism of $(\R^{2n-1}, \xi_{\st})$ that is translation in the $z$-coordinate will be
$$
\mbox{$Z_\tau : \R^{2n-1} \to \R^{2n-1}$ \quad where \quad $Z_\tau(q,z) = (q, z +\tau)$}
$$
using coordinates $(q, z) \in \R^{2n-2} \times \R$.

\subsection{Boundary connected sum}\label{sec:BCS}

\subsubsection{Abstract boundary connected sum}\label{s:AbstractBCS}
	Consider the $\R^{2n}$ with polar coordinates 
	$(u_1, \varphi_1, \dots, u_{n-1}, \varphi_{n-1}, v, t)$ equipped with the radial
	Liouville form and vector field
	$$
		\theta := \sum_{i=1}^{n-1} u_i\, d\varphi_i + v\, dt\,\quad\mbox{and}\quad L := \sum_{i=1}^{n-1} u_i \tfrac{\p}{\p u_i}
		+ v \tfrac{\p}{\p v}.
	$$
	and denote by $L^t:\R^{2n}\to\R^{2n}$ the Liouville flow.
	
	 A \emph{gluing disc} for 
	a contact shell $(W, \zeta)$ is a smooth embedding $\iota: D \to \p W$, where $D \subset \R^{2n}$ is 
	compact domain, star-shaped with respect to $L$, and with piecewise smooth boundary 
	such that  $\iota^* \alpha = \theta$  for a choice of  a contact form 
	$\alpha$ for $\zeta$ in $\Op \p W$.  Note this implies $\iota(0) \in \p W$ is a gluing place in the sense of 
	Section~\ref{sec:shells} and that the Reeb vector field $R_{\alpha}$ is transverse to $\iota(D)$.
	
	Given contact shells $(W^{2n+1}_\pm, \zeta_\pm)$  with gluing discs $\iota_\pm: D \to \p W_\pm$
	such that $\iota_+$ preserves and $\iota_-$ reverses orientation, the Reeb flows define contact embeddings
	\begin{equation}\label{eq:Darboux-gluing}
	\begin{split}
	&\Phi_+: D \times (-\eps, 0] \to \Op \iota_+(D) \quad\mbox{with} \;\Phi_+^*\alpha_+ = dz + \theta,\\ 
	&	\Phi_-:  D \times [0, \eps) \to \Op \iota_-(D) \quad\mbox{with} \;  \Phi_-^*\alpha_- = dz + \theta, 
		\end{split}
	\end{equation}
	such that $\Phi_\pm|_{D \times 0} = \iota_\pm$.
	For $\ell > 0$ pick a  smooth function $\beta: [-\ell,\ell]\to\R_{\geq 0}$ such that $\beta(z)=0$ for $z$ near $\pm\ell$
	and denote  $D(z):=L^{-\beta(z)}(D)$. 
	Define the \emph{abstract boundary connected sum} to be the almost contact manifold
	\begin{equation}\label{e:ABCS}
		(W_+ \#_{T} W_-\,, \zeta_+ \#_{T} \zeta_-) := 
		\Big((W_+, \zeta_+) \cup (T, \ker(dz+\theta)) \cup (W_-, \zeta_-)\Big)/\sim
	\end{equation}
	where
	\begin{equation}\label{e:Tube}
		T = \{(p,z) \in \R^{2n} \times [-\ell, \ell]: p \in D(z)\} \subset \R^{2n+1}
	\end{equation}
	and one identifies 
	$$
		\Phi_+(p, 0) \sim (p, -\ell) \in T \quad\mbox{and}\quad \Phi_-(p,0) \sim (p,\ell) \in T\,.
	$$

\subsubsection{Abstract connected sum of $S^1$-model contact shells}\label{s:S1ConnectSum}

Consider a Hamiltonian contact shell $(B_{K,C}, \eta_{K, \rho})$
associated to a contact Hamiltonian $(K, \Delta)$.
There are canonical gluing discs
$$
	D_{\pm} = \{ u \leq 1,\, v \leq K(u,\pm 1)\} \subset \R^{2n}
$$
with maps $\iota_{\pm}: D_{\pm} \to (\p B_{K,C}, \eta_{K, \rho})$
$$
	 \iota_\pm(q, v, t)
	= (q, \pm 1, \rho_{(q,\pm1)}^{-1}(v), t) \in \R^{2n-1} \times \R^{2},
$$
where $\iota_\pm(0, 0) = (0, \pm1, 0)$ are the north and south poles of $B_K$.

For two contact Hamiltonians $K_\pm: \Delta \to \R$ assume $E(u) = K_\pm(u, \pm 1)$ is well defined.
For any $\ell > 0$ and a smooth function $\beta: [-\ell, \ell] \to \R_{\geq 0}$ such that $\beta = 0$ near $z = \pm\ell$,
define the domain
\begin{equation}\label{e:ConnectedSumDomain}
	\Delta\#_{\beta, \ell}\Delta := Z_{1+\ell}^{-1}(\Delta) \cup T_{\beta, \ell} \cup Z_{1+\ell}(\Delta) \subset \R^{2n-1}
\end{equation}
where 
\begin{equation}\label{e:domaintube}
T_{\beta, \ell} := \{u \leq e^{-\beta(z)},\, \abs{z} \leq 1\} \subset \R^{2n-1}
\end{equation}
 and define the contact Hamiltonian 
$K_+ \#_{\beta} K_-: \Delta\#_{\beta, \ell}\Delta \to \R$ by
$$
	(K_+\#_\beta K_-)(u,z) = \begin{cases}
	(K_+ \circ Z_{1+\ell})(u,z)&\quad\mbox{on $Z_{1+\ell}^{-1}(\Delta)$}\\
	e^{-\beta(z)} E(u) &\quad\mbox{for $(z,q) \in T_{\beta, \ell}$}\\
	(K_- \circ Z_{1+\ell}^{-1})(u,z) &\quad\mbox{on $Z_{1+\ell}(\Delta)$}
	\end{cases}
$$
Going forward we will drop $\beta$ from the notation when $\beta \equiv 0$.
\begin{figure}
   \centering
   \def\svgwidth{400pt}
   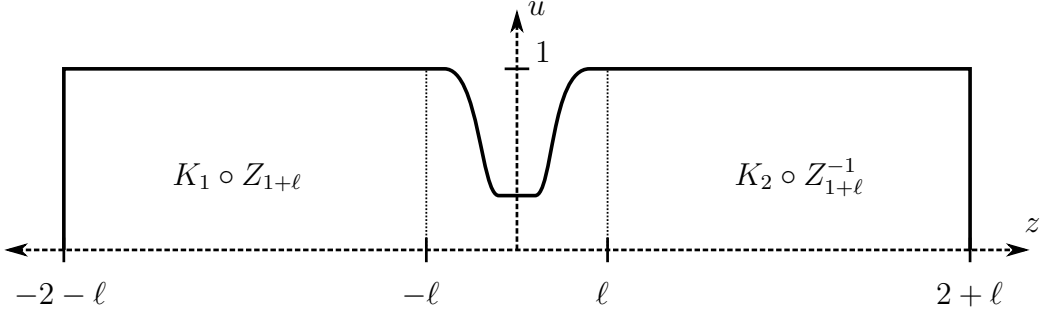
   \caption{The domain of the Hamiltonian $K_+\#_\beta K_-: \Delta\#_{\beta, \ell}\Delta \to \R$.}
  \label{f:ConnectSumDomain}
\end{figure} 

It follows from Example~\ref{ex:domain} below that
$\Delta\#_{\beta, \ell}\Delta$ is star-shaped since it is contactomorphic to $\Delta\#_{\ell} \Delta$, 
which is star-shaped with respect to $Z = \tfrac{\p}{\p z} + L$, and hence $(K_+ \#_\beta K_-, \Delta\#_{\beta, \ell}\Delta)$
defines an $S^1$-model contact shell 
$$(B_{K_+ \#_\beta K_-}, \eta_{K_+ \#_\beta K_-})$$
as in Section~\ref{sec:ContactHamShell}.  It is straightforward to check that we have the following lemma.
\begin{lemma}\label{l:abstract=S1}
	The contact shell $(B_{K_+ \#_\beta K_-}, \eta_{K_+ \#_\beta K_-})$ 
	is equivalent to the abstract connected sum $(B_{K_+} \#_T B_{K_-}, \eta_{K_+} \#_T \eta_{K_-})$
	with tube
	$$
		T = \{u \leq e^{-\beta(z)},\, v \leq e^{-\beta(z)}E(u)\} \subset \R^{2n+1}
	$$
	where the connected sum is done at the north pole of $B_{K_+}$ and the south pole of
	$B_{K_-}$.
\end{lemma}

\subsubsection{Ambient boundary connected sum}
	Suppose in an almost contact manifold $(W^{2n+1}, \xi)$ there are disjoint codimension $0$ submanifolds 
	$W_\pm \subset \Int W$ with piecewise smooth boundary such that $\xi$ is a genuine contact structure in $\Op \p W_\pm$.
	Assume the contact shells $(W_\pm, \xi)$ are equipped with gluing discs
	$\iota_\pm: D \to \p W_\pm$
	where $\iota_{\pm}^* \alpha = \theta$ for a contact form $\alpha$ for $\xi$ such that
	$\iota_+$ preserves and $\iota_-$ reverses orientation.  
	
	For a smooth embedding	
		$\gamma: [0,1] \to \Int W$
		such that 
		\begin{itemize}
		\item   $\gamma(0) = \iota_+(0)$, $\gamma(1) = \iota_-(0)$, 
		and  $\gamma(t) \notin W_+ \cup W_-$ otherwise;
	\item $\xi$ is a genuine contact structure on $\Op \Gamma$ where $\Gamma:=\gamma([0,1])$;
	\item $\gamma$ is transverse to $\xi$,
	\end{itemize}
	we can think of $(W_+ \cup \Op \Gamma \cup W_-, \xi)$ as an ambient boundary connected sum of
	the shells $(W_\pm, \xi)$.  This is made precise with the following lemma.
	
	\begin{lemma}\label{l:Ambient=Abstract}
		Every neighborhood $(W_+ \cup \Op \Gamma \cup W_-, \xi)$ contains the image of
		an almost contact embedding of an abstract connected sum $(W_+ \#_{T} W_-, \xi \#_{T} \xi)$.
\end{lemma}
\begin{proof}
The gluing discs $\iota_\pm:D \to\p W_\pm$ extend to Darboux embeddings  
 \begin{equation*} 
 \Phi_\pm: D \times (\mp\ell-\eps, \mp\ell+\eps) \to \Op \iota_\pm (D) \quad\mbox{with } \;\Phi_\pm^* \alpha = dz + \theta
 \mbox{ and } \Phi_\pm|_{D \times \mp \ell} = \iota_\pm
\end{equation*}
and moreover one can ensure $\Phi_+^{-1} (\Gamma)=0 \times [-\ell,-\ell+\eps )$ and $\Phi_-^{-1}(\Gamma)=0 \times (\ell-\eps,\ell]$.

By the neighborhood theorem for transverse curves in a contact manifold, for $N > 0$ sufficiently large the embeddings
$\Phi_\pm$ can be extended (after possibly decreasing $\eps$)  to a contact embedding  
$$
	\Phi: \Big(D \times (-\ell-\eps, \ell+\eps)\Big)\cup \Big(L^{-N}(D) \times [-\ell,\ell] \Big)\cup 
	\Big(D \times (\ell-\eps, \ell+\eps)\Big) \to \Int W,
$$ 
whose image is contained in $\Op(\iota_+(D) \cup \Gamma \cup \iota_-(D))$ and such that $\Phi(0 \times [-\ell,\ell])=\Gamma$.
Picking $\beta:[-\ell,\ell]\to\R_{\geq 0}$ such that the tube
 $$T=  \{(p,z) \in [-\ell, \ell] \times \R^{2n} : p \in L^{-\beta(z)}D \} $$
is contained in the domain of $\Phi$, we can now use $\Phi$ to define the required contact embedding
 $(W_+ \#_{T} W_-, \xi \#_{T} \xi) \to (W,\xi)$.
 \end{proof}


\subsection{Filling a connected sum of a shell with a neighborhood of an overtwisted disc}

For the rest of this section we will let $(K, \Delta)$ be a special Hamiltonian associated to a special function 
$k: \R_{\geq 0} \to \R$ such that 
\begin{equation}\label{e:special3}
\mbox{$K(u,z) = k(u)$ when $z \in \Op [z_D', z_D]$}
\end{equation}
where $z_D' \in (-1, z_D)$ is sufficiently close to $z_D$
and recall that $E(u) := K(u, \pm 1)$ is well-defined and satisfies $K \leq E$.  For $\eps' > 0$, define
$$
	K' = K - \eps' \quad\mbox{and}\quad \Delta' = \{u \leq 1-\eps',\, \abs{z} \leq 1-\eps'\}
$$
and assume $\eps' > 0$ is small enough so that $K'|_{\p \Delta'} > 0$.

The goal of this subsection is  the proof of Proposition~\ref{p:ConnectSumDisc}
and its parametric version Proposition~\ref{p:ConnectSumDisc-param}.
All the connected sums as in Sections~\ref{s:AbstractBCS} and \ref{s:S1ConnectSum} will be done
with a fixed choice of function $\beta: [-\ell, \ell] \to \R_{\geq 0}$, which we will suppress from the notation.  In particular
we will be considering abstract connected sums such as
$$
	(B_{K} \# B_{K}, \eta_{K} \# \eta_{K}) \quad\mbox{and}\quad
	(B_{K\#K}, \eta_{K \# K})
$$
where we will always use the north pole gluing place on the first factor and the south pole gluing place on the second factor.
We will also freely use Lemma~\ref{l:abstract=S1} to identify such connect sums.

By Lemma~\ref{l:HamiltonianShellsAnnulus} we can arrange the inclusion 
$$(B_{K'\!,C}, \eta_{K'\!, \rho}) \hookrightarrow (B_{K,C}, \eta_{K, \rho})$$   to be  a subordination map,
so that we have a $(2n+1)$-dimensional contact annulus
	$$
		(\A, \xi_{\A}) := (B_{K,C} \setminus \Int B_{K'\!,C},\, \ker \eta_{K, \rho}|_{\A}).
	$$
Define the contact ball $(\B, \xi_\B) \subset (\A, \xi_\A)$ given by
	\begin{equation}\label{e:ModelOtBall}
		 \B:= \{(x,v,t) \in \A : z(x) \in [-1, z_D]\}
	\end{equation}
and by design the $2n$-dimensional disc $(D_K, \eta_K) \subset (\p\B, \xi_\B)$ appears with the correct coorientation.	
\begin{figure}[h]
   \centering
   \def\svgwidth{450pt}
   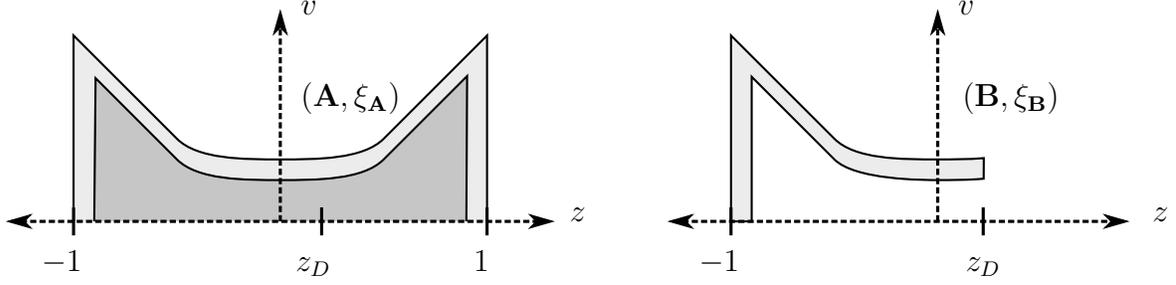
   \caption{On the left: The union of the grey regions is $B_K$,
   the dark grey region is $B_{K'}$, and the light grey region is
   $(\A, \xi_{\A})$.
   On the right: The contact ball $(\B, \xi_{\B}) \subset B_K$
   obtained from $(\A, \xi_{\A})$.}
  \label{f:embedHamiltonian2}
\end{figure}

\subsubsection{Non-parametric version}\label{sec:filling}

To prove Proposition~\ref{p:ConnectSumDisc} it 
will suffice to show the contact shell $(B_K \# \B, \eta_K \# \xi_\B)$, defined as a subset of
$(B_K \# B_K, \eta_K \# \eta_K)$, is equivalent to a genuine contact structure.
Denoting by $\iota\colon \Delta \to \Delta \# \Delta$ the inclusion into the right hand factor,
we will prove in Lemma~\ref{p:ContEmb}(i) below that
there is a family of contact embeddings 
\begin{equation}\label{e:Theta}
	\Theta_\sigma: \Delta \to \Int(\Delta \# \Delta) \quad\mbox{for $\sigma \in [0,1]$ with $\Theta_0 = \iota$}
\end{equation}
such that $\Theta_\sigma = \iota$ in $\Op \{z \in [z_D, 1]\}$ for all $\sigma \in [0,1]$
and $\Theta := \Theta_1$ satisfies 
$$(\Theta_*K' , \Theta(\Delta')) < (K \# K, \Delta \# \Delta).$$

\begin{figure}[h]
   \centering
   \def\svgwidth{300pt}
   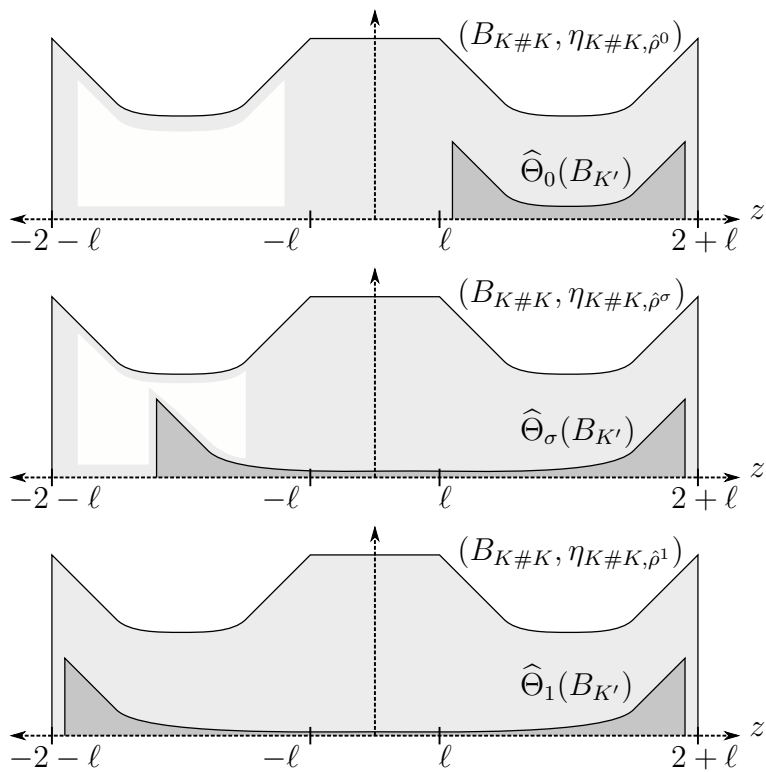
   \caption{Images of the almost contact embeddings $\wh{\Theta}_\sigma : B_{K'} \to B_{K\#K}$ in dark grey.
   The white regions denote where outside of $\wh{\Theta}_\sigma(B_{K'})$ the almost contact structure
   $\eta_{K\#K,\hat{\rho}^\sigma}$ is not genuine.}
  \label{f:embedHamiltonian3}
\end{figure} 

\begin{proof}[Proof of Proposition~\ref{p:ConnectSumDisc}]
It suffices to prove $(B_K \# \B, \eta_K \# \xi_\B)$ is equivalent to a genuine contact structure,
since it is dominated by $(B_{K_0} \# B, \eta_{K_0} \# \xi)$ if we pick $\eps' > 0$ sufficiently small in the definition of $(\B, \xi_\B)$.

By Lemmas~\ref{l:HamiltonianShellsAnnulus}, \ref{l:PhiToTildePhi}, and \ref{lem:contain} we can pick
a family of contact shell structures on $(B_{K} \# B_{K}, \eta_{K\#K,\hat\rho^{\sigma}})$ 
such that there is   a family of contact shell embeddings
\begin{equation}\label{e:hatTheta}
	\wh\Theta_\sigma: (B_{K'}, \eta_{K'}) \to 
	(B_{K} \# B_{K}, \eta_{K\#K,\hat\rho_{\sigma}})
\end{equation}
with  $\wh\Theta_1$ a subordination map.  
We can arrange that $\eta_{K\#K, \hat\rho_0} = \eta_K \# \eta_K$ and 
for all $\sigma \in [0,1]$ to have
$$
\mbox{$\eta_{K\#K, \hat\rho_{\sigma}} = \eta_{K} \# \eta_{K}$ \quad on \quad 
$\Op \hat\iota\{z \in [z_D, 1]\}$}
$$
where 
$\hat\iota: B_{K} \to B_K \# B_K$ is the inclusion into the right hand factor.

We can pick an isotopy $\{\Psi_{\sigma}\}_{\sigma \in [0,1]}$ of
$B_K \# B_K$ based at the identity and supported away from the boundary such that
\begin{enumerate}
\item $\Psi_\sigma \circ \hat\iota = \wh\Theta_\sigma: B_{K'} \to B_K \# B_K$\,,
\item $\Psi_\sigma = \Id$ on $\Op \hat\iota\{z \in [z_D, 1]\}$\,, \mbox{ and}
\item $\Psi_1(B_K \# \A) = (B_K \# B_K) \setminus \Int \wh\Theta(B_{K'})$.
\end{enumerate}
Observe that  a point 
in $\Op \partial(B_K \# \B)$ is one of the following regions
\begin{enumerate}
\item $\Op \partial(B_K \# B_K)$ where $\Psi_\sigma = \Id$ and $\eta_K \# \eta_\B = \eta_K \# \eta_K = \eta_{K\#K,\hat\rho_{\sigma}}$
\item $\Op \hat\iota(\{z = z_D\})$ where $\Psi_\sigma = \Id$ and $\eta_K \# \eta_\B = \eta_K \# \eta_K = \eta_{K\#K,\hat\rho_{\sigma}}$
\item $\Op \hat\iota(\partial B_{K'})$ where $\eta_{K}\#\eta_\B = \hat\iota_*\eta_{K'} 
= \hat\iota_* \wh\Theta_{\sigma}^*(\eta_{K\#K,\hat\rho^{\sigma}})
= \Psi_{\sigma}^*(\eta_{K\#K,\hat\rho_{\sigma}})$
\end{enumerate}
This shows $\xi_{\sigma} := \Psi_{\sigma}^*(\eta_{K\#K,\hat\rho_{\sigma}})$ is a family of equivalent contact shells on $B_K \# \B$
with $\xi_0 = \eta_K \# \eta_\B$.
We know $\eta_{K\#K,\hat\rho_1}$ is a contact structure away from
$\Int \wh\Theta(B_{K'})$, since $\wh\Theta_1$ is a subordination map, and therefore $\xi_1$ is a genuine contact
structure on $B_K \# \B$.
\end{proof}



\subsubsection{Parametric version}\label{sec:final-filling}
Recall the family on Hamiltonians $K^{(s)}=sK +(1-s)E$ for $s\in[0,1]$ from \eqref{e:ParametricHam}.

Let $(\B, \xi)$ be the contact ball from \eqref{e:ModelOtBall} and
let $({}^I B_K \# {}^I \B, {}^I \zeta)$ be the family of contact shells fibered over $I = [0,1]$ 
with fiber $(B_{K^{(s)}} \# \B, \zeta^s)$ over $s \in [0,1]$ where
$$ 
	 \zeta^s = \eta_{K^{(s)}} \# \xi_\B \quad\mbox{ and }\quad
	 (B_{K^{(s)}}\# \B, \eta_{K^{(s)}} \# \xi) \subset (B_{K^{(s)}}\# B_{K}, \eta_{K^{(s)}} \# \eta_K).
$$ 
We may assume  that $\zeta^s$ is a genuine contact structure when $s \in \Op\{0\}$ since 
$K^{(0)} = E$ is positive.

Let us first prove the following proposition similar to Proposition~\ref{p:ConnectSumDisc-param}.

\begin{prop}\label{p:LocalConnectSumDisc-param}
The fibered family of contact shells ${}^I \zeta$ is homotopic relative to
$$
	\Op\{s = 0\} \cup \bigcup_{s \in [0,1]} \Op \partial(B_{K^{(s)}}\# \B) \subset {}^I B_{K}\# {}^I \B
$$
through fibered families of contact shells on ${}^I B_{K}\# {}^I \B$ to a fibered family of genuine contact structures.
\end{prop}
\begin{proof}
	Inspecting the proof of Proposition~\ref{p:ConnectSumDisc}
	shows that it can be done parametrically.  In particular we can get a family of contact shell embeddings
	$$
		\wh\Theta_\sigma^s: (B_{K'}, \eta_{K'}) \to (B_{K^{(s)}}\#B_K, \eta_{K^{(s)}\#K, \hat\rho_\sigma^s})
	$$
	and associated isotopies $\{\Psi_\sigma^s\}_{\sigma \in [0,1]}$ of $B_{K^{(s)}} \# B_K$, which lead to
	contact shell structures
	$$
		\hat\zeta_\sigma^s := (\Psi_\sigma^s)^*(\eta_{K^{(s)}\#K, \hat\rho^\sigma_s}) \mbox{ on $B_{K^{(s)}} \# \B$}
	$$
	that define a family of fibered contact shells ${}^I \hat\zeta_\sigma$ on ${}^I B_{K}\#{}^I \B$.
	It follows from the second part of Lemma~\ref{p:ContEmb} that
	\begin{equation}\label{e:Embed2}
		\Big((\Theta_\sigma)_*K', \Theta_\sigma(\Delta')\Big) < (K^{(s)}\#K, \Delta \# \Delta) \quad\mbox{if $s \in \Op\{0\}$,}
	\end{equation}
	and therefore we can arrange for $\wh\Theta_\sigma^s$ to be a subordination map when $s \in \Op\{0\}$.
	
	With this set-up the proof of Proposition~\ref{p:ConnectSumDisc} shows we can ensure
	the family of fibered contact shells ${}^I \hat\zeta_\sigma$ is such that ${}^I \hat\zeta_0 = {}^I \zeta$ as well as
	\begin{enumerate}
	\item $\hat\zeta_\sigma^s = \zeta^s$ on $\Op \partial(B_{K^{(s)}}\# \B)$ for all $s$ and $\sigma$,
	\item $\hat\zeta_1^s$ is a genuine contact structure for all $s$, and
	\item $\hat\zeta_\sigma^s$ is a genuine contact structure for all $(\sigma, s) \in [0,1] \times [0, 3a]$
	for some $a > 0$.
	\end{enumerate}
	Pick any smooth function
	$$
	f: [0,1] \times [0,1] \to [0,1] \quad\mbox{with}\quad
		f(\sigma, s) = \begin{cases} 
		0 & \mbox{ if $\sigma = 0$}\\
		0 & \mbox{ if $s \in [0, a]$}\\
		1 & \mbox{ if $s \in [2a, 1]$ and $\sigma =1$}
		\end{cases}
	$$
	and define the family of contact shells $\zeta_\sigma^s := \hat\zeta_{f(\sigma, s)}^s$ on $B_{K^{(s)}}\#\B$,
	which represents a homotopy of fibered families of contact shells
	$$
	 \{{}^I \zeta_\sigma\}_{\sigma \in [0,1]}  \quad\mbox{on}\quad {}^I B_K\# {}^I \B.
	$$
	It follows from item (i) and the fact $f(\sigma, s) = 0$ if $s \in [0,a]$, that this homotopy is relative is the appropriate set.
	Observe that $\zeta_1^s$ is a genuine contact structure for all $s \in [0,1]$, since either $s \leq 3a$
	and $\zeta_1^s := \hat\zeta_{f(1,s)}^s$ is a genuine by item (iii), or
	$s \geq 2a$ and $\zeta_1^s :=  \hat\zeta_{f(1,s)}^s = \hat\zeta_{1}^s$ is genuine by item (ii).
	Therefore we have the desired homotopy between ${}^I \zeta = {}^I \zeta_0$ and a fibered family
	of genuine contact structures ${}^I \zeta_1$.
\end{proof}

\begin{proof}[Proof of Proposition \ref{p:ConnectSumDisc-param}]
	Recall that $({}^TB_{K_0} \# {}^TB, {}^T\eta_{K_0} \# {}^T\xi)$ is the fibered contact shell, which at the point
	$\tau \in T = D^q$ is given by
	$$
		(B_{K_{0}^{(\delta(\tau))}} \# B, \eta_{K_{0}^{(\delta(\tau))}} \# \xi)
	$$	
	where $\delta: T \to [0,1]$ is a bump function that vanishes near the boundary.
	It suffices to prove $({}^TB_K \# {}^T\B, {}^T\eta_{K} \# {}^T\xi_\B)$ 
	is fibered equivalent to a fibered contact structure over $T$,
	since it is dominated by $({}^TB_{K_0} \# {}^TB, {}^T\eta_{K_0} \# {}^T\xi)$
	if we pick $\eps' > 0$ sufficiently small in the definition of $(\B, \xi_\B)$.
	
	In the notation of Proposition~\ref{p:LocalConnectSumDisc-param}
	we have the identification
	$$
		\zeta^{\delta(\tau)} = \eta_{K^{(\delta(\tau))}} \# \xi_\B 
		\quad\mbox{as contact shell structures on $B_{K^{(\delta(\tau))}} \# \B$}
	$$
	and a fibered contact structure on $({}^TB_K \# {}^T\B, {}^T\zeta_1)$ with contact structure
	$$
		\zeta^{\delta(\tau)}_1 \mbox{ on the fiber } B_{K^{(\delta(\tau))}} \# \B.
	$$
	Since $\delta(\tau) = 0$ if $\tau \in \Op \partial T$, the homotopy constructed in Proposition~\ref{p:LocalConnectSumDisc-param} 
	when used fiberwise gives a homotopy between ${}^T\eta_{K} \# {}^T\xi_\B$ and ${}^T\zeta_1$ that shows they are
	fibered equivalent.
\end{proof}


\subsection{Main lemma}\label{sec:OvertwistDisc}


		\begin{figure}[h]
   \centering
   \def\svgwidth{450pt}
   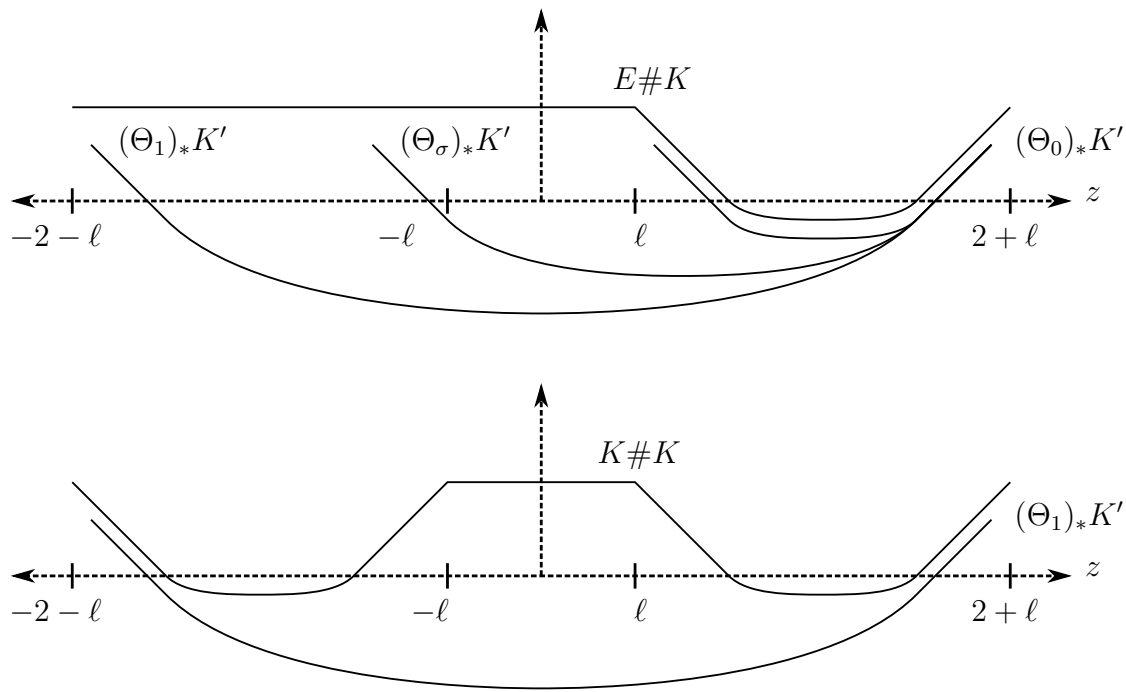
   \caption{Schematic representative of Lemma~\ref{p:ContEmb} where 
   $\Delta = \{\abs{z} \leq 1\} \subset \R$.
   Here $K'(z) = K(z) - \eps'$ is restricted to $\Delta' = \{\abs{z} \leq 1-\eps'\}$}
  \label{f:embedHamiltonian}
\end{figure} 
  
Consider the connected sums $(K\#_\beta K, \Delta\#_{\beta, \ell}\Delta)$ and
$(E\#_\beta K, \Delta\#_{\beta, \ell}\Delta)$ as in Section~\ref{s:S1ConnectSum}.
The main goal of this section will be to prove the following lemma, which we will break up into two sublemmas below.
  
\begin{lemma}\label{p:ContEmb}
	There is a family of contact embeddings for $\sigma\in[0,1]$
	$$
		\Theta_\sigma: \Delta \to \Delta\#_{\beta, \ell}\Delta \quad\mbox{with
		$\Theta_\sigma = Z_{1+\ell}$ on $\Op\{z \in [z_D, 1]\}$}
	$$
	based at $\Theta_0 := Z_{1+\ell}$ such that
	\begin{enumerate}
	\item $\left((\Theta_1)_*K' ,\Theta_1(\Delta')\right) < \left(K\#_\beta K, \Delta\#_{\beta, \ell}\Delta\right)$ and
	\item $\left((\Theta_\sigma)_*K',\Theta_\sigma(\Delta')\right) < \left(E\#_\beta K, \Delta\#_{\beta, \ell}\Delta\right)$ for all $\sigma \in [0,1]$.
	\end{enumerate}
\end{lemma}
\begin{proof}
	It follows from Lemma~\ref{l:ChangeNeck} that it suffices to prove this lemma when $\beta \equiv 0$
	and this special case is proved in Lemma~\ref{p:ContEmb1}.
\end{proof}

Let us remark that Section~\ref{sec:filling} only used the first part of Lemma~\ref{p:ContEmb}, while 
Section~\ref{sec:final-filling} used the both parts.

\begin{remark} In the $3$-dimensional case, where $\Delta = [-1,1]$,
this lemma essentially follows from Lemma \ref{lm:3-domin}.
\end{remark}

\subsubsection{Transverse scaling and simplifying the neck region}

\textbf{Transverse scaling.}
	A orientation preserving diffeomorphism $h: \R \to \R$ defines a contactomorphism
	$\Phi_h$ of $(\R^{2n-1}, \xi_\st)$ by
	$$
		\Phi_h(u_i, \varphi_i, z) = (h'(z)u_i, \varphi_i, h(z)) 
	$$
	where $\Phi_{h}^{-1} = \Phi_{h^{-1}}$.  
	By \eqref{e:ContPushK} we have 
	\begin{equation}\label{e:PushS}
		(\Phi_{h})_{*}H(u,z) = h'(h^{-1}(z))\, H(\tfrac{u}{h'(h^{-1}(z))}, h^{-1}(z))
	\end{equation}
	for a contact Hamiltonian $H(u,z) : \R^{2n-1} \to \R$.

\begin{ex}\label{ex:Scaledomain}{\rm
For our purposes $\Phi$ should be thought of as a way to manipulate the $z$-variable at the cost of a scaling factor on
the $u$-variable, in particular we have a contactomorphism between domains in $(\R^{2n-1}, \xi_\st)$
$$
	\Phi_h: \{ u \leq f(z),\, z \in [a,b]\} \to \{u \leq (h' \cdot f)(h^{-1}(z)),\, z \in [h(a), h(b)]\}
$$
where $f: \R \to \R_{>0}$.}
\end{ex}

This contactomorphism allows us to reduce the proof of Lemma~\ref{p:ContEmb} to when $\beta \equiv 0$.
\begin{lemma}\label{l:ChangeNeck}
	For every connected sum $(K \#_{\beta} K, \Delta \#_{\beta, \ell} \Delta)$, if $\ell' > \ell$ is sufficiently
	large, then there is a a contact embedding
	$$
		\Phi: \Delta\#_{\ell'}\Delta \to \Delta\#_{\beta, \ell}\Delta \quad\mbox{with}\quad
		\Phi = Z_{\pm(\ell-\ell')} \mbox{ on } \Op \{\pm z \geq \ell'\}
	$$
	such that 
	$(\Phi_*(K \# K), \Phi(\Delta\#_{\ell'}\Delta)) \leq (K \#_{\beta} K, \Delta \#_{\beta, \ell} \Delta).$
\end{lemma}
\begin{proof}
	Recall $E(u) := K(u, \pm 1) > 0$ with $K \leq E$ and pick a constant
	\begin{equation}\label{e:CE}
		0 < C < \frac{\min(E)}{\max(E)} \leq 1.
	\end{equation}
	Pick a diffeomorphism $h:[-\ell', \ell'] \to [-\ell, \ell]$ with $h'(z) = 1$ on $z \in \Op\{\pm \ell\}$
	and
	\begin{equation}\label{e:hprime}
		h'(h^{-1}(z)) \leq C e^{-\beta(z)}
	\end{equation}
	which is possible provided
	$$
		\ell' > \frac{1}{2C} \int_{-\ell}^{\ell} e^{\beta(z)}dz
	$$
	Extend $h$ by translation to get a diffeomorphism $h: \R \to \R$ and consider the associated contactomorphism 
	$\Phi_h: (\R^{2n-1}, \xi) \to (\R^{2n-1}, \xi)$ from \eqref{e:PushS}.  This is the desired contact embedding
	for by \eqref{e:hprime} we have
	$$
		\Phi_h(\Delta\#_{\ell'}\Delta) = \{u \leq h'(h^{-1}(z))\,,\,\, z \in [-2-\ell, 2+\ell]\} \subset \Delta\#_{\beta,\ell}\Delta\,.
	$$
	To check the order on the Hamiltonians, it suffices to check on $\Phi_h(T_{\ell'})$ where we have
	$$
		(\Phi_h)_*E(u,z) = h'(h^{-1}(z)) E(\tfrac{u}{h'(h^{-1}(z))})
		 < e^{-\beta(z)} E(u) = (K\#_{\beta}K)(u,z)
	$$
	by \eqref{e:CE} and \eqref{e:hprime}.
\end{proof}

\subsubsection{The twist contactomorphism and a special case of Lemma~\ref{p:ContEmb}}

We will use the transverse scaling contactomorphisms $\Phi_h$ together with the following contactomorphism.

\textbf{Twist contactomorphism.}
	For $g \in C^\infty(\R)$ and $z_0 \in \R$,
	define
	$$
		\Psi_{g,z_0}(u_i, \varphi_i, z) := \left(\frac{u_i}{1+g(z)u},\,\varphi_i - \int_{z_0}^z g(s)\,ds,\, z \right)
	$$
	which is a contactomorphism between the subsets of $(\R^{2n-1}, \xi_\st)$
	$$
		\Psi_{g, z_0}: \{1+g(z)u > 0\} \to \{1-g(z)u > 0\}
	$$
	where $\Psi_{g, z_0}^{-1} = \Psi_{-g, z_0}$.
	By \eqref{e:ContPushK} we have
	\begin{equation}\label{e:PushT}
		(\Psi_{g, z_0})_{*}H(u,z) = (1-g(z)u)\, H(\tfrac{u}{1-g(z)u}, z)
	\end{equation}
	for a contact Hamiltonian $H(u,z) : \R^{2n-1} \to \R$.
	
\begin{ex}\label{ex:domain}{\rm
For our purposes $\Psi$ should be thought of as a way to manipulate the $u$-variable at the cost of a rotation
in the angular coordinates, in particular we have a contactomorphism between domains in $(\R^{2n-1}, \xi_\st)$
$$
	\Psi_{g}: \{u \leq f_2(z)\} \to \{u \leq f_1(z)\}
$$	
where $f_j: \R \to \R_{>0}$ and $g(z) =  \tfrac{1}{f_1(z)} - \tfrac{1}{f_2(z)}$.}
\end{ex}

\textbf{Composing twist and scaling.}
Fix an orientation preserving diffeomorphism $h:\R \to \R$ and define $g(z) := 1 - \frac{1}{h'(h^{-1}(z))}$.
It follows from Examples~\ref{ex:Scaledomain} and \ref{ex:domain}
\begin{equation}\label{e:Gamma}
	\Gamma_{h, z_0}:= \Psi_{g,z_0} \circ \Phi_{h} : \{u \leq 1,\, z \in [a,b]\} \to \{u \leq 1,\, z \in [h(a), h(b)]\}
\end{equation}
is a contactomorphism of these domains in $(\R^{2n-1}, \xi_{\st})$.  So $\Gamma_{h,z_0}$ lets us change the $z$-length of a region 
without changing the $u$-width, albeit still at the cost of a rotation in the angular coordinates.

Computing shows that 
$$
	\Gamma_{h, z_0}(u_i, \varphi_i, z) = \left(\frac{h'(z)u_i}{1+(h'(z) -1)u}\,,\,\, \varphi_i - \int_{z_0}^z \Big(1 - \tfrac{1}{h'(h^{-1}(z))}\Big)\,ds\,,\,\, h(z)\right)
$$
so if $h(z) = z + \tau$ on $z \in A \subset \R$ and $z_0 \in h(A)$, then $\Gamma_{h, z_0}$ is just translation
\begin{equation}\label{e:GammaZ}
	\Gamma_{h, z_0} = Z_{\tau} \quad\mbox{on}\quad \{z \in A\} \subset \R^{2n-1}.
\end{equation}
If we define
\begin{equation}\label{e:PushTh}
	\tilde{h}(u,z) := h'(h^{-1}(z)) - \big(h'(h^{-1}(z)) - 1\big)u
\end{equation}
then for a contact Hamiltonian $H(u,z): \R^{2n-1} \to \R$ we have that
\begin{equation}\label{e:PushH}
	(\Gamma_{h,z_0})_*H(u,z) = \tilde{h}(u,z)\,\, H\!\left(\frac{u}{\tilde{h}(u,z)}, h^{-1}(z)\right).
\end{equation}

\begin{figure}[h]
   \centering
   \def\svgwidth{450pt}
   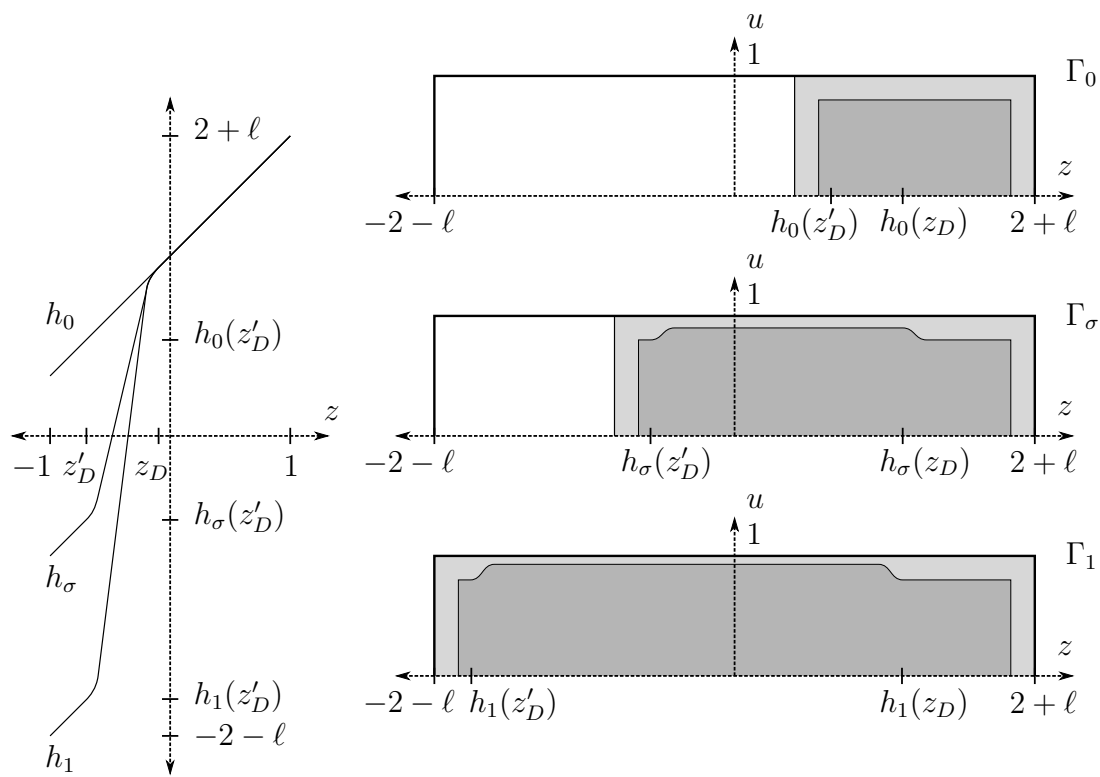
   \caption{The family of diffeomorphisms $h_\sigma$ and embeddings $\Gamma_\sigma: \Delta \to \Delta\#_\ell \Delta$.  
   The union of the grey regions denote the image $\Gamma_\sigma(\Delta)$ while
   the dark grey regions denote the image $\Gamma_\sigma(\Delta')$ for $\Delta' = \{\abs{z} \leq 1-\eps',\, u \leq 1-\eps'\}$.}
  \label{f:graphH}
\end{figure}

\textbf{Proving Lemma~\ref{p:ContEmb} when $\beta \equiv 0$}

Pick a family of diffeomorphisms $h_\sigma: \R \to \R$ for $\sigma \in [0,1]$ such that
	\begin{equation}\label{e:defh}
		h_\sigma(z) = \begin{cases}
			z+(1-2\sigma)(1+\ell) & \mbox{ for $z \in \Op (-\infty, z_D']$}\\
			h_\sigma'(z) \geq 1 & \mbox{ for $z \in [z_D', z_D]$}\\
			z+1+\ell & \mbox{ for $z \in \Op [z_D, \infty) $}
			\end{cases}
	\end{equation}	
	Recall the contactomorphism $\Gamma_{h, z_0}$ from \eqref{e:Gamma} and define the contact embeddings
	\begin{equation}\label{e:GammaS}
			\Gamma_\sigma := \Gamma_{h_\sigma, 2+\ell}: \Delta \to \Delta\#_{\ell}\Delta
			\quad\mbox{for $s \in [0,1]$.}
	\end{equation}
	By \eqref{e:GammaZ}, we see $\Gamma_0 = Z_{1+\ell}$ and on $\Op\{z \in [z_D, 1]\}$
	we have $\Gamma_\sigma = Z_{1+\ell}$ for all $\sigma \in [0,1]$.	
With this family of contactomorphisms we can prove Lemma~\ref{p:ContEmb} with the simplifying assumption that $\beta = 0$.

\begin{lemma}\label{p:ContEmb1}
	The family of contact embeddings $\Gamma_\sigma: \Delta \to \Delta \#_{\ell} \Delta$ for $\sigma \in [0,1]$
	satisfy:
	\begin{enumerate}
	\item $(\Gamma_\sigma)_*K \leq E\# K$ on $\Gamma_\sigma(\Delta)$ for all $\sigma \in [0,1]$, and
	\item $(\Gamma_1)_*K \leq K\# K$ on $\Gamma_1(\Delta)$.
	\end{enumerate}
\end{lemma}

\begin{proof}[Proof of Lemma~\ref{p:ContEmb1}]
By \eqref{e:PushH} we have
$$
	(\Gamma_\sigma)_*K(u,z) = \tilde{h}_\sigma(u,z)\,\, K\!\left(\frac{u}{\tilde{h}_\sigma(u,z)}, h^{-1}_\sigma(z)\right),
$$
where recall from \eqref{e:PushTh} that
$$\tilde{h}_\sigma(u,z):= h_\sigma'(h_\sigma^{-1}(z)) - \big(h_\sigma'(h_\sigma^{-1}(z)) - 1\big)u \geq 1$$ 
where the inequality follows from $h'_\sigma(h^{-1}_\sigma(z)) \geq 1$ and $u \leq 1$.
Bringing these two together, we have
\begin{equation}\label{e:GsK}
	(\Gamma_\sigma)_*K(u,z) = \begin{cases} K(u,h_\sigma^{-1}(z)) & \mbox{ if } z \in \Op h_\sigma([-1,\, z_D'])\\
	\leq k(u) &  \mbox{ if } z \in h_\sigma([z_D',\, z_D])\\ 
	K(u,z -(1+\ell)) & \mbox{ if } z \in \Op [z_D+1+\ell,\, 2+\ell]
	\end{cases}
\end{equation}
using that $h_\sigma$ is translations on the ends, while on the middle we have
\begin{align*}
	(\Gamma_\sigma)_*K(u,z) &= (\Gamma_\sigma)_*k(u,z) = \tilde{h}_\sigma(u,z)\,\, k\!\left(\frac{u}{\tilde{h}_\sigma(u,z)}\right) \leq
	k(u) 
\end{align*}
using \eqref{e:special3} and \eqref{e:subscale}.  

To verify $(i)$, since
$$
	(E\# K)(u,z) = \begin{cases} E(u) & \mbox{ if } z \in [-2-\ell,\, \ell]\\
	K(u,z-(1+\ell)) & \mbox{ if } z \in [\ell,\, 2+\ell]
	\end{cases}
$$
it follows from \eqref{e:GsK} and the inequalities
\begin{equation}\label{e:INEM}
	k(u) \leq K(u,z) \leq E(u).
\end{equation}
that it suffices to check
$$
	K(u, h_\sigma^{-1}(z)) \leq K(u, z-(1+\ell)) \quad\mbox{when}\quad z \in [\ell,\, h_\sigma(z_D')].
$$
Since $h_\sigma^{-1}(z) = z-(1-2\sigma)(1+\ell)$ here, this is equivalent to
$$
	K(u, z+2\sigma(1+\ell)) \leq K(u, z) \quad\mbox{when}\quad z \in [-1,\, z_D' -2\sigma(1+\ell)]
$$
and this latter condition follows from Definition~\ref{def:SpecialHam}.

To verify $(ii)$, using \eqref{e:defh} we see \eqref{e:GsK} at $\sigma=1$ becomes
$$
	(\Gamma_1)_*K(u,z) = \begin{cases} K(u,z+(1+\ell)) & \mbox{ if } z \in \Op [-2-\ell,\, z_D'-1-\ell]\\
	\leq k(u) &  \mbox{ if } z \in [z_D' -1-\ell,\, z_D + 1+\ell]\\
	K(u,z-(1+\ell)) & \mbox{ if } z \in \Op [z_D+1+\ell,\, 2+\ell]
	\end{cases}
$$
while by definition
$$
	(K\# K)(u,z) = \begin{cases} K(u,z+(1+\ell)) & \mbox{ if } z \in [-2-\ell,\, -\ell]\\
	E(u) &  \mbox{ if } z \in [-\ell,\, \ell]\\
	K(u,z-(1+\ell)) & \mbox{ if } z \in [\ell,\, 2+\ell]
	\end{cases}
$$
so $(ii)$ follows from \eqref{e:INEM}.
\end{proof}




\section{Contact structures with holes}\label{sec:holes}  \label{sec:further}

The goal of this section is   Proposition \ref{prop:to saucers}
and its parametric version   Proposition \ref{prop:to-saucers-param}, which are the first steps in proving
    Propositions \ref{prop:unique-model}
and  Proposition  \ref{prop:unique-model-param}.

\subsection{Semi-contact structures}\label{sec:semi-holonomic}

Let $\Sigma$ be a closed $2n$-dimensional manifold. A {\em semi-contact} structure on an annulus $C=\Sigma\times[a,b]$ is smooth family $\{\zeta_s\}_{s \in [a,b]}$ such that $\zeta_s$ is a germ of a contact structure along the slice $\Sigma_s:=\Sigma\times s$.
If $\{\alpha_s\}_{s \in [a,b]}$ is a smooth family of $1$-forms with $\zeta_s = \ker \alpha_s$ on $\Op \Sigma_s$, then one gets an almost contact structure $(\lambda, \omega)$ on $C$ where
$$
\mbox{$\lambda(x,s) = \alpha_s(x,s)$ and $\omega(x,s) = d\alpha_s(x,s)$.}
$$
It follows that every semi-contact structure on $C$ defines an almost contact structure on $C$ that equals $\zeta_s$ on 
$TC|_{\Sigma_s}$. 

Given a contact structure $\xi$ on $\Sigma\times\R$ and a smooth family of functions $\psi_s:\Sigma\to\R$ for $s\in[a,b]$, if
we pick
$\Psi_s: \Op \Sigma_s \to \Op(\mbox{graph } \psi_s) \subset \Sigma \times \R$ to be a
smooth family of diffeomorphisms such that $\Psi_s|_{\Sigma_s} = \Id \times \psi_s$, then
we can define a semi-contact structure on $\Sigma\times[a,b]$ by 
$\zeta_s := \Psi_s^*\xi$.  Any semi-contact structure of this form will be said to be of {\em immersion type}.

\begin{remark} {\rm The term  is motivated by the fact that on  the boundary of each domain
$\Sigma^{[a',b']}:=\Sigma\times[a',b']$ for $a\leq a'<b'<b$ the structure  $\zeta|_{\p \Sigma^{[a',b']}} $ is induced from the genuine contact structure $\xi$ by an immersion 
$\p \Sigma^{[a',b']}\to \Sigma\times\R$. Of course, this is an immersion of a very special type, which maps the  boundary components $\Sigma\times a'$ and $\Sigma\times b'$ onto  intersecting graphical hypersurfaces. }
\end{remark}
%
%
  
 \subsection{Saucers}\label{sec:saucers}
  \medskip 
 A {\it saucer} is a domain $B \subset D\times\R$, where $D$ is a $2n$-disc possibly with a piecewise smooth boundary, 
 of the form
 $$B=\{(w,v)\in D\times\R:\; f_-(w)\leq v\leq f_+(w)\}$$
 where $f_\pm:D\to\R$ are smooth functions such that
 $f_-<f_+$ on $\Int D$ and whose $\infty$-jets coincide along $\p D$.
Observe that every saucer comes with a family of discs 
$$
D_s=\{(w,v) \in D \times \R:\; v=(1-s)f_-(w)+s f_+(w)\}
\quad\mbox{for $s \in [0,1]$}
$$ 
such that the interiors $\Int D_s$ foliate $\Int B$ and
the family of discs $D_s$ coincide with their $\infty$-jets along their common boundary $S=\p D_s$, which is called the {\em border} of the saucer $B$.
 
A \emph{semi-contact} structure on a saucer $B$ is a family $\{\zeta_s\}_{s \in [0,1]}$ of germs of contact structures along the
discs $D_s$ for $s\in[0,1]$, which coincide as germs along the border $S$.  
As in Section~\ref{sec:semi-holonomic} a semi-contact structure on 
a saucer $B$ defines an almost contact structure $\xi$ on $B$.  Furthermore $(B, \xi)$ is a contact shell since 
$\zeta_0$ and $\zeta_1$ are germs of contact structures on $D_0$ and $D_1$ and the family $\zeta_s$ 
coincide along the border of $B$.

\subsection{Regular semi-contact saucers}\label{sec:Regsaucers}

In $(\R^{2n+1},\xi^{2n+1}_\st= \{\lambda_\st^{2n-1} + vdt=0\})$ where $v := -y_n$ and $t:= x_n$, 
define the hyperplane $\Pi:=\{v=0\}$. Observe the characteristic foliation on $\Pi \subset (\R^{2n+1},\xi^{2n+1}_\st)$ is formed by the fibers of 
the projection $\pi: \Pi\to\R^{2n-1}$ given by $\pi(x, t)=x$ for $x\in\R^{2n-1}$.

\begin{figure}[h]
\begin{center}
\includegraphics[scale=.5]{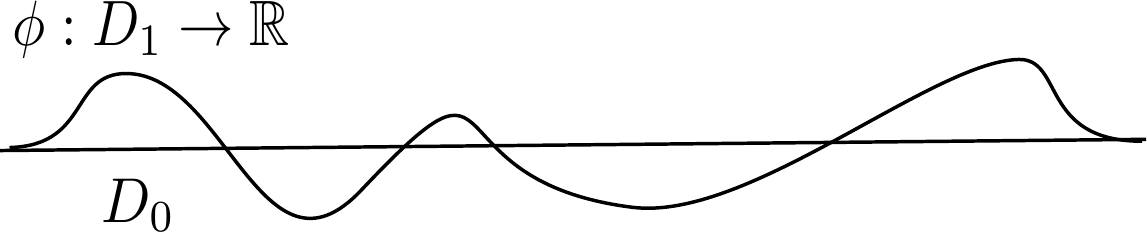}
\caption{A typical regular contact saucer.}\label{fig: saucer}
\end{center}\end{figure}

Let $D\subset\Pi$ be a $2n$-disc and let $\phi:D\to\R$ be a smooth function such that
$\phi>0$ on $\Int D\cap\Op(\p D)$ and whose $\infty$-jet vanishes on $\p D$. Let $F: D \to \R$ be a function, compactly supported in $\Int D$, so that $\phi + F$ is positive on $\Int D$. Define the saucer   $B := \{(w,v) \in D \times \R : w \in D,\; 0\leq v \leq \phi(w) + F(w)\}$.  Note that up to a canonical  diffeomorphism the saucer $B$ is independent of a choice of the function $F$.
There is a natural family of diffeomorphisms between $D_s \subset B$
and the graphs
$$\Gamma_{s\phi}:= \{v = s\phi(w),\;w\in D\} \subset \R^{2n+1}$$
whose $\infty$-jets coincide along the border.

Define $\sigma_\phi = \{\zeta_s\}$ to be the semi-contact structure on $B$ where
$\zeta_s$ is the pullback of the germ of the contact structure on $\Gamma_{s\phi} \subset (\R^{2n+1}, \xi_{\std}^{2n+1})$.
We see that $\phi$ defines the contact shell $(B, \sigma_\phi)$ up to diffeomorphism of the domain.

Parametrize $B$ with coordinates $(w,s) \in D \times [0,1]$, so that $D_{s_0}=\{s=s_0\}$ and consider the map
$$
	\Phi: B \to \R^{2n+1} \quad\mbox{where}\quad \Phi(w,s) = (w, s\phi(w)).
$$
If $\phi$ is positive everywhere on $\Int D$, then $\Phi$ is an embedding,
and hence $\sigma_\phi$ is a genuine contact structure since it can be identified with $\Phi^*\xi_{\st}^{2n+1}$.
Similarly for $2n$-discs $D' \subset D$ and associated semi-contact structures $\sigma_{\phi'}$ and $\sigma_{\phi}$    a  contact shell  $\sigma_{\phi'}$ is dominated by a shell $\sigma_{\phi}$ if $\phi' \leq \phi|_{D'}$ and 
$\phi|_{\Int D \setminus D'} > 0$.

An embedded $2n$-disc $D \subset \Pi$ is called {\em regular} if 
  \begin{itemize}
  \item   the characteristic foliation $\FF$ on $D \subset (\R^{2n+1}, \xi_\st^{2n+1})$ is diffeomorphic to
  the characteristic foliation on the standard round disc in $\Pi$;
    \item the ball $\Delta:= D/\FF$ with its induced contact structure is star-shaped. 
    \end{itemize}
An embedded $2n$-disc $D \subset (M^{2n+1}, \xi)$ in a contact manifold is {\em regular} if the contact
germ of $\xi$ on $D$ is contactomorphic to the contact germ of a regular disc in $\Pi$.
A semi-contact saucer is \emph{regular} if it is equivalent to a semi-contact saucer of the form
$(B, \sigma_\phi)$ defined over a regular $2n$-disc $D \subset \Pi$.

\begin{figure}\begin{center}
\includegraphics[scale=.5]{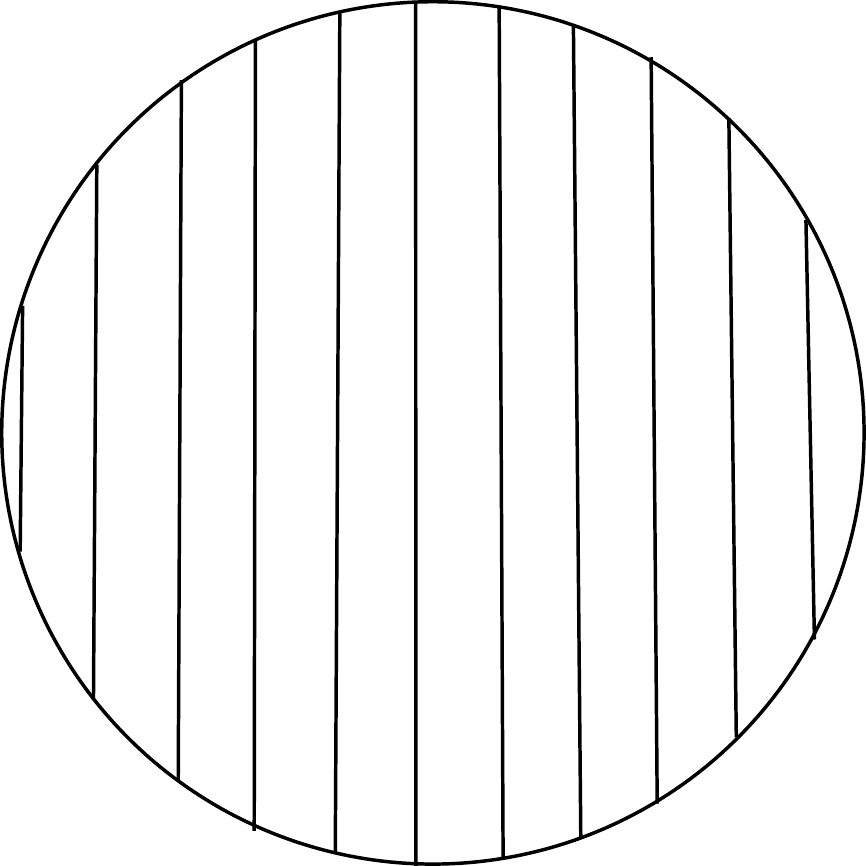}
\caption{A regular foliation on the disc.}\label{fig: regular fol}
\end{center}\end{figure}

In Section~\ref{sec:RedSaucers} we will prove the following proposition.
 \begin{proposition}\label{prop:to saucers}
 Let $M$ be a $(2n+1)$-manifold,  $A\subset M$ a closed subset, and
  $\xi_0$   an almost contact structure on $M$ that is genuine on $\Op A\subset M$. There exists a finite number of embedded saucers $B_i\subset M$ for $i=1, \dots, N$ such that  $\xi_0$ is homotopic relative to $A$ to an almost contact structure $\xi_1$ which is genuine on $M\setminus\bigcup_{i=1}^N B_i$ and whose restriction to each saucer $B_i$ is semi-contact and regular.  
 \end{proposition}

 \subsection{Fibered saucers}\label{sec:fibered-saucer}
    
Slightly stretching the definition of a fibered shell we  will allow $(2n+1)$-dimensional discs $B^\tau$ for $\tau\in \p T$ to degenerate into $2n$-dimensional discs, as in the following definition of fibered saucers.
  A domain ${}^TB \subset T \times D\times\R$   is called a {\it fibered saucer} if $T=D^q$ and
   it has a form
 $${}^TB=\{(\tau, x,v)\in T\times D\times\R: f_-(\tau,x)\leq v\leq f_+(\tau,x)\},$$
 where
  $f_\pm:T\times D\to\R$ are two   $C^\infty$-functions such that
 $f_-(\tau,x)<f_+(\tau,x)$ for all $(\tau,x) \in \Int(T \times D)$ and
 $f_\pm$ coincide along $\p\left(T \times D\right) $ together with their $\infty$-jet.
  Every fibered saucer comes with a family of discs 
 $$
 D^\tau_s=\{(\tau, x,v): x\in D,\; v=(1-s)f_-(\tau,x)+sf_+(\tau, x)\}
 $$
 where for fixed $\tau\in T$ the discs $\{D^\tau_s\}_{s \in [0,1]}$ coincide with their $\infty$-jets along  their common boundary 
 $S^\tau=\p D^\tau_s$. We call the union ${}^TS:=\bigcup_{\tau\in T}S^\tau$    the {\em border} of the  fibered saucer ${}^TB$.
  
 A fibered semi-contact structure ${}^T\xi$ on a fibered saucer  $B$ is a family $\zeta^\tau_s$ of germs of contact structures along discs $D^\tau_s $ for $s\in[0,1]$ and $\tau\in T$, which coincide along the border ${}^TS$. A fibered semi-contact structure defines a fibered almost contact structure on ${}^TB$. In particular, any fibered semi-contact structure on a fibered saucer ${}^TB$ defines a fibered contact shell.

    A fibered semi-contact structure on a fibered saucer $B$ is called 
   {\em regular} if the saucer $(B^\tau,\xi^\tau)$ is regular for each $\tau\in\Int T$.
  More precisely  a fibered    semi-contact saucer  ${}^T\zeta=\left({}^TB, {}^T\xi\right)$ is regular if there exists a regular $2n$-ball   $D\subset\Pi$   and a  $C^\infty$-function  $\phi:{}^TD=\bigcup\limits_{\tau\in T}\{\tau\times D\}\to\R $   such that
   \begin{description}
   \item{-} $\phi$ vanishes with its $\infty$-jet along $\p ({}^TD)$, and
  $\phi >0$ on  $\Op \p({}^T D)\cap\Int {}^T D$;
   \item{-}  for each $s\in[0,1]$ the contact structure $\zeta_s$ is induced by an  embedding  onto a neighborhood of the graph
   $\{y_n=s\phi(\tau,x), \tau\in T, x\in D^\tau\}\subset\R^{2n+1}_\std;$
   \item{-} the disc $D$ is regular.
   \end{description}
    Thus a  fibered regular semi-contact saucer is determined by the function $\phi$,  and we will denote it by ${}^T\sigma_\phi$. 
  



 
\subsection{Interval model}

Proposition~\ref{prop:to saucers} says any contact shell dominates a collection of regular semi-contact saucers.  So the next step
towards proving Proposition~\ref{prop:unique-model} will be to relate regular semi-contact structures and circle model contact shells
and this will be the goal of the remainder of the section.

We will start by introducing one more model contact shell, which we call an {\em interval model}, and it will help us interpolate between regular semi-contact saucers and circle models shells.

Recall that  the standard contact $(\R^{2n-1}, \xi_{\st})$ with $\xi_\st$ is 
given by the contact form 
$$\lambda_\st = dz + \sum_{i=1}^{n-1} u_i \,d\varphi_i\,.$$
In this section the notation $(v,t)$ stands for 
 canonical coordinates on the cotangent  bundle  $T^*I$.

For $\Delta \subset \R^{2n-1}$ a compact star-shaped domain and a contact Hamiltonian
 \begin{equation}\label{e:KBoundaryInterval}
	K: \Delta \times S^1 \to \R \quad\mbox{such that \quad $K|_{\partial \Delta \times S^1} > 0$ and $K|_{\Delta \times \{0\}} > 0$}
\end{equation}
we will build a contact shell structure, similar to the circle model,
on a piecewise smooth $(2n+1)$-dimensional ball 
$$(B^{I}_K, \eta_K^{I}) \subset \Delta \times T^*I$$
which we will refer to as the \emph{interval model contact shell} for $K$.

For any constant $C > - \min(K)$, define the domain
$$
	B_{K,C}^{I} := \{(x,v,t) \in \Delta \times T^*I : 0 \leq v \leq K(x,t)+C\}
$$
which is a piecewise smooth $(2n+1)$-dimensional ball in $\R^{2n-1} \times T^*I$
whose diffeomorphism type is independent of the choice of $C$.
Denote the boundary by
\begin{align*}
	\Sigma_{K,C}^{I} &= \p B_{K,C}^{I} = \Sigma_{0,K,C}^I \cup \Sigma_{1,K,C}^I \cup \Sigma_{2,K,C}^I \quad\mbox{where}\\
	\Sigma_{0,K,C}^I &= \{(x,v,t): v = 0\} \subset \Delta \times T^*I\\
	\Sigma_{1,K,C}^I &= \{(x,v,t): v = K(x,t)+C\} \subset \Delta \times T^*I\\
	\Sigma_{2,K,C}^I &= \{(x,v,t): 0 \leq v \leq K(x,t)+C\,,\,\, (x,t) \in \p(\Delta \times I) \} \subset \Delta \times T^*I\,.
\end{align*}
Now pick a smooth family of functions
\begin{equation}\label{e:FamilyRhoI}
	\rho_{(x,t)}: \R_{\geq 0} \to \R \quad\mbox{for $(x,t) \in \Delta \times I$} \quad\mbox{such that}
\end{equation}
\begin{enumerate}
\item $\rho_{(x,t)}(v) = v$ when $v \in \Op \{0\}$,
\item $\rho_{(x,t)}(v) = v - C$ for $(x,t,v) \in \Op\{v \geq K(x,t) +C\}$, and
\item $\p_v \rho_{(x,t)}(v) > 0$ for $(x,t) \in \Op \p(\Delta \times I)$
\end{enumerate}
which is possible by \eqref{e:KBoundaryInterval}, and consider the distribution on $\Delta \times T^*I$
$$
	\ker \alpha_\rho \quad\mbox{for the $1$-form}\quad
	\alpha_\rho = \lambda_{\st} + \rho\,dt. 
$$
We now have the following lemma, whose proof is analogous to Lemma~\ref{l:CircleShell}.
\begin{lemma}
	The almost contact structure given by $\alpha_\rho$ defines a contact shell
	$(B^{I}_{K,C}, \eta_{K,\rho}^{I})$ that is independent of
	the choice of $\rho$ and $C$, up to equivalence.
	If $K > 0$, then the contact germ $(\Sigma^{I}_K, \eta_K^{I})$ extends
	canonically to a contact structure on $B^{I}_K$.
\end{lemma}
Similarly, we also have a direct description of the contact germ $(\Sigma_K^{I}, \eta_K^{I})$ without the shell
given by gluing together the contact germs on the hypersurfaces
\begin{align*}
	\tilde\Sigma_{0,K}^I &= \{(x,v,t): v = 0\} \subset \Delta \times T^*I\\
	\tilde\Sigma_{1,K}^I &= \{(x,v,t): v = K(x,t)\} \subset \Delta \times T^*I\\
	\tilde\Sigma_{2,K}^I &= \{(x,v,t): 0 \leq v \leq K(x,t)\,,\,\, (x,t) \in \p(\Delta \times I) \} \subset \Delta \times T^*I
\end{align*}
to form a contact germ on $\tilde{\Sigma}_K^{I} := \tilde\Sigma_{0,K}^{I}  \cup \tilde\Sigma_{1,K}^{I} \cup \tilde\Sigma_{2,K}^{I}$.
\begin{lemma}\label{l:EasyIntervalGerm}
	The contact germs on $\Sigma_K^{I}$ and $\tilde{\Sigma}_K^{I}$ are contactomorphic.
\end{lemma}
The proof is completely analogous to Lemma \ref{l:EasyCircleGerm}. Notice one important distinction compared to the circle model: the contact germ on $\tilde{\Sigma}_K^I$ is defined by a global immersion of the sphere into $\Delta \times T^*I$ (piecewise smooth, topologically embedded at the non-smooth points). This property allows us to use the interval model as a bridge between regular contact saucers and the circle model.


\begin{figure}
   \centering
   \def\svgwidth{200pt}
   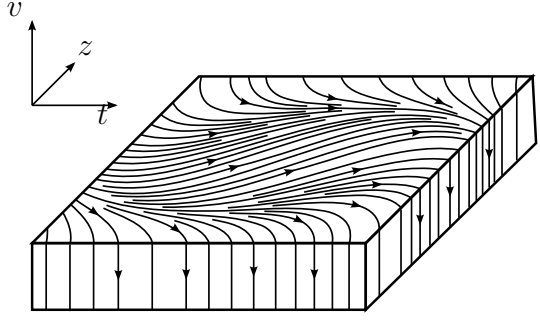
   \caption{The interval model with its characteristic distribution}
  \label{f:poly2}
\end{figure} 

\subsection{Relations between the model contact shells}\label{lm:domination}

We will now establish some domination relations between our three models. 
 
\begin{prop}\label{l:equiv12}
	For star-shaped domains $\Delta' \subset \Int \Delta$, let 
	$K : \Delta \times S^1 \to \R$ be such that $K|_{\Delta \times 0} > 0$ and $K|_{\Delta\setminus\Delta' \times S^1} > 0$.
	For $K' := K|_{\Delta' \times S^1}$, the  interval model contact shell  
	$(B_K^I, \eta^I_K)$ dominates  the circle model contact shell $(B_{K'}, \eta_{K'})$.  
\end{prop}


\begin{proof}
	Fix a $C > -\min\, K$ and a $\rho$ as in \eqref{e:FamilyRhoI} that defines contact shell models
	$$
	\mbox{$(B_K, \eta_{K, \rho})$, $(B_{K'}, \eta_{K'\!, \rho})$, and $(B_K^{I}, \eta_{K, \rho}^{I})$.}
	$$
	Take any $\eps > 0$ such that $K|_{\Delta \times [-\eps, \eps]} > \eps$ and
 	consider the domain 
\begin{align} \label{eq:keyhole}
	B_{K}^{\eps} &:= B_K^I\setminus\left(\{v\leq\eps\}\cup\{t\in[-\eps,\eps]\}\right)
	\cr&=
	\{(x,v,t) \in \Delta \times T^*I : \eps\leq v \leq K(x,t)+C\,,\,\, \eps\leq t\leq1-\eps\}.
\end{align}
	Note that $\eta^I_{K, \rho}$ restricted to $B_{K}^{\eps}$ defines a contact shell
	$(B_K^\eps,\eta^\eps_{K,\rho})$ that we will call the {\em keyhole} model, 
	and it follows from \eqref{eq:keyhole} that $(B_K^\eps,\eta^\eps_{K, \rho})$ is dominated by $(B^I_K,\eta^I_{K, \rho})$.
	 It remains to show for sufficiently small $\eps$ the shell $(B_K^\eps,\eta^\eps_{K,\rho})$ 
	 dominates $(B_{K'},\eta_{K'\!,\rho})$.
	
	Note that $(B_{K}^{\eps}, \eta_{K,\rho}^\eps)$ can be  cut out of $(B_K, \eta_{K,\rho})$ 
	by the same inequalities \eqref{eq:keyhole} where $(v,t)$ are viewed as coordinates $v=r^2$ and $t=\frac{\phi}{2\pi}$ on $\R^2$,
	rather then on $T^*I$.  This embedding is shown on Figure \ref{fig:keyhole} and explains the term `keyhole'. 

\begin{figure}\begin{center}
\includegraphics[scale=.5]{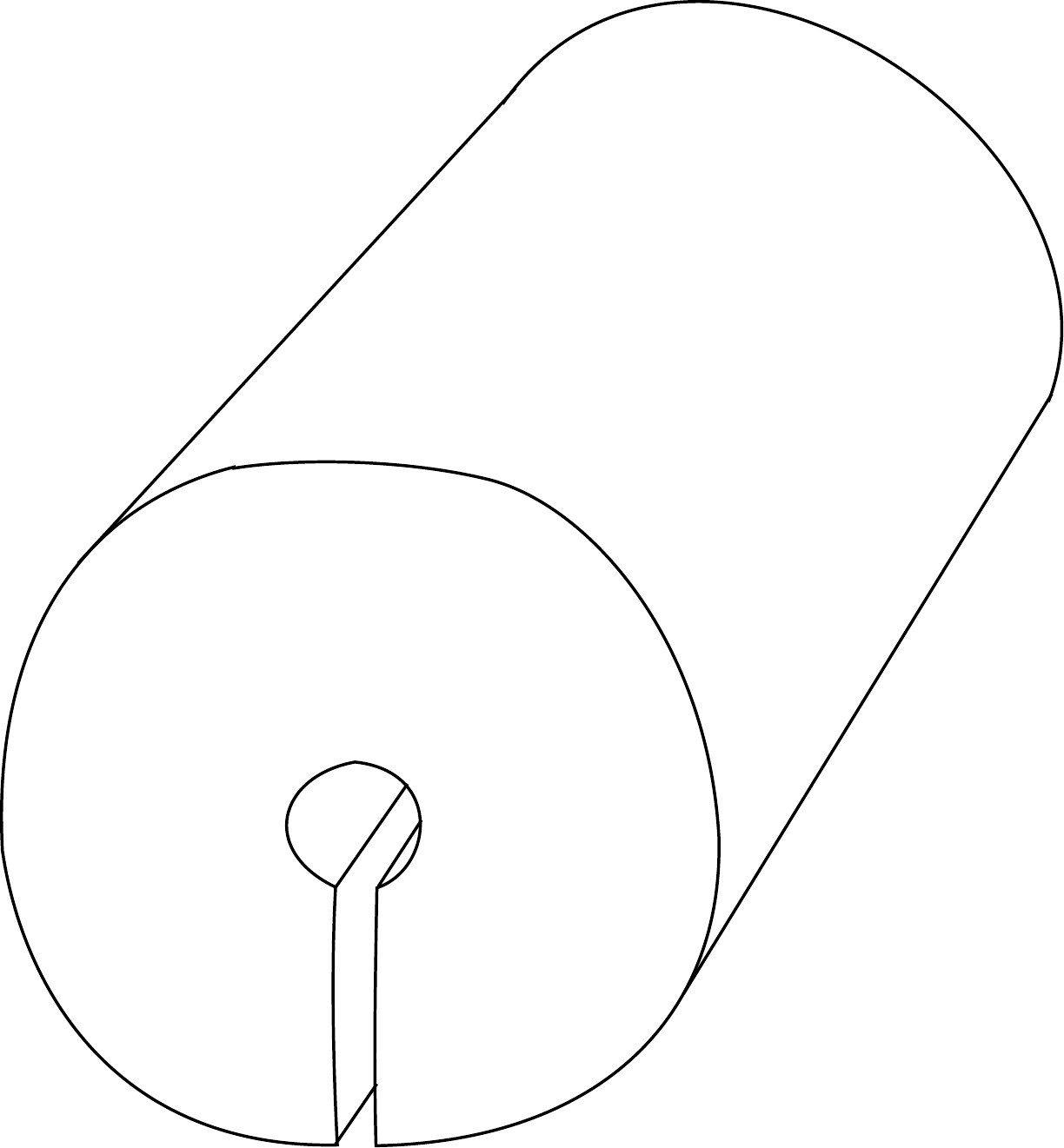}
\caption{The keyhole model inside $B_K$.}\label{fig:keyhole}
\end{center}\end{figure}

	 
	 For standard coordinates $(q, p) \in \R^2$ where $q = \sqrt{v}\cos(2\pi t)$ and 
	 $p = \sqrt{v}\sin(2\pi t)$, by the assumptions on $\rho$ in \eqref{e:FamilyRhoI}, the $1$-form on $\Delta \times \R^2$ 
	 defining $\eta_{K, \rho}$ can be written as
	 $$
	 	\alpha_\rho = \lambda_\st + \frac{\rho(v)}{2\pi v}(qdp-pdq)
	 $$
	 and on $\Delta \times \Op\{(q,0) \in \R^2: q \geq -2\delta\}$ is a genuine contact form for some $\delta >0$.
	 
	 Pick a smooth function $k: \Delta \to [-\delta, \infty)$ such that both $k(x) = -\delta$ on $\Op \p \Delta$
	 and $k(x) = K(x,0)$ on $\Op \Delta'$, and 
	 $$
	 	\Gamma_k := \{(x, q, 0) \in \Delta \times \R^2 : -2\delta \leq q \leq k(x)\} \subset B_K.
	 $$
	 Consider a smooth isotopy $\{\psi_s\}_{s \in [0,1]}$ of $\Delta \times \{q \geq -2\delta,\, p = 0\}$ 
	 of the form $\psi_s(x,q) = (x, g_s(x,q))$, supported away from $\p B_K$, and such that
	 $$
	 	\psi_1(\Gamma_k) = \{(x, q, 0) \in \Delta \times \R^2 : -2\delta \leq q \leq -\delta\} \subset B_K.
	 $$
	Since this isotopy preserves $\alpha_\rho|_{\Delta \times \{p=0\}}  = \lambda_{\st}$, it follows from a Moser-method
	argument (c.f.\ \cite[Theorem 2.6.13]{Gei08}) that $\psi_s$ can be extended to a  contact isotopy $\Psi_s$ of $B_K$  supported in 
	$\Delta \times \Op\{(q,0) \in \R^2: q \geq -2\delta\}$.
	
	If $\eps$ is small enough, then the contactomorphism $\Psi_1$ satisfies $\Psi_1(B_{K'})\subset B_K^\eps$
	and hence the keyhole model shell $(B_K^\eps,\eta^\eps_{K, \rho})$ dominates  the circular model shell 
	$(B_{K'}, \eta_{K', \rho})$.
\end{proof}	 
 
We also have the following parametric version of  Proposition \ref{l:equiv12}, whose proof is analogous.
\begin{prop}\label{l:equiv12-param}
	Let $K^\tau : \Delta \times S^1 \to \R$ be a family of contact Hamiltonians  such that 
	$K^\tau|_{\Delta \times 0} > 0, $ and  $K^\tau|_{\p \Delta \times S^1} > 0$. If
	$\Delta' \subset \Int \Delta$ is a star-shaped domain and $K'^{\tau}:= K^\tau|_{\Delta' \times S^1}$,  
	then the fibered  shell $^{T}\eta^I_{{}^T\!K}$ dominates  ${}^T\eta_{{}^T\! K'}$.  
\end{prop}

The next proposition relates our saucer models from the previous section with the interval models discussed here.  
\begin{proposition} \label{lm: saucer dom interval}
Let $\zeta = (B, \xi)$ be a regular semi-contact saucer viewed as a shell,
then $\zeta$ dominates an interval model $\eta^I_K$ for some $K:\Delta\times I \to \R$.
\end{proposition}
To prove this we will need the following two lemmas \ref{lm:boundary-behavior} and  \ref{lm:two-saucers}.

\begin{lemma}\label{lm:boundary-behavior}
Let $\Delta$ be a compact contact manifold with boundary with a fixed contact form.
Suppose $h, g:\Delta\to\Delta$ are contactomorphisms that are the time $1$ maps of isotopies
generated by contact Hamiltonians $H,G:\Delta\times I\to\R$ that vanish with their $\infty$-jet on $\p\Delta$.
If $h=g$ on $\Op\p \Delta$, then $h$ can be generated as the time $1$ map of a Hamiltonian $\wt H:\Delta\times I\to\R$
where $\wt H=G$ on $ \Op\p(\Delta \times I)$.
\end{lemma} 

\begin{proof}
Denote by $h_t$ and $g_t$ the contact diffeotopies generated by $H$ and $G$.  
Pick a contact diffeotopy such that $\wh{h}_t = h_t$ on $\p\Delta$ as $\infty$-jets and
$$
	\wh{h}_t =\begin{cases}
	g_t &\mbox{when $t \in [0,\eps]$}\\
	h_t &\mbox{when $t \in [2\eps,1-2\eps]$}\\
	g_t \circ g_1^{-1}\circ h_1 &\mbox{when $t \in [1-\eps,1]$}
	\end{cases}
$$
Observe $\wh{H} = G$ when $t \in [0,\eps] \cup [1-\eps, 1]$ if $\wh{h}_t$ is generated by $\wh{H}: \Delta \times I \to \R$.
Hence without loss of generality we may assume $H = G$ when $t \in \Op \p I$.

Since $h_t$ and $g_t$ are $C^\infty$-small on $\Op \p\Delta$, we can pick an isotopy 
$$
\psi^s: \Delta \times I \to \Delta \times I \quad\mbox{for $s\in [0,1]$ with $\psi^0 = \Id$}
$$
supported in $\Delta \times \Int I$ and such that for $x \in \Op \p\Delta$ we have
$$
	\mbox{$\psi^s(\cdot, t)$ is a contactomorphism and
	 $\psi^1(h_t(x), t) = (g_t(x), t)$.}
$$
Applying the Gray--Moser argument parametrically in $t$ builds a contact isotopy
$\tilde{\psi}_t$ such that $\tilde{\psi}_t = \Id$ when $t \in \Op \p I$ and $\tilde\psi_1(h_t(x)) = g_t(x)$
when $x \in \Op \p\Delta$.  Defining $\tilde{h}_t:= \tilde\psi_t \circ h_t$ and $\wt{H}$ to be its generating
contact Hamiltonian gives the result.
\end{proof}
 
 For the following lemma let $\Pi := \{(w,t,v) \in \R^{2n-1} \times T^*\R: v = 0\} \subset (\R^{2n+1}, \xi_\st^{2n+1})$.
 
\begin{lemma}\label{lm:two-saucers}
For a star-shaped domain $\Delta\subset
(\R^{2n-1}, \xi_\st)$ consider the disc
$$D=\{(w,t)\in \Delta \times \R: \; h_-(w)\leq t \leq h_+(w)\} \subset \Pi$$
where $h_\pm:\Delta\to\R$ are $C^{\infty}$-functions such that $h_-<h_+$.
If $(B,\sigma_\phi)$ is a contact saucer defined over $D$, then it is equivalent to a contact 
saucer $ (\wt B,\sigma_{\wt\phi})$ defined over
$\wt D=\Delta\times[0,1]$.
\end{lemma}

 \begin{proof}
 We can assume that 
 $B=\{0\leq v\leq\Phi(w,t): (w,t)\in D\}$,
   where the function  $\Phi $ is positive on $\Int D$  and coincides with $\phi$  
 on $\Op\p D$, and hence $\p B=D\cup\Gamma$  where $\Gamma:=\{v=\Phi(w,t): (w,t)\in D\}$. 
 
Given a point $w\in\Delta$, consider a leaf $\ell_w$ through the point $(w,h_-(w)) \in D$
of the characteristic foliation $\FF_\phi$ on the graph $\Gamma_\phi:=\{v=\phi(w,t): (w,t)\in D\}$
and let $(w', h_+(w')) \in D$ be the other end of the leaf $\ell_w$.   This rule $f_\phi(w):=w'$
defines the holonomy contactomorphism $ f_\phi:\Delta \to\Delta$. 

 Consider the  diffeomorphism $G:D\to \wt D=\Delta\times[0,1]$
defined by the formula
$$G(w,t)=\left(w,\frac{t- h_-(w)}{h_+(w)-h_-(w)}\right).$$
Since $\lambda_{\st} + vdt$ restricted to $D$ and $\wt{D}$ is $\lambda_{\st}$ and is preserved by $G$,
it follows that $G$ extends to a contactomorphism $G:\Op D \to \Op\wt D$.
The diffeomorphism $G$ moves the  graph of the function $\phi|_{\Op\p D}$ 
onto a graph of some function $  \wt \phi:\Op\p\wt D\to\R$ whose $\infty$-jet vanishes on $\p\wt D$. Pick any smooth extension $\wt \phi: \wt{D} \to \R$. 

The characteristic foliation $\FF_{\wt\phi}$ on the graph $\Gamma_{\wt\phi} = \{v=\wt\phi(w,t): (w,t)\in \wt{D}\}$ 
is represented by $\tfrac{\partial}{\partial t} - X_{\wt\phi_t}$ where $X_{\wt\phi_t}$ is the contact vector field on $\Delta$ for 
$\wt\phi_t: \Delta \to \R$ thought of as a contact Hamiltonian.   It follows that
the holonomy contactomorphism $f_{\wt\phi}:\Delta\to\Delta$, defined similarly to $f_\phi$, coincides
with the time one map of the contact isotopy of $\Delta$ defined by $-\wt\phi:\Delta\times[0,1]\to\R$
thought of as a time-dependent contact Hamiltonian.
According  to Lemma \ref{lm:boundary-behavior} we can modify $\wt\phi$, keeping it fixed over $\Op\p D$, to make the holonomy contactomorphism $f_{\wt\phi}$  equal to $f_\phi$.  

Since the holonomy maps $f_{\wt\phi}$ and $f_\phi$ are equal, it follows that there is a diffeomorphism
$F:\Gamma_\phi\to \Gamma_{\wt\phi}$ equal to $G$ on $\Op{\p\Gamma_\phi}$, mapping
the characteristic foliation $\FF_\phi$ to the characteristic foliation $\FF_{\wt\phi}$, and with the form
$$F(w,t,\phi(v,t))=(f(w,t), \wt\phi(f(w,t))) \quad\mbox{for $(v,t)\in D$}$$
for some diffeomorphism $f:D\to\wt D$.  It follows that $F$ extends to a contactomorphism of
neighborhoods $F: \Op \Gamma_\phi \to \Op \Gamma_{\wt\phi}$.

Let $(\wt B,\sigma_{\wt\phi})$ be a contact saucer over $\wt D$, where $\wt B=\{0\leq v\leq\wt\Phi(w,t),\, (w,t)\in\wt D\}$ for some function $\wt\Phi:\wt D\to\R$ that coincides with $\wt\phi$ on $\Op\p\wt D$ and is positive on $\Int\wt D$. 
Note $\p\wt B=\wt D\cup\wt \Gamma$ where $\wt\Gamma:=\{v=\wt\Phi(w,t),\, (w,t)\in\wt D\}$.
Let us define a diffeomorphism $H:\p B\to\p\wt B$  so that
 $$
 H|_{\Op D}=G\;\;\; \hbox{and} \;\;\; H|_{\Gamma}(w,t,\Phi(v,t))=(f(w,t),\wt\Phi(f(v,t))).$$  
This diffeomorphism matches the  traces of contact structures on  the boundaries $\p B$ and $\p\wt B$ of the saucers, and hence extends to a contactomorphism 
 between $\Op\p B$ and $\Op\p\wt B$ which can be further extended to an equivalence between the saucers $(B,\sigma_\phi)$ and $(\wt B, \sigma_{\wt\phi})$.
 \end{proof}
 
\begin{proof}[Proof of Proposition \ref{lm: saucer dom interval}]
By the definition of regular semi-contact saucer from Section~\ref{sec:saucers}, we may assume that $\zeta = (B, \sigma_\phi)$ 
where $B$ is defined over a regular domain $D\subset\Pi:=\{v=0\}\subset\R^{2n+1}$ and $\sigma_\phi$ is given by
a family of contact structures $\zeta_s$ on neighborhoods of graphs $D_s:=\{y_n=s\phi(w),\; w\in D\}$ where $\phi:D\to\R$ is a $C^\infty$-function supported in $D$ which is positive on $\Op\p D\cap\Int D$.

By the regularity assumption of $D$ the projection $\pi:D\to\R^{2n-1}$ is equivalent to the linear 
projection of the round ball and the image $\Delta=\pi(D)$ is contactomorphic to a star-shaped domain in $\R^{2n-1}$. 
Choose a slightly smaller star shaped ball $\Delta'\subset\Int \Delta$ such that
$\phi|_{\Int D \setminus \Int D'}>0$ where $D':=\pi^{-1}(\Delta')\cap D$. Note that the characteristic foliation on $D'$ is \emph{not} a regular foliation, rather it is diffeomorphic to the product foliation of the $2n-1$ disc and the interval. There exist functions $h_\pm:\Delta'\to\R$, $h_-<h_+$, such that $D'=
\{(w,t);\; h_-(w)\leq t\leq h_+(w), \; w\in\Delta'\}\subset\Pi$. Choose a function $\phi':D'\to\R$ that defines an immersion type semi-contact saucer $(B',\sigma_{\phi'})$ over $D'$
such that $\phi'\leq\phi|_{D'}$ and hence $(B',\sigma_{\phi'})$ is dominated by $(B, \sigma_\phi)$.  

Hence, we can apply Lemma \ref{lm:two-saucers} and find a function $\wt\phi:\wt D:=\Delta'\times[0,1]\to\R$ which is positive near $\p \wt D$ and such that the corresponding saucer $(\wt B, \sigma_{\wt\phi})$ over $\wt D$ is equivalent to $(B',\sigma_{\phi'})$.
 
 Let us rescale the saucer $(\wt B,\sigma_{\wt\phi})$ by an affine contactomorphism of $\R^{2n+1}_\st$
 $$(z,u,\varphi, t,v)\mapsto \left((1+\delta)^2z, (1+\delta)^2u, \varphi, (1+\delta) t-\frac\delta2, (1+\delta)v\right)$$
 for $\delta > 0$,
  to an equivalent  saucer $(\wh B, \sigma_{\wh\phi})$ over the domain
 $\wh D:=\wh \Delta\times[-\frac\delta2,1+\frac\delta2]$.
The notation $\varphi$ stands for the tuple $(\varphi_1,\dots,\varphi_{n-1})$ of angular coordinates.
 Note that $\Delta'\subset\Int\wh \Delta$ and we may choose $\delta$ sufficiently small so
  that $\wh\phi|_{\Int\wh D\setminus\Int \wt D}>0$.
   
   Then the restriction  $K:=\wh\phi|_{\wt D} $  of the function $\wh \phi$ to the domain $\wt D=\Delta'\times[0,1]$ defines an interval  model shell $(B^I_K,\eta^I_K)$. It is dominated by the saucer $(\wh B, \sigma_{\wh\phi})$,
   which is equivalent to $(B',\sigma_{\phi'})$, which is in turn dominated by $(B,\sigma_{\phi})$ .
   \end{proof}

Similarly, one can prove    the following parametric version of Proposition \ref{lm: saucer dom interval}.
\begin{proposition}\label{lm:saucer-dom-interval-param}
Let ${}^T\zeta=({}^TB,{}^T\xi)$ be a fibered  regular semi-contact saucer.
 Then ${}^T\zeta$ dominates a fibered  interval model ${}^T\eta^I_K$ for some $K:\left({}^T \Delta\right)\times I \to \R$.   
\end{proposition}

\begin{remark}{\rm Let us point out that the shells $(B^\tau, \xi^\tau)$ degenerate when $\tau$ approaches $\p T$. Hence, the subordination map
$({}^TB^I_K,{}^T\eta^I_K)\to ({}^TB,{}^T\xi)$ has to cover an embedding
$T\to\Int T$.
}
\end{remark}

The next proposition is the main result in this section. 

 \begin{prop}\label{prop:saucer-to-circle}
If $(B, \xi) = \sigma_\phi$ is a regular semi-contact saucer viewed as a shell,
then there is a time-independent contact Hamiltonian $K:\Delta\to\R$
such that $(B, \xi)$ dominates the circle model contact shell $(B_K, \eta_K)$.
\end{prop}

\begin{proof}
 We first use  Proposition \ref{lm: saucer dom interval} to find
  a dominated by $\zeta$ interval model $\eta^I_{\wt K}$ for some ${\wt K}:D\times I \to \R$. Then, we apply Proposition \ref{l:equiv12} to get a circular  model
  contact shell $(B_{K'}, \eta_{K'})$ dominated by $\left(B^I_{\wt K},   \eta^I_{\wt K}\right))$. Finally, choosing a time-independent contact Hamiltonian $K<K'$ and applying  Lemma \ref{l:HamiltonianShellsAnnulus} we get the required  circle model contact shell 
  $\left(B_K, \eta_K\right)$  dominated by $ (B, \xi)$.
\end{proof}

Similarly, the parametric versions Propositions~\ref{lm:saucer-dom-interval-param} and \ref{l:equiv12-param} prove the following.

\begin{prop}\label{prop:saucer-to-circle-param}
Let $({}^TB, {}^T\xi)$ be a fibered  regular semi-contact saucer,
then there exists a family of time-independent  contact Hamiltonians
$K^\tau:\Delta\to\R$ for $\tau\in T$  which satisfy $K_\tau>0$ for $\tau\in \p T$, such that $({}^TB, {}^T\xi)$ dominates the corresponding fibered  circle model contact shell $({}^TB_K, {}^T\eta_K)$.  
\end{prop}

\section{Reduction to saucers}\label{sec:RedSaucers}
     
\subsection{Construction of contact structures in the complement of saucers}\label{sec:to-saucers} 
The goal of this section is to prove the Proposition~\ref{prop:to saucers}.
 The starting point of the proof is  Gromov's $h$-principle for contact structures  on open manifolds, which we will now formulate.
  Given a $(2n+1)$-dimensional manifold  $M$, possibly with boundary, a closet subset  $A\subset M$  and  a  contact structure
  $\xi_0$   on $\Op A\subset M$  define
$\Cont(M;A,\xi_0)$ to be the space of contact structures on $M$ that  coincide with 
$\xi_0$ on $\Op A$ and 
  $\cont(M;A,\xi_0)$ to be the space of {\em almost} contact structures
that agree with $\xi_0$ on $\Op A$. Let 
$j:\Cont(M;A,\xi_0)\to\cont(M;A,\xi_0)$   be  the inclusion map.
We say that the pair $(M,A)$ is {\em relatively open} if for any point $x\in M\setminus A$ either  there   exists a path in $M\setminus A$ connecting $x$ with a boundary point of $M$, or a proper path $\gamma:[0,\infty)\to M\setminus A$ with $\gamma(0)=x$.
 \begin{theorem}\label{thm:Gromov-open}{\rm[M. Gromov, \cite{Gro69,Gr-PDR}]}
 Let $M$ be a $(2n+1)$-manifold,  $A\subset M$ a closed subset, and
  $\xi_0$   a contact structure on  $\Op A\subset M$. 
  Suppose that  $(M,A)$ is relatively open.
  Then the inclusion $$j:\Cont(M;A,\xi_0)\to\cont(M;A,\xi_0)$$ is a homotopy equivalence. 
 \end{theorem}
As we will see, Proposition \ref{prop:to saucers} follows from the following special case.

   \begin{lemma}\label{lm:annuli-to-saucers} For a closed manifold $\Sigma$, any semi-contact structure  $\xi=\{\zeta_s\}$  on the annulus 
    $C=\Sigma\times[0,1]$ is homotopic relative $\Op\p C$ to an almost contact structure $\wt \xi$ which is a genuine contact structure in the complement of   finitely many saucers $B_1,\dots, B_k\subset C$ and such that
    the restriction $\wt\xi|_{B_j}$ for $j=1,\dots, k$, 
     is semi-contact and regular.
    \end{lemma}
    
    \begin{figure}\begin{center}
\includegraphics[scale=.5]{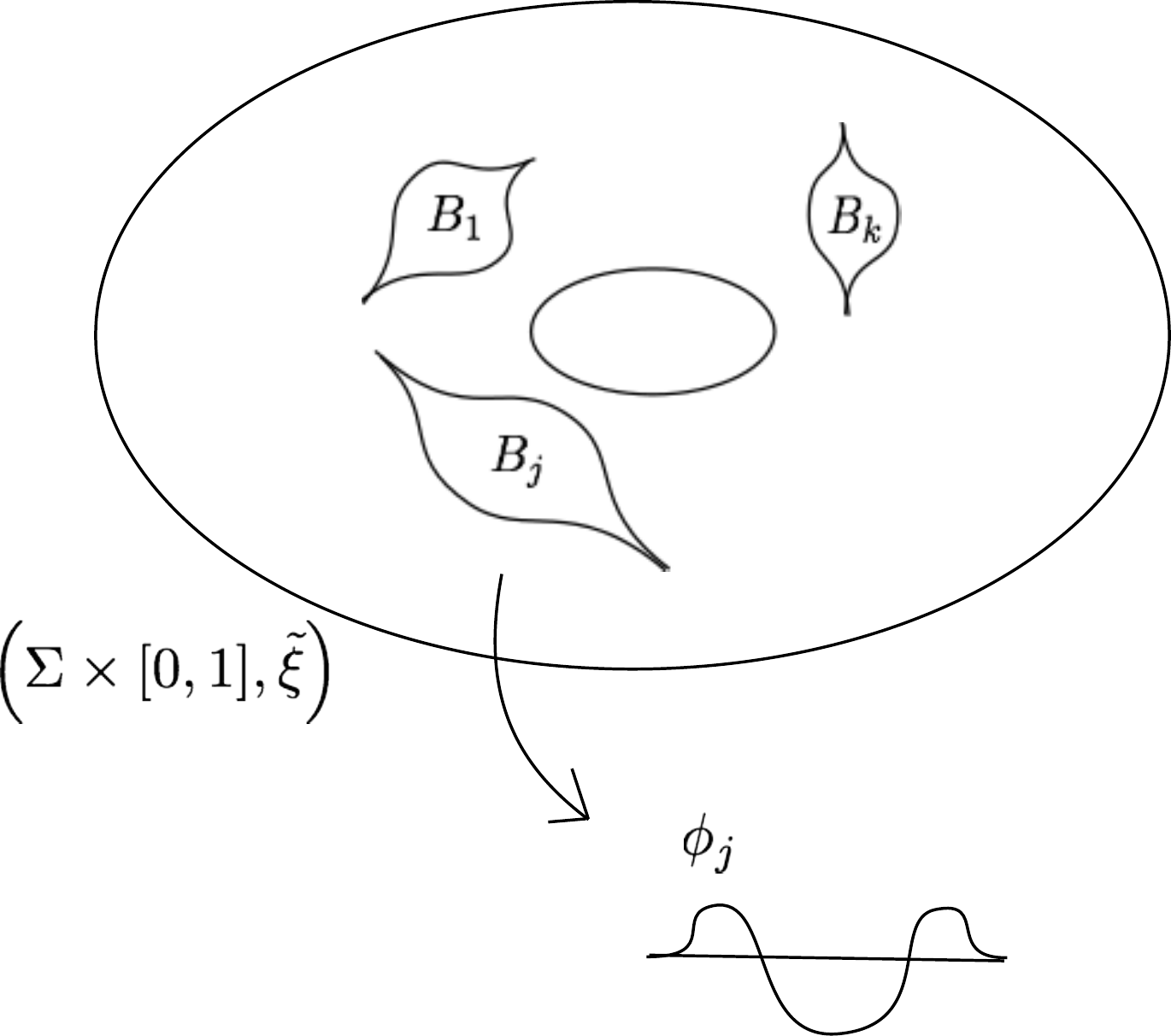}
\caption{A representation of the statement of Lemma \ref{lm:annuli-to-saucers}.}\label{fig: lemma32}
\end{center}\end{figure}
    
  \begin{proof}[Proof of Proposition \ref{prop:to saucers}] Choose an embedded annulus $C=S^{2n}\times[0,1]\subset M\setminus A$ and first use  the existence part of Gromov's $h$-principle \ref{thm:Gromov-open}, to deform $\xi_0$ relative to $A$ to an almost contact structure which is genuine on $M\setminus C$. Next with use the $1$-parametric part of Theorem  \ref{thm:Gromov-open}, applied to the family of neighborhoods of spheres $S^{2n}\times t$ for  $t\in[0,1]$ to make the almost contact structure semi-contact on $C$. Finally we use Lemma \ref{lm:annuli-to-saucers} to complete the proof.
 \end{proof}

 We will need two lemmas in order  to prove Lemma \ref{lm:annuli-to-saucers}.

 Observe we are free to partition $[0,1] = \bigcup_{i=0}^{N} [a_i, a_{i+1}]$
 where $a_i < a_{j}$ for $i < j$ and prove the lemma for the restriction of the semi-contact structure to
 each $\Sigma \times [a_i, a_{i+1}]$, which we will do multiple times through the proof.
 
   \begin{lemma}\label{lm:partition-to-special}
  For a closed manifold $\Sigma$ and a semi-contact structure $\{\zeta_s\}$ on 
  $\Sigma\times[0,1]$ there exists $N>0$ such that the restriction  of 
 $\{\zeta_s\}$  to $ \Sigma\times[a_i,a_{i+1}]$ for $a_i:=a+ \frac iN$ is of immersion type for each $i=0,\dots, N-1$.
 \end{lemma}
 \begin{proof}
 Choose  $\eps>0$ such  that the contact structure $\zeta_s$ is defined on $\Sigma\times[s-\eps,s+\eps]$ for each $s\in[a,b]$. We will   view $\zeta_s$ as a family of contact structures   on $\Sigma\times[-\eps,\eps]$. 
 For each $s_0$ and $\sigma > 0$ sufficiently small, a Darboux-Moser type argument implies there is an isotopy 
 $\phi^{s_0}_s: \Sigma \times [-\sigma, \sigma] \to \Sigma \times [-\eps, \eps]$ such that $\phi^{s_0}_{s_0} = \Id$ and 
 $(\phi^{s}_{s_0})^*\zeta_s = \zeta_{s_0}$ for $s \in [s_0, s_0+\sigma]$.  Moreover by shrinking $\sigma$ if necessary we can ensure that the hypersurfaces $\phi_{s_0}^s(\Sigma\times 0)$ are graphical in $\Sigma\times[-\eps,\eps]$.
Hence for any $s_0\in[a,b]$ the restriction of $\zeta_s$ to $[s_0, s_0+\sigma]$ is of immersion type and therefore choosing 
$N>\frac1\sigma$ we get the required partition of $\Sigma\times[a,b]$ into the annuli of immersion type. 
\end{proof}

\begin{lemma}\label{lm:Reeb-trick}
Let $\{\xi_s\}$ be a semi-contact structure on $\Sigma\times[0,1]$. Then, after partitioning, $\{\xi_s\}$ is equivalent to a semi-contact structure $\{\zeta_s\}$ which is immersion type satisfying the following properties. There exists a smooth
  function $\psi:\Sigma\to [-\frac R2,\frac R2]$ and a contact structure $\mu$ on $\Sigma\times[-R,R]$ such that
  \begin{itemize}
  \item the germ of contact structure $\zeta_s = \Psi_s^*\mu$ on $\Sigma\times[s-\delta, s+\delta]$ where
  $$
  \Psi_s: \Sigma\times[s-\delta, s+\delta] \to \Sigma\times[-R,R] \quad
  $$ by the embedding $(x,s+t)\mapsto (x, s\psi(x)+t)\in\Sigma\times[-R,R]$ where $x\in\Sigma$ and $t\in[-\delta,\delta]$;
  \item  
 there  are closed domains $V\subset\Sigma$ and $\wh V\subset \Int V$ such that $\psi|_V >0$,  and the contact structure $\mu$ is   transverse to   graphs of functions $s\psi $ over $\wh W=\Sigma\setminus\Int\wh V$ for all $s\in[0,1]$.
  \end{itemize} 
\end{lemma}
\begin{proof}
 Let us endow $\Sigma\times[-R,R]$  with the product metric.
 Using  Lemma \ref{lm:partition-to-special}  we may assume,
 by passing to a partition, that the
 semi-contact structure $\{\xi_s\}$ on the annulus $C$ is of immersion type.  
 So there is a contact structure $\mu$ on $\Sigma\times[-R,R]$ and a 
    smooth family of 
    embeddings $$\Psi_s: \wh{C}:=\Sigma \times [-\delta, \delta] \to \Sigma\times[-R,R] \quad\mbox{for $s \in [0,1]$}$$ 
    such that $\Psi_s(x, u) = (x, \psi_s(x)+u)$ for $u\in[-\delta,\delta]$, and $\Psi_s^*\mu$ is identified with $\zeta_s$. Furthermore, the partition argument also allows us  to assume that $\psi_0=0$ and that the $C^1$-norm of $\psi_1$ is arbitrary small. We will impose the appropriate bound on its $C^1$-norm further down in the proof.

    Define  a   subset  
    \begin{align*}
	  \wh V&:=\left\{x\in\Sigma: \mbox{ the angle between } \mu_{x,0} \hbox{ and }  T_{x,0}\Sigma \;\hbox {in $T_{x,0}\wh{C}$ is
	at most } \tfrac\pi4\right\}.\\
\end{align*}

Assuming $\mu$ and $\Sigma$ cooriented, we present $\wh V$ as a
  disjoint union 
$\wh V=\wh V_+\cup \wh V_-$  defined by whether  the vector field $\frac{\p}{\p u}$ defines the same or opposite coorientation of   $\mu|_\Sigma $
and $T\Sigma$.

  We can choose  a contact form $\alpha$ for $\mu$  such that its Reeb vector field $\Reeb$  satisfies $\Reeb(x,u)=\pm\frac{\p}{\p u}$ for $x\in\wh V_\pm$ and $u\leq\sigma$ for a sufficiently small $\sigma>0$.
   We will assume  $\sigma>0$ small enough to also ensure  that 
 $$\abs{\mbox{angle}(\mu_{x,u},\mu_{x,0}) } \leq \tfrac{ \pi}{16}$$
 for $x\in\Sigma$ and $|u|\leq\sigma$.
 
  Consider any smooth function $H:\Sigma\to[-1,1]$  equal to $\pm 1$ on $\wh V_{\pm}$. We further extend this function to
 $\Sigma\times[-R,R]$ as   independent of the coordinate  $u$   and then multiply by the cut-off function $\beta(u)$ equal to $1$ on $[-\frac R3,\frac R3]$ and equal to $0$ outside $[-\frac R2,\frac R2]$. We will keep the notation $H$ for the resulting function on $\Sigma\times[-R,R]$.
  
 Let $h_t:\Sigma\times[-R,R]\to\Sigma\times[-R,R]$. $t\in[0,1]$, be the contact  isotopy generated by $H$ as its contact Hamiltonian.   Choose $\eps\in(0,\frac\sigma2)$  so small that 
 $dh_\eps$ rotates any hyperplane by an angle $<\frac\pi{16}$ and 
 $||h_\eps||_{C^0}<\sigma$.
 Then   $\Gamma_\eps:=h_\eps(\Sigma\times0)$ is graphical, $\Gamma_\eps=\{u=\phi_\eps(x)\}$ for a   function $\phi_\eps:\Sigma\to\R$ which satisfies the following conditions
 \begin{itemize}
 \item $\phi_\eps|_{\wh V}=\eps$;
 \item for every $x\in\wh  W:=\Sigma\setminus\Int \wh V$ we have
${\mathrm{angle}}(T_{z}\Sigma,\mu_z)>\frac\pi8$, $z=(x,\phi_\eps(x))\in\Gamma_\eps$.  \end{itemize}
 
As it was explained above,  the function $\psi_1\colon\Sigma\to\R$  entering the definition of  the semi-contact annulus can be chosen arbitrarily $C^1$-small. Hence, by requiring it to satisfy the conditions
$|\psi_1|< \frac\eps2$ and $||\psi_1||_{C^1}<\tan\frac\pi{16}$ we can conclude that 
 $\wt\Gamma  :=h_\eps(\{u= \psi_1(x),\,x\in\Sigma\})$ is graphical, $\wt\Gamma  =\{u= \wt\psi_1 (x)\}$ for a   function $\wt\psi_1:\Sigma\to\R$ which satisfies the following conditions
\begin{itemize}
 \item $\wt\psi_1|_{\wh V}\geq\frac\eps2$;
 \item for every $x\in \wh W=\Sigma\setminus\Int \wh V$  and $s\in[0,1]$ we have
$${\mathrm{angle}}(T_{z}\wt\Sigma_s,\mu_z)>\frac\pi{16},\;\;z=(x,s\wt\psi_1(x))\in\wt\Gamma_s=  \{u= s\wt\psi_1 (x)\}.$$  \end{itemize}
Let us denote   $  V:=\{\wt\psi_1 \geq\frac\eps4\}$, so that $\wh V\subset\Int V$. 
Note that the family  of contact structures $\wt\zeta_s$  on $\Sigma\times s\subset\Sigma\times[0,1]$ induced by $\mu$ on the neighborhoods of graphs $\Gamma_s$ of functions $  s\wt\psi$ defines a semi-contact structure $\{\zeta_s\}$ which is 
equivalent to $\{\xi_s\}$. This concludes the proof of Lemma \ref{lm:Reeb-trick}.
\end{proof}

\begin{proof}[Proof of Lemma \ref{lm:annuli-to-saucers}]
According to  Lemma \ref{lm:Reeb-trick}
 we  may assume that 
  the semi-contact structure $\{\zeta_s\}$ on 
  $\Sigma\times[0,1]$  is of immersion type  with   the following properties. There exist  a contact structure $\mu$ on $\Sigma\times[-R,R]$ and a function $\psi:\Sigma\to [-\frac R2,\frac R2]$ such that
  \begin{description}
  \item{(i)} the germ of contact structure $\zeta_s$ is induced from $\mu$ on  a neighborhood of the $\Sigma\times s$ by an embedding $(x,s)\mapsto (x, s\psi(x)+t)\in\Sigma\times[-R,R] , x\in\Sigma, t\in[-\delta,\delta]$;
  \item {(ii)}
 there  are closed domains $V\subset\Sigma$ and $\wh V\subset \Int V$ such that $\psi|_V >0$,  and the contact structure $\mu$ is transverse to   graphs of functions $s\psi $ over $\wh W:=\Sigma\setminus\Int\wh V$ for all $s\in[0,1]$.
\end{description}
 
  We will keep the notation  $\psi$ for  the restriction of $\psi$ to $\wh W$.
  Note that $\psi$ can be presented as the difference
  $\psi=\psi_+-\psi_-$ of two positive functions $\psi_\pm \in C^\infty(\wh W)$such that  the graphs of the functions $s\psi_\pm$ are transverse to $\mu$. 
  
  Let $\{U_i\}_{i=1}^N$   be a  finite covering  of $W:=\Sigma\setminus\Int V$ by interiors of balls with smooth boundaries and such that $\bigcup\limits_1^N\overline{U_i}\subset \wh W$.

   Let $\{\lambda^\pm_i: \Sigma \to [0,1]\}$ be two partitions of unity on $W$ subordinate to the covering  $\{U_i\}_{i=1}^N$,   $
   \sum_{i=1}^N\lambda^\pm_i|_{W}=1$, such that
   $\Supp(\lambda^i_-)\Subset\Supp(\lambda^i_+),\, i=1,\dots, N,$  
   and $   \sum_{i=1}^N\lambda^\pm_i|_{\wh W}\leq 1$.
      
       For $0 \leq k \leq N$ define 
   $$L_k:=\sum_{i=1}^k\lambda^+_i
   $$
   noting $L_N|_{W}=1$ and $L_N|_{\wh V}=0$ shows $\wh V \subset U \subset \Sigma \setminus W$,
   where $U:= \{L_N < 1\}$.

 For $1 \leq i \leq N$ define the functions
$$\psi^\pm_i:=\psi^\pm\lambda^\pm_i: \Sigma \to \R
\quad\mbox{and}\quad
 \psi_i:=\psi^+_i-\psi^-_i: \Sigma \to \R$$
 and for $0 \leq k \leq N$ the functions
 $$
 \Psi_k^\pm:=\sum\limits_1^k\psi^\pm_i \quad\mbox{and}\quad
  \Psi_k:= \Psi_k^+-\Psi_k^-= \sum\limits_1^k\psi_i\,.
 $$
 One can further ensure that the graphs of the functions $\Psi^\pm$ are transverse to $\mu$.
Let $\Gamma(\Psi_k):=\{u=\Psi_k(x),\;x\in\Sigma\}\subset\wh C=\Sigma\times[-R,R]$ be the graph of   $\Psi_k$  and likewise $\Gamma(L_k):=\{u=L_k(x),\; x\in\Sigma\}\subset C= \Sigma\times[-1,1]$ be 
 the graph of $L_k$.   Set  $\Gamma_L:= \bigcup_{k=0}^N \Gamma(L_k)$
 and consider the map $p: \Gamma_L \to \Sigma\times[-R,R]$
 $$
 p(x,u)=(x,\psi^+(x)u -\Psi^-_k(x))\; \hbox{ for}\;(x,u)\in \Gamma(L_k).
 $$
  This map is well defined because if $(x,u)\in   \Gamma(L_i)\cap  \Gamma(L_j)$ for $0\leq i<j\leq N$ then
 $\Psi^-_i(x)=\Psi^-_j(x).$ Indeed, $L_i(x)=L_j(x)$ implies that $\psi^+_l(x)=0$ for all $i<l\leq j$, and hence $\psi^-_l(x)=0$ for   $i<l\leq j$ because $\Supp(\psi^-_l)\subset\Supp(\psi^+_l)$.
 Note that  $p({ \Gamma_L})=\bigcup\limits_0^N \Gamma(\Psi_i)$.
 The map $p$  extends to an immersion $P:\Op{ \Gamma_L}\to \Sigma\times[-R,R]$. See Figure \ref{fig: stack}.
 
 \begin{figure}\begin{center}
\includegraphics[scale=.5]{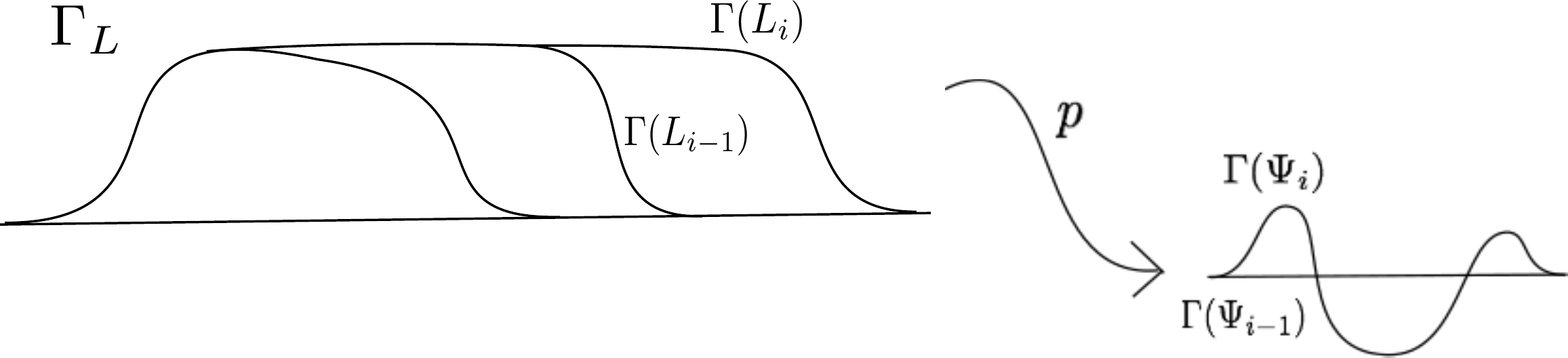}
\caption{The set $\Gamma_L$ used in the proof of the lemma. The region below a bump function is a saucer, so a partition of unity decomposes a general region into small saucers. Each saucer comes equipped with a map to $\wh C$ making it a contact shell.}\label{fig: stack}
\end{center}\end{figure}
 
 The complement $C\setminus \Gamma_L$ is a union of  the domain $\Omega:=\{(x,u);\;\Phi_N(x)\leq u\leq 1, x\in U\}$ and the interiors of saucers $B_i$ bounded by the graphs $\Gamma(L_{i-1})$ and $\Gamma(L_i)$ over the balls $\overline{U_i}$, $i=1,\dots N$.
 
 We extend the immersion $P$   to $\Omega$ as a diffeomorphism
 $\Omega\to\{(x,u);  \Psi_N(x)\leq u\leq  \psi_1(x),\;x\in U\}$ which is fiberwise linear with respect to the projection to $\Sigma$, 
 
It  remains to extend the induced contact structure $P^*\mu$ on $\Op{ \Gamma_L}\cup \Omega$ as a regular semi-contact structure on the saucers. The following lemma is obvious.
 
 \begin{lemma}\label{lm:foliation}
 Let $\Sigma$  be a  hypersurface in a $(2n+1)$-dimensional contact manifold transversal to the contact structure, and $f:D^{2n}\to \Sigma$ a smooth embedding of the unit $2n$-ball. Then there exists $\eps>0$ such that the   disc $f(D_\eps^{2n}) $
 is  regular. 
\end{lemma}
 
 It follows that the covering by balls $U_i^+$ can be chosen to ensure   that the  discs $(\overline{U}^+_i, \zeta_0)$ are regular.
 If the function  $ \psi$ is  sufficiently $C^1$-small then    for each  $i=1,\dots, N$ the graphs $\Gamma^-_i \subset \Gamma(\Psi_{i-1})$ and $\Gamma^+_i \subset \Gamma(\Psi_i)$ of the functions $\Psi_{i-1}|_{\ol U_i^+}$ and $\Psi_{i}|_{\ol U_i^+}$ respectively, are regular as well.
Hence there is a  contactomorphism $g_i$ between a neighborhood    $O_i\supset\Gamma_i^-$  and a neighborhood of a disc in $\Pi=\{y_n=0\}\subset\R^{n+1}$.
Again, if $ \psi$   is  sufficiently small then  the neighborhood $O_i$ contains the disc $\Gamma_i^+$ and, moreover, $g_i(\Gamma^+_i)$ is transverse to the vector field $\frac\p{\p y_n}$, which ensures the regularity of the contact saucer $B_i$, $i=1,\dots, N$.

  Finally, it remains to observe that the required $C^1$-smallness  of the function  $ \psi$  can be achieved by passing to a partition. This concludes the proof of Lemma \ref{lm:annuli-to-saucers}. 
 \end{proof}  

\subsection{Contact structures in the complement of saucers. Parametric version} \label{saucers-param}
We prove in this section the following parametric version of Proposition
 \ref{prop:to saucers}.
  \begin{proposition}\label{prop:to-saucers-param}
 Let $ M$ be a $(2n+1)$-manifold and $A\subset M$ a closed set, so that  $M\setminus A$ is connected. Let ${}^T\xi_0$ be a fibered almost contact structure on ${}^TM$ which is genuine on $(T\times\Op A)\cup (\p T\times  M)\subset T\times M$. Then there exist a finite number of (possibly overlapping) discs $T_i\subset T$ and disjoint embedded fibered saucers ${}^{T_i}B_i\subset {}^TM$, $i=1,\dots, N$,
 such that  ${}^T\xi_0$ is homotopic rel.\ $(T\times A)\cup \p T\times M $ to a fibered almost contact structure ${}^T\xi_1$ which is is genuine on ${}^TM\setminus\bigcup B_i$ and whose restriction to each fibered saucer ${}^{T_i}B_i$ is semi-contact and regular.  
 Moreover, we can choose the discs in such a way that any non-empty intersection $T_{i_1}\cap\dots\cap T_{i_k}$, $1\leq i_1<\dots<i_k\leq L$, is again a disc with piecewise smooth boundary.
 \end{proposition}

As in the non-parametric case the following lemma is the main part of the proof.
   \begin{lemma}\label{lm:annuli-to-saucers-param} Any  fibered  semi-contact structure  ${}^T\xi=\{\zeta^\tau_s\}_{s\in[0,1],\tau\in T}$  on the fibered annulus 
    ${}^TC:=T\times \Sigma\times[0,1]$ is homotopic relative $T\times \Op\p C\cup\p T\times C$ to a fibered almost contact structure ${}^T\wt \xi$ which is a genuine contact structure in the complement of   finitely many  fibered saucers ${}^{T_1}B_1,\dots, {}^{T_k}B_k\subset {}^TC$, and such that
    the restriction ${}^{T_j}\wt\xi|_{{}^{T_j}B_j}$ for each  $j=1,\dots, k$ 
     is semi-contact and regular.
     Moreover, we can choose the discs $T_i\subset T$ in such a way that any non-empty intersection $T_{i_1}\cap\dots\cap T_{i_p}$, $1\leq i_1<\dots<i_p\leq k$, is again a disc with piecewise smooth boundary.

    \end{lemma}
     First, note that 
  Lemma \ref{lm:partition-to-special}   
  have the following parametric analog.

       \begin{lemma}\label{lm:partition-to-special-param}
 Given a  fibered semi-contact structure ${}^T\!\xi$ on $T\times\Sigma\times[0,1]$ there exists $N>0$ such that the restriction  of 
 ${}^T\xi$  to $ T\times \Sigma\times[a_i,a_{i+1}]$, $a_i:=a+i\frac{b-a}N$ is of immersion type for each $i=0,\dots, N-1$.
 \end{lemma}

 \begin{proof}[Proof of Lemma \ref{lm:annuli-to-saucers-param}]
Using Lemma \ref{lm:partition-to-special-param}, we can assume that the fibered
 semi-contact structure ${}^T\xi$ on the annulus 
    ${}^TC=T\times C$ is of immersion type, 
    i.e.\ there exists a fibered contact structure 
    ${}^T \mu=\{\mu^\tau_{\tau\in T}\}$ 
    on ${}^T\wh C:=T\times\Sigma\times[-R,R]$ and a family of functions  
    $\psi^\tau_s:\Sigma\to[- r,r]$, $r<R$,
    $s\in[0,1], \tau\in T$, such that the contact structure $\zeta^\tau_s$  on a neighborhood of $\Sigma^\tau _s$ is induced from 
    $\mu^\tau$ by an embedding of this neighborhood onto a neighborhood of the graph of the function $\psi^\tau_s$. 
    
     Let $T'\subset\Int T'', T''\subset \Int T$   be two slightly smaller  compact parameter spaces such that the semi-contact structure $\mu^\tau$ is contact for $\tau\in T\setminus\Int T'$, i.e.\ $\psi^\tau_s(x)>\psi^\tau_{s'}(x)$ for any $x\in\Sigma, \tau\in T\setminus\Int T'$ and $s>s'$.

Similarly to the non-parametric case (see Lemma \ref{lm:Reeb-trick}) we can  reduce to the case when
the following property holds: $\psi^\tau_0=0,\psi^\tau_s=s\psi^\tau$ and 
there exist    domains $\wh W, W \subset T\times\Sigma$, $\Int \wh W\Supset W$ such that $\psi^\tau(x)>0$ for $(x,\tau)\in
T'\times\Sigma\setminus\Int W$, and the contact structures $\mu^\tau$ transverse to the graphs $\Gamma^\tau_s $ of the functions
$s\psi^\tau$ over $V^\tau:=(\Sigma\setminus\Int W)\cap{\tau\times\Sigma}$.

 Denote by $\psi$ the function $\psi(\tau,x,s):=\psi^\tau(x,s)$ for $(\tau,x,s)\in W$, and present $\psi$   as the difference of two positive functions, $\psi=\psi^+-\psi^-$.

  Let us choose a  finite covering  of $ W$ denoted $\{U_i\}$, 
 such that  $U_i=\Int T_i\times \Int\Delta^\pm_i $,
  where $\Delta_i\subset\Sigma,\; T_i\subset T''$, are  balls with smooth boundaries, $ i=1,\dots, N$, and  $\bigcup\limits_1^N \overline{U_i}\subset \wh W$. Here   we  choose the discs $T_i\subset T$ in such a way that any non-empty intersection $T_{i_1}\cap\dots\cap T_{i_k}$, $1\leq i_1<\dots<i_k\leq L$, is again a disc with piecewise smooth boundary.  More geometric constraints on the coverings will be imposed below.
  
  Let $\lambda^\pm_i$ be two  partitions   of unity over $\wh W$ subordinated to  $U_i$ so that $\Supp(\lambda_i^-)\Subset\Supp(\lambda_i^+)$, 
     $$\sum\lambda^\pm_i|_{ W}=1,\;\hbox{ and} \; \sum\lambda^\pm_i|_{\wh W}\leq 1.$$

 Let us denote
$$\psi^\pm_i:=\psi^\pm\lambda^\pm_i,\;
 \psi_i:=\psi^+_i-\psi^-_i,\; i=1,\dots, N.$$ 
 Set $$\Psi_k^\pm:=\sum\limits_1^k\psi^\pm_i,\;\; \Psi_k:=  \Psi_k^+-\Psi_k^-=\sum\limits_1^k\psi_i,\;\Phi_k:=\sum\limits_1^k\lambda^+_i.$$
Note that $L_N|_{ W}=1$ and $L_N|_{\wh V}=0$, so $V \subset U:= \{L_N(x) < 1\} \subset T\times\Sigma \setminus   W$.

In ${}^TC=T\times\Sigma\times[0,1]$ we let $\Gamma(L_k)$ be the graph of the function $L_k$, and in ${}^T\wh C = T\times\Sigma\times[-R,R]$ we let $\Gamma(\Psi_k)$ be the graph of the function  $\Psi_k$. Set
$\Gamma_L=\bigcup\limits_0^N \Gamma(L_i)\subset {}^TC$.

Consider the map $p:{ \Gamma_L}\to T\times \Sigma\times[-R,R]$ given by the formula
 $$p(\tau,x,s)=(\tau,x,  \psi^+(\tau,x)s -\Psi^-_i(\tau,x))\; \hbox{ for}\;(\tau,x,s)\in \Gamma(L_i).$$
  This map is well defined because if $(\tau, x,s)\in   \Gamma(L_i)\cap  \Gamma(L_j)$ for $0\leq i<j\leq N$ then
 $\Psi^-_i(\tau,x)=\Psi^-_j(\tau,x).$ Indeed, $\Phi_i(\tau, x)=\Phi_j(\tau,x)$ implies that $\psi^+_l(\tau,x)=0$ for all $i<l\leq j$, and hence $\psi^-_l(\tau,x)=0$ for   $i<l\leq j$ because $\Supp(\psi^-_l)\subset\Supp(\psi^+_l)$.
 Note that  $p({ \Gamma_L})=\bigcup\limits_0^N \Gamma(\Psi_i)$.
 The map $p$  extends to an immersion $P:\Op{ \Gamma_L}\to T\times \Sigma\times[-R,R]$.

 The complement ${}^TC\setminus \Gamma_L$ is a union of the 
 domain $\Omega:=\{(x,s);\;\Phi_N(\tau,x)\leq s\leq 1, \tau\in T, x\in U\}$ and  interiors of fibered saucers ${}^{T_i}B_i$ bounded by the graphs $\Gamma(L_{i-1})$ and $\Gamma(L_i)$ over the balls $\overline{U_i}$, $i=1,\dots N$.
 
 We extend the immersion $P$   to $\Omega$ as a fiberwise linear,
 with respect to the projection to $T\times \Sigma$,
 diffeomorphism
  $$\Omega\to\{(\tau, x,u);  \Psi_N(\tau,x)\leq u\leq  \psi(\tau, x),\;(\tau,x)\in U\}.$$
 
It  remains to extend the induced contact structure $P^*\left({}^T\mu\right)$ on $\Op{ \Gamma_L}\cup \Omega$ as a fibered regular semi-contact structure to the saucers.  
 
 It follows from Lemma \ref {lm:foliation} that the covering by balls $U_i$ can be chosen to ensure   that for each $i=1,\dots, N$ and  $\tau \in T^+_i$ the  disc  $(\tau\times\Delta^+_i, \zeta^\tau_0)$ is regular.
 If the functions $ \psi_1$  and $ \psi_0$ are sufficiently $C^1$-close than    for each  $i=1,\dots, N$ the graphs $\Gamma^-_i \subset \Gamma(\Psi_{i-1})$ and $\Gamma^+_i \subset \Gamma(\Psi_i)$ of the functions $\Psi_{i-1}|_{\ol U_i^+}$ and $\Psi_{i}|_{\ol U_i^+}$, respectively, are fibered over $T_i^+$ by regular discs  as well.
Hence, there is a fibered over $T_i^+$ contactomorphism $g_i$ between a neighborhood    $O_i\supset\Gamma_i^-$  and a neighborhood of a  fibered disc in $T_i\times \Pi=\{y_n=0\}\subset\R^{n+1}$.
Again, if $ \psi$ is  sufficiently close then  the neighborhood $O_i$ contains the disc $\Gamma_i^+$ and, moreover, $g_i(\Gamma^+_i)$ is transverse to the vector field $\frac\p{\p y_n}$, which ensures the regularity of the fibered  contact saucer ${}^{T_i}B_i$, $i=1,\dots, N$.

  Finally, it remains to observe that the required smallness of the function  $ \psi$    can   be   achieved by passing to a partition.   \end{proof}  
  \medskip
 
  \begin{proof}[Proof of Proposition \ref{prop:to-saucers-param}] First assume $T = D^q$. Choose an embedded annulus $C=S^{2n}\times[0,1]\subset M\setminus A$ and first use   the $q$-parametric part of Gromov's $h$-principle \ref{thm:Gromov-open} to deform ${}^T\xi_0$ rel.\ $(T\times A)\cup(\p T\times M)$ to a fibered almost contact structure which is genuine fibered contact structure on $T\times M\setminus C$. Next, with use of the $(q+1)$-parametric part of Theorem \ref{thm:Gromov-open}   applied to the family of neighborhoods of spheres $\tau\times S^{2n}\times t, \tau\in T, t\in[0,1]$  we make the fibered almost contact structure semi-contact on ${}^TC=T\times C$. Finally we use Lemma \ref{lm:annuli-to-saucers-param} to complete the proof. For general $T$ we  triangulate it and  inductively over skeleta apply the previous proof to each simplex.
 \end{proof}


\section{Reduction to a universal model}\label{sec:universal}
We prove in this section Proposition \ref{prop:unique-model}.
  
  \subsection{Equivariant coverings}\label{sec:equi-cov}
The key step in the proof of Proposition \ref{prop:unique-model} is the following  
 \begin{proposition}\label{prop:circular-to-universal-saucers}
 For a fixed dimension, there is a finite list of saucers $\{(B_p,\zeta_p)\}$ for $p=1,\dots, L,$
with the following property.
 For any circle model contact shell $(B_K, \eta_K)$ defined by a
   time-independent contact Hamiltonian $K\colon\Delta\to\R$,
    there exist finitely many  disjoint balls $B_i\subset B_K$ for $i=1,\dots, q$
     so that the contact shell $(B_K, \eta_K)$ is homotopic relative to $\Op\p B_K$ to an almost contact structure 
     $\xi$ that is genuinely contact on $B_K \setminus \bigcup_{i=1}^{q} B_i$
     and each contact shell  $\xi|_{B_i}$ is equivalent to one of the saucers $(B_p,\zeta_p)$ for $p=1,\dots, L$.  
     \end{proposition}

\begin{remark}  {\rm
  The proof of  Proposition  \ref{prop:circular-to-universal-saucers} follows roughly the same scheme as the proof of Proposition 
   \ref{prop:to saucers}
  but uses the idea of equivariant coverings in a crucial way. The basic idea can be seen in the following trivial observation about real functions. Consider the piecewise constant function $\phi: \R \to \R$ which is equal to $1$ on $[0,1) \cup [2,3)$, equal to $-3$ on $[1,2)$, and $0$ elsewhere. Let the group $\Z$ act on $\R$ by translation: $j \in \Z$ being identified with the map $x \mapsto x + j$. Then the function $\sum\limits_{j=1}^k \phi \circ j$ is equal to $1$ on $[0,1) \cup [k+2, k+3)$ and it is strictly negative on $[1,k+2)$.
  
The key point of this example is two-fold: firstly, that a function which is negative on an arbitrarily large portion of its support can be written as a sum of functions which are negative on a small subset of their support. And secondly, that in fact these functions can be taken to be translations of a single function by a group action.}
\end{remark}
 
  Consider $\R^{2n+1}$ with the contact structure $\xi_\st$ given by the form
  $$\alpha=dz+\sum\limits_1^{n-1}(x_idy_i-y_idx_i)-y_ndx_n = dz + \sum\limits_1^{n-1}u_id\varphi_i - y_ndx_n.$$
  Denote $\Pi=\{y_n=0\}$.
  In the group of contactomorphisms $\Diff(\R^{2n+1},\xi_\st)$ consider the $2n$-dimensional lattice  $\Theta$ generated by the following transformations:
  \begin{itemize}
  \item translations  
  $$T_z:  (x, y,z)\mapsto (x , y,z+1)$$ $$T_{x_n}: (x_1,\dots, x_n , y ,z)\mapsto (x_1,\dots, x_n+\frac12, y, z),$$
    \item sheers
 $$S_{y_j}:(x , y_1,\dots, y_j,\dots, y_{n},z,t)\mapsto (x , y_1,\dots, y_j+1,\dots, y_{n},z+x_j),\; j=1,\dots, n-1$$
 $$S_{x_j}:(x_1,\dots, x_j,\dots, x_n , y ,z)\mapsto (x_1,\dots, x_j+1,\dots, x_n , y,   z-y_j),\; j=1,\dots, n-1.$$
 \end{itemize}

 Notice that $\Theta$ preserves $\Pi$, we have 
 $S_{y_j}S_{x_j}S_{y_j}^{-1}S_{x_j}^{-1} = T_z^2$, and all other transformations commute.
 Hence every element of $\Theta$ may be written as
 $$
 	S_{x_1}^{k_1} \cdots S_{x_{n-1}}^{k_{n-1}}
	S_{y_1}^{l_1} \cdots S_{y_{n-1}}^{l_{n-1}}
	T_{x_n}^{k_n}T_{z}^{l_n}
 $$
 and from this it follows that $\Theta$ acts properly discontinuously on $\Pi$, that is
for any compact set $Q \subseteq \Pi$, the set
 \begin{equation}\label{eq:SQ}
S(Q):=\{g\in\Theta: g(Q)\cap Q\neq\varnothing\}\subset\Theta
\end{equation}
 is finite.  
 
For any positive number $N$, let $C_N$ be the scaling contactomorphism
 $$(x,y,z)\mapsto (Nx,Ny,N^2z).$$ 
 Let $\Theta_N = C^{-1}_N\Theta C_N$, that is $\Theta_N$ is the group generated by translations $T_{j,N}:=C_N^{-1}\circ T_j\circ C_N, T_{z,N}:=C_N^{-1}\circ T_z\circ C_N$ and sheers
 $S_{j,N}:=C_N^{-1}\circ S_j\circ C_N$. Say that a compact set $Q$ \emph{generates a $\Theta_N$-equivariant cover of $\Pi$} if $\Theta_N \cdot \Int(Q) = \Pi$.  Below we will always assume that $N$ is a positive integer, in this case the element $T_{x_n} = T_{x_n, N}^N$ belongs to $\Theta_N$. We denote  by $\Upsilon$ the normal subgroup in $\Theta_N$ generated by $T_{x_n}^2$, and by $\wh\Theta_N$ the quotient group $\Theta_N/\Upsilon$.
 We call a compact set $Q\subset\Pi$ {\em sufficiently small} if $T_{x_n}^2(Q)\cap Q=\varnothing$.
 
Notice that the quotient of $\R^{2n+1}$ by the contactomorphism $T^2_{x_n}$ is the contact manifold 
$$\left(\R^{2n-1} \times T^*S^1, \ker(dz-\sum\limits_1^{n-1}(x_idy_i-y_idx_i)+vdt)\right)$$
where $v = -y_n$ is identified with the fiber coordinate of $T^*S^1$ and the base coordinate $t \in \R/\Z$ is given by the quotient by translation  $T^2_{x_n}$. Denote this quotient by $\pi:\R^{2n+1} \to \R^{2n-1} \times T^*S^1$.  The group $\wh\Theta_N$ can be viewed as a subgroup of the group of contactomorphisms of $\R^{2n-1} \times T^*S^1$ preserving
$\wh \Pi=\pi(\Pi)\simeq\R^{2n-1}\times \R/\Z$.
Any compactly supported function $\Phi :\Pi \to \R$ defines a function $\sum\limits_{h \in \Upsilon} \Phi \circ h^{-1}$ which is 1-periodic in the $x_n$-variable, and therefore defines a function $\wh\Phi: \wh\Pi \to \R$.

 \begin{remark}\label{rm:generator}{\rm
 \begin{description}
 \item{a)} If  $Q$  generates a  $\Theta$-equivariant covering of $\Pi$   then $Q_N:=C_N^{-1}(Q)$ generates a  $\Theta_N$-equivariant covering. For a sufficiently large $N$ the set $Q_N$ is sufficiently small.
 \item{b)} Suppose that $a>\frac12$. Then the parallelepiped $$P=\{|z|\leq a, \;|x_j|, |y_j|\leq a,\ 1\leq j\leq  n-1,\; 0<x_n\leq a,\;  y_n=0\}\subset\Pi$$ generates  a  
 $\Theta$-equivariant covering of $\Pi$. If $a<1$ then $P$ is sufficiently small. In particular, there are sufficiently small sets generating equivariant coverings.
 \item{c)} If $Q'\subset Q\subset \Pi$ are two compact sets, and $Q'$  generates a  $\Theta_N$-equivariant covering of $\Pi$, then so does $Q$. 
   \end{description}
 }
  \end{remark}

  Let  us {\bf fix for the rest of the paper}  a regular sufficiently small disc $Q\subset\Pi$  and a smaller disc $Q^\prime\subset \Int Q$ which generates a $\Theta$-equivariant covering of $\Pi$.  We Denote by $m$ the cardinality $|S(Q)|$ of the set $S(Q)$.
  
 We also fix  two non-negative    $C^\infty$-functions $\phi_+,\phi_-:\Pi \to \R$   which are  supported in $Q$ which satisfy the following conditions:
  \begin{enumerate}
 \item $\phi_+|_{\Int Q}>0, \phi_-|_{ Q^\prime}>0$, and $\phi_-|_{\Op(\p Q)} = 0$;
 \item  $\max(\phi|_{Q^\prime}) < -(m+1)\mu $, where $\phi:=\phi_+-\phi_-$,\; $\mu:=\max(\phi)$ and  $m$ is  the cardinality of the set $S(Q)$ defined in \eqref{eq:SQ};
 \item   denote $\phi^s=\phi_+-s\phi_-$, $s\in[0,1]$, (so that $\phi^s\geq\phi^1=\phi$). For any finite subset $F\subset \Theta$ denote
 $$\Phi^s_F:=\mu+\sum\limits_{g\in F} \phi^s\circ g^{-1}|_Q,\;  
   \; s\in[0,1].$$
 Then the graph  $y_n= \Phi^{s}_F(q),\; q\in Q,$ with the induced contact structure is regular.
\end{enumerate}

\begin{remark}{\rm
  In condition (iii) elements $g\in F$ with $g(Q) \cap Q=\varnothing$ are irrelevant, so it suffices to verify (iii) only for    subsets $F$ of the  finite set $ S(Q)$. Hence the condition can always be satisfied  by taking $\phi_+$ and $\phi_-$ sufficiently small (e.g.\ replacing the pair $(\phi_+,\phi_-)$ which satisfy (i) and (ii) by $(\eps\phi_+,\eps\phi_-)$ for a sufficiently small $\eps>0$.)  }
\end{remark}
 
Let us linearly order in any way all the elements of $\Theta$:
$g_1,g_2,\dots $ and order accordingly $\wh\Theta_N$, we fix this ordering during the paper. Define functions $\Pi \to \R$ by the formulas
$$
	\Phi_k:= \mu+ \sum\limits_{j=1}^k\phi\circ g_j^{-1}  \quad\mbox{and}\quad
	\Psi_k= \mu+ \sum\limits_{j=1}^k\phi_+\circ g_j^{-1} 
$$
for $k=0,1\dots,  $ and note $\Phi_0=\Psi_0\equiv \mu$. 
   
Let 
$${}^{\Phi}\Gamma^-_k :=\{y_n=  \Phi_{k-1}(x_1,\ldots,x_n,y_1,\ldots, y_{n-1},z);\; (x_1,\ldots,x_n,y_1,\ldots, y_{n-1},z) \in g_k(Q)\}$$ be
 the graph of $ \Phi_{k-1}$ over the set $g_k(Q)$, 
and  similarly denote by ${}^\Phi\Gamma^+_k$ the graph of $\ \Phi_k$ over $g_k(Q)$. Denote by ${}^\Psi\Gamma^+_k$ the graph of $ \Psi_k$ over $g_k(Q)$ and by ${}^\Psi\Gamma^-_k$ the graph of $ \Psi_{k-1}$
over $g_k(Q)$. Define $B_k$ to be the saucer
\begin{align*}
B_k: &= 
 \{ \Psi_{k-1}(x_1,\ldots,x_n,y_1,\ldots, y_{n-1},z)\leq y_n\leq  \Psi_{k}(x_1,\ldots,x_n,y_1,\ldots, y_{n-1},z),\\
 & (x_1,\ldots,x_n,y_1,\ldots, y_{n-1},z) \in g_k(Q)\}
 \end{align*}
 bounded by ${}^\Psi\Gamma_k^-$ and ${}^\Psi\Gamma_k^+$ .
Similar to the proof of Proposition \ref{prop:to saucers} we observe that there is an immersion $\Op\p B_k\to\R^{2n+1}$ which maps ${}^\Psi\Gamma_k^\pm\to {}^\Phi\Gamma_k^\pm$ diffeomorphically. The induced contact structure $\zeta_k$ with its canonical  semi-contact extension to $B_k$ defines a shell structure on the saucer $B_k$.
 
More generally for $s \in [0,1]$ set 
$$\Phi^{s}_k:=\mu+ \sum_{j=1}^k\phi^u\circ g_j^{-1} $$ so that
$\Phi^1_k=\Phi_k$ and $\Phi^0_k=\Psi_k$.
Define the regular contact saucer  $(B_k, \zeta_k^s)$ , $k=1,\dots,$ as induced by an immersion of $\Op B_k\to \R^{2n+1}$ that
 maps $\p B_k = {}^\Psi\Gamma_k^+ \cup {}^\Psi\Gamma_k^-$ 
 diffeomorphically onto the graphs of  $ \Phi^{s}_k$ and $ \Phi^{s}_{k-1}$  over $ g_k(Q)$.
 Regularity is ensured by the above condition (iii).
 
 Therefore, the above construction builds
$1$-parameter families for $s \in [0,1]$ of regular contact saucers $(B_k, \zeta_k^s)$ for $k=0,1, \dots  $.
However  as the next lemma shows, up to equivalence the number of these is always bounded by $ L=2^m$, where $m:=\abs{S(Q)}$ is the cardinality of the set 
  from \eqref{eq:SQ}.

\begin{lemma}\label{lm:finite-contact}
 Up to equivalence the above construction builds at most $L=2^m$, $m= \abs{S(Q)}$,
 one-parametric families of regular contact saucers
 $(B_k,\zeta^{u}_k)$. 
\end{lemma}
\begin{proof}
By the contactomorphism $g_k^{-1}$ we know $\zeta^{s}_k$ is equivalent to a saucer  whose boundary contact germ is  defined by the two graphs over $Q$: 
$$
\mbox{$\{y_n = (\Phi^{s}_{k-1} \circ g_k)|_Q\}$ and $\{y_n = (\Phi^{u}_{k} \circ g_k)|_Q\}$.}
$$
However $\Phi^{s}_k|_Q \circ g_k = \mu+\sum\limits_1^k(\phi^s\circ (g_j^{-1}g_k))|_Q$ and
the number of different sums of this type is bounded above the number $L=2^{m }$ of finite subsets of the set $S(Q)$.
 
 \end{proof}


Given   a positive  $N$  and  an element $g\in \Theta_N$ we denote
 $\phi_{g,N}:=\frac1N\phi\circ C_N\circ g^{-1}$.
 Notice that by the contactomorphism $C_N^{-1}$ the regular semi-contact  saucer which is defined over the domain $Q$ by the  functions $\Phi^s_{k-1}$ and $\Phi^s_{k}$ is equivalent to the saucer  over the domain $ C_N^{-1}(Q)$ defined by the functions $\Phi^s_{k-1,N}:=\frac\mu N+ \sum\limits_{j=1}^{k-1}\phi^u_{g_j,N}$ and  $\Phi^s_{k,N}:=\frac\mu N+ \sum\limits_{j=1}^{k}\phi^s_{g_j,N}$, where
 $\phi^s_{g,N}:=\frac1N\phi^s\circ C_N\circ g^{-1}$.
 
Consider the function $\wh{\phi_{g,N}}: \wt\Pi \to \R$ We note that $\wh{\phi_{g,N}}=\wh{\phi_{g',N}}$ if $g,g'$ are in the same conjugacy class from $\wh\Theta_N=\Theta_N/\Upsilon$, so that in the notation for $\wh{\phi_{g,N}}$ we can use $g\in\wh \Theta_N$.

 \begin{lemma}\label{lm:positive-decomposition}
With the above choices of $Q,Q',\phi_+,\phi_-$,  for  any bounded open domains $U^\prime$ and $U\Supset U^\prime$ in $\R^{2n-1}$ and any $C^\infty$-function $K\colon
U\to\R$  which is positive on $(U\setminus\overline U^\prime)$, there exist
$N>0$ and a finite subset $\Lambda \subset\wh\Theta_N$ 
     such that
     \begin{itemize}
     \item  $ U^\prime \times S^1 \Subset  \bigcup \limits_{g\in \Lambda} g(\Int \pi(Q'_N))\subset   \bigcup \limits_{g\in \Lambda} g(\Int \pi(Q_N)) \Subset  U \times S^1;$ 
 \item  $  \sum\limits_{g\in  \Lambda}\wh{\phi_{ g,N}}  < \;
  \begin{cases}
  -\frac{2\mu} N& \hbox{on}\; U'\times S^1;\cr
  K - \frac\mu N& \hbox{on}\; (U\setminus U')\times S^1.
  \end{cases}
   $
   \end{itemize}
 \end{lemma}
  
 \begin{proof}
 Suppose $K\colon U \to \R$ is given. Since $K$ is positive on $U \setminus U'$, we may fix some $\eps > 0$ with the property that the set $P := \{(x,y,z) \in U \setminus U': K(x,y,z) > \eps\}$ disconnects $U'$ from $\p U$. Notice that the conclusion of the lemma only becomes stronger if we enlarge the set $U' \Subset U$. With this in mind we redefine $U'$ to be the interior of the union of all components of $U \setminus P$ which are disjoint from $\p U$.

Set $\wh Q_N:= \pi(Q_N),\;\wh Q'_N:= \pi(Q'_N)$. For a sufficiently large $N$ there exists a finite set $  \Lambda\subset\wh\Theta$ such that $(U' \cup P) \times S^1\Supset \bigcup \limits_{g\in \Lambda}g(\wh Q_N) \supset \bigcup \limits_{g\in \Lambda}g(\wh Q^\prime_N) \Supset U' \times S^1$.  Furthermore, suppose that
  \begin{equation}\label{eq:largeN}
  N> (m+1)\mu\eps^{-1}   .
  \end{equation}

 Then, using  \eqref{eq:largeN} we get on $P \times S^1$
 $$  \sum\limits_{g\in \Lambda}\phi_{g,N} <   \frac{m \mu}N  = \frac{(m+1) \mu}N -\frac\mu N  < \eps-\frac\mu N <K-\frac\mu N.$$
On the other hand, on $U'\times S^1$ we have

$$
  \sum\limits_{g\in\Lambda}\phi_{g,N} <-\frac{(m+1)\mu}N + \frac{(m-1)\mu}N  < - \frac{2\mu} N.
$$
  Indeed this holds, for given $(x,y,z) \in g(Q'_N)$  because
  according to inequality (ii) a single negative term $\wh{\phi_{g,N}}(x,y,z)$ is  larger in absolute value by at least $\frac{2\mu}N$ than the sum of all   positive terms  (the denominator $N$ appears because of the scaling factor of the function in the definition of  $\phi_{g,N}$).
\end{proof}

  \begin{proof}[Proof of Proposition \ref{prop:circular-to-universal-saucers}]
  Let $U=\Int \Delta$ and $U'\Subset U$ be a star-shaped subset such that $K|_{U\setminus U'}>0$. Apply Lemma \ref{lm:positive-decomposition}, providing   an integer $N>0$, and a finite set $\Lambda\subset\wh\Theta_N$.
 Then the corresponding  function 
$$\Phi=\Phi^{S^1,N}=\frac\mu N+\sum\limits_{g\in \Lambda}\phi_{g,N} :\Delta\times S^1\to\R$$ satisfies
 $\Phi (w,t)< K(v)$ for $w\in \Delta\setminus U'$ and $\Phi|_{U'\times S^1}<-\frac\mu N$.   According to Proposition  \ref{prop:shallow} there exists a contact Hamiltonian $\wt K$ so that $\eta_{\wt K}$ is dominated by $\eta_K$,
 where $\wt K|_{\Delta\setminus U'}=K|_{\Delta\setminus U'}$  and $\wt K|_{U'}>-\frac\mu N$. Therefore $\Phi(w,t)<\wt K(w)$ for all $(w,t) \in \Delta \times S^1$.
 The function $\Phi$ is equal $\frac\mu N>0$ near $\p \Delta\times S^1$ and hence defines a circular shell model $\eta_{\Phi}$ which is dominated by $\eta_{\wt K}$. Hence it sufficient to prove the required extension result for $\eta_{\Phi}$.  We order $\Lambda$ using the chosen ordering of $\wh\Theta$ and  define functions 
 $$\Phi_k= \frac\mu N+\sum\limits_{j=1}^k \phi_{g_j,N}:\Delta\times S^1\to\R,\; k=0,\dots,|\Lambda|,$$ where $|\Lambda|$ is the cardinality of $\Lambda$.
 We have $\Phi_0=\frac\mu N$ and $\Phi_{|\Lambda|}=\Phi$. The  shells $\eta_{\Phi_k}$ and 
 $\eta_{\Phi_{k-1}}$ differ by one of the  regular saucers $(B_p,\zeta_p)$, from the finite list provided by Lemma \ref{lm:finite-contact}, while the    shell $\eta_{\Phi_0}$ is solid, since $\Phi_0>0$ everywhere.   
 \end{proof}

 Now we are ready to prove Proposition \ref{prop:unique-model}.
\begin{proof}[Proof of Proposition \ref{prop:unique-model}]
 Proposition  \ref{prop:to saucers} allows us to   assume that $\xi$ is contact outside of a finite collection of disjoint saucers  $\{B_i\}_{1\leq i\leq N}$, so that the restriction  $\xi|_{B_i}$ for each $i=1,\dots, N$, is a regular semi-contact saucer.  Using Proposition \ref{prop:saucer-to-circle} we replace saucers by
  circular model shells  defined  by time-independent contact Hamiltonians.
  Applying
Proposition \ref{prop:circular-to-universal-saucers}
  we   can further reduce to the case of a contact structure 
  in the complement of saucers from the finite list  $(B_p,\zeta_p)$, $p=1,\dots, L$.  Using again  Proposition \ref{prop:saucer-to-circle} we replace saucers by
  circular model shells $\left(B_{K_p},\eta_{K_p}\right)$ defined  by time-independent contact Hamiltonians.  We may then choose any special Hamiltonian $K_\univ$ satisfying $K_\univ(x) < \min_p K_p(x)$.
  \end{proof}

 \subsection{The standardization of the holes in the parametric case}\label{sec:standard-param}
 In this section we prove Proposition \ref{prop:unique-model-param}.
 
Given a special Hamiltonian $K:\Delta_\cyl\to\R$ we recall  the notation from Section~\ref{sec:univ-holes}:
 
 $$K^{(s)}: = sK + (1-s)E, \; s\in[0,1], \;\hbox{where  }\; E(u,z):=K(u,1),$$
%

\begin{lemma}\label{lm:universal-eps}
There exist a special Hamiltonian $K_\univ:\Delta_\cyl\to\R$ and a non-increasing function   $\theta:[0,1]\to[0,1]$ with  
  $\theta(0)=0$ and $\theta(1)=1$,  which depend  only on the choice of
$Q,Q',\phi_+,\phi_- $,  and such that for each $p=1\dots, 2^m$
 there exists     a family of subordination
maps $$\eta_{K^{(\theta(s))}_{\univ}}\to\eta^{s}_p:=(B_p,\zeta^{s}_p), \;  s\in[0,1], $$ 
  which are solid for $s=0$.
 \end{lemma}
\begin{proof}    We note that there exists $\delta>0$ such that the regular saucer $(B_p,\zeta^{s}_p)$ is solid for $s\in[0,\delta]$, i.e.\ the contact structure on its boundary is extended inside as a genuine contact structure.
Proposition \ref{prop:saucer-to-circle-param} implies that   the family of    saucers  $(B_p,\zeta^{s}_p)$ dominates a family of  circular models $\eta_{\wt K^{s}_p}$, where $\wt K^{s}_p>0$ for  $s\in[0,\delta']$, $p=1,\dots, 2^m$, for some $\delta'<\delta$. 
 We also  note that Lemma \ref{lm:changing-B} allows us to assume that  the domain $\Delta$ in the definition of the Hamiltonians $\wt K_p^{s}$ coincides with    $\Delta_\cyl$.  Choose as $K_\univ$ any  special contact Hamiltonian which satisfies 
 $$K_\univ<\mathop{\min}\limits_{s\in[0,1],p=1,\dots, L=2^m}\wt K^s_p.$$ (See Example \ref{ex:special}.) There exists $\delta''\in(0,\delta')$ such that $\wt K^s_p>K^{(0)}_\univ$ for all $s\in[0,1],\; p=1,\dots, 2^m$.
 Choose a non-decreasing function $\theta:[0,1]\to[0,1]$ such that
 $\theta(s)=0$ for $s\in[0,\tfrac{\delta'}2]$ and $\theta(s)=1$ for $s\in[\delta'',1]$. Then $K_\univ^{(\theta(s))}<\wt K^s_p$ for all $s\in[0,1], p=1,\dots, 2^n$. Hence, by Lemma  \ref{l:HamiltonianShellsAnnulus} one can arrange the inclusion maps
 $\eta_{K_\univ^{(\theta(s))}}\to \eta^s_p$ be subordinations.
   \end{proof} 
 
 \begin{remark}
{\rm It is not clear that any Hamiltonian $K_\univ$ satisfying Proposition \ref{prop:unique-model} also satisfies Lemma \ref{lm:universal-eps}, or conversely. But once we know that there are two Hamiltonians separately satisfying Proposition \ref{prop:unique-model} and Lemma \ref{lm:universal-eps} we can simply choose $K_\univ$ to be less than both of them, and this Hamiltonian will suffice for both.}
\end{remark}
 
 Let  $T=D^q\subset\R^q$ be the unit disc. Choose any decreasing $C^\infty$-function $\theta:[0,1]\to [0,1]$ which is equal to $1$ on $[0,\frac13]$ and to $0$ on $[\frac23,1]$.

  \begin{proposition}\label{prop:circular-to-universal-param}
  There is a universal finite  list of families saucers $(B^{s}_p,\zeta^s_p)$, $p=1,\dots, L$, $s\in[0,1]$, where $L$ depends only on dimension $n$, with the following property. 
 Let  
   $K^\tau:\Delta\to\R$, $\tau\in T$,  be a family of 
   time-independent contact Hamiltonians parameterized by the unit disc  $T=D^q\subset\R^q$, 
   and such that $K^\tau(x)>0$ for $(\tau,x)\in \p(T\times\Delta)$.    Let $({}^TB=T\times B,{}^T\eta)$ be the fibered circular shell defined by this family.
   Denote by   ${}^T\eta_p$ the shell corresponding to the family of saucers $\eta^{(\theta(||\tau||))}_p$.
    Then  there exist finitely many balls $B_i\subset B $, $i=1,\dots, N$, with piecewise smooth boundary, so that the fibered  contact shell ${}^T\eta $ 
    (viewed as a fibered almost contact structure on ${}^TB $) is homotopic rel.\ $\Op \p (T\times B) $ 
    to a fibered almost contact structure ${}^T\xi$ which is genuinely 
    contact on $T\times(B \setminus \bigcup B_i)$, and  such that each fibered contact shell  ${}^T\xi|_{B_i}$  is equivalent to  one of the fibered saucer shells ${}^T\eta_p$, $p=1,\dots, L$.
     \end{proposition}
     
   \begin{proof}   
First, we can choose $\ol K^s: \Delta \to \R$, $s \in [0,1]$, so that $\ol K^{||\tau||} \leq K^{ \tau }$ everywhere, $\ol K^1 > 0$ and $\ol K^s|_{\p \Delta} >0$ for all $s\in[0,1]$. It, therefore, suffices to prove the proposition for the family $\ol K^{||\tau||}$. We can also assume that $\ol K^0(x) \leq \ol K^s(x)$ for any $x \in \Delta$, $s \in [0,1]$.

   Let $U=\Int \Delta$ and $U'\Subset U$ be a star-shaped subset such that $\ol K^{||\tau||}|_{U\setminus U'}>0$ for all $\tau\in T$. We also choose $\delta>0$ so that $\ol K^s > 0$ for all $s \in [1-\delta, 1]$. Apply Lemma \ref{lm:positive-decomposition} to $\ol K^0$, providing an integer $N>0$, and a finite set $\Lambda=\{g_1,\dots, g_k\}\subset\wh \Theta_N$, so that the function 
   $$\Phi=\Phi^{S^1,N}=\frac\mu N+ \sum\limits_{g\in \Lambda}\phi_{g,N} :\Delta\times S^1\to\R$$ satisfies
 $\Phi (w,t)< \ol K^s(w)$ for $w\in \Delta\setminus U'$  and $\Phi|_{U'\times S^1}<-\frac\mu N$. Choosing $N$ large enough we  can also arrange that  $\mathop{\min}\limits_{||\tau||\geq1-\delta} \ol K^\tau>
 \Psi =\Psi^{S^1,N}=\frac\mu N+ \sum\limits_{g\in 
 \Lambda}\phi_{+,g,N}=\frac\mu N+ \frac1N\sum\limits_{g\in 
 \Lambda}\left(\phi_{+}\circ C_N\circ g^{-1}\right)$.
  
According to Proposition  \ref{prop:shallow-param} there exists  a family of  functions $\wt K^s $, $s\in[0,1]$,
such that
	 \begin{itemize}
	 \item $\wt K^s=\ol K^s$ on 
	 $\Delta \setminus  U'$,  $\wt{K}^s > -\frac\mu N$,  $s\in[0,1]$;
	 \item $\wt K^s= \ol K^s$ for $s\in[1-\delta,1]$;
\item there exists a family of subordination maps   $h^s: \eta_{\ol K^s}\to \eta_{\wt{K^s}}$
which are identity maps 	for $s\in[1-\delta,1]$.
\end{itemize} 

  In particular, $\Phi(w,t)<\wt K^s(w)$ for all $w\in \Delta, t\in S^1, s\in[0,1]$. 
 
 Recall the notation $\phi^s= \phi_+-s\phi_-,\; \phi^s_{g_j,N}=\frac1N\phi^s\circ C_N\circ g_i^{-1} $ from Section \ref{sec:equi-cov}.
  Choose a diffeomorphism $f:[0,1]\to[0,1]$ such $f(1-\frac\delta2)=\frac23, f(1-\delta)=\frac13$. Then the function $\wt\theta:=\theta\circ f$ satisfies  $\wt\theta|_{[0,1-\delta]} = 1$  and    $\wt\theta|_{[1-\frac\delta2,1]} = 0$.
  
  We define the families of  functions $\Phi^s_i:= \frac\mu N+ 
  \sum\limits_{j=1}^i \left( \phi^{\wt{\theta}(s)}_{g_j,N}\right):\Delta\times S^1\to\R$, 
  so we have $\Phi^s_0=\frac1N$, $\Phi^s_{|\Lambda|}=\Phi$ for $s\leq 1-\delta$, and $\Phi^1_{|\Lambda|}=\Psi$. Here ${|\Lambda|}$ is the cardinality of $\Lambda$.

 The function  $\Phi^s:=\Phi^s_k$ for each $s\in[0,1]$ is equal $\frac1N>0$ near $  \p \Delta\times S^1$, and 
it satisfies the inequality $\Phi^s<\wt K^s$. Indeed, for $s\in[0,1-\delta]$ we have  $\Phi^s=\Phi<\wt K^s$, and for $s\in[1-\delta,1]$ we have
$\Phi^s<\Psi<\ol K^u=\wt K^s$.
Therefore, the  family of circular shell models $\eta_{\Phi^{||\tau||}}$  is dominated by $\eta_{\wt K^{||\tau||}}$, and hence it is sufficient to prove the required extension result for the family $\eta_{\Phi^{||\tau||}}$.  

The family of   model shells $\eta_{\Phi^{||\tau||}_i}$ and 
 $\eta_{\Phi^{||\tau||}_{i-1}}, \tau\in T$, differ by one of  the regular saucer families    $(B_p,\zeta^{\wt\theta(||\tau||)}_p)$, $p=1,\dots, L=2^m$, from the finite list provided by Lemma \ref{lm:finite-contact}.   The   shell $\eta_{\Phi^{||\tau||}_0}$ is solid  for all $\tau\in T$, since $\Phi^{||\tau||}_0>0$ everywhere. Similarly the saucers  $(B_p,\zeta^{||\tau||}_p)$
 for $\tau\in\Op\p T$ are solid   for $\tau\in\Op\p T$, because we have  $\Phi^{||\tau||}_j\geq\Phi^{||\tau||}_{j-1}$ for all $j=1,\dots, |\Lambda|$.  
But the fibered saucer  corresponding to the family  $(B_p,\zeta^{\wt\theta(||\tau||)}_p)$ is equivalent to  ${}^T\eta_p$, $p=1,\dots, L$.
    \end{proof}

   
   \medskip
  
   \begin{proof}[Proof of Proposition \ref{prop:unique-model-param}]
 Proposition  \ref{prop:to-saucers-param}  allows us to   assume that ${}^T\xi_0$ is fibered  contact outside of a finite collection of disjoint    $\{B_i\}_{1\leq i\leq N}$, so that the restriction  for each $i=1,\dots, N$, $\xi_0|_{B_i}$ is a fibered  regular semi-contact saucer. Applying
Proposition \ref{prop:circular-to-universal-param} we further reduce the holes to a finite list of fibered saucers ${}^T\eta_p, p=1,\dots, L=2^m$. Proposition \ref{prop:saucer-to-circle-param} allows us to replace  each saucer ${}^T\eta_p$, $p=1,\dots, L,$ by a fibered circular model shell  ${}^T\eta_{K_p}$ defined by  a family of
  time-independent  contact Hamiltonian $K_p^\tau, \tau\in T$.
But then, using     Lemma \ref{lm:universal-eps} we conclude that each  circular model shell  ${}^T\eta_{K_p}$ dominates      the fibered    circular model   ${}^T\eta_{K_{\univ}}:=\{\eta_{K^{ (\theta(||\tau||)}_{\univ}})\}_{\tau\in T}$. 
\end{proof}


\section{Leafwise contact structures}\label{sec:foliations}
\label{sec:foliations-proofs}

 Theorem \ref{thm:existence-fol} follows from  Theorem \ref{thm:classif-fol} because any leafwise almost contact structure is homotopic to a structure from $\cont_\ot(\FF;  h_1,\dots, h_n))$  for an appropriate choice of embeddings $h_1,\dots, h_N$. Hence, it is sufficient to prove Theorem  \ref{thm:classif-fol}.
 
 We begin with the following lemma, which we already  used in Section \ref{sec:proof-main} in the  proof  of Theorem \ref{thm:main-existence-T}.
  \begin{lemma}\label{lm:choosing-paths} 
    Let $U$ be a   connected manifold of dimension $m>1$,  
     $T  $   a compact contractible set,  $T_1,\dots, T_k\subset T$ its  compact    subsets  such that
     \begin{description}
     \item{$(\star)$}   any      intersection 
      $T_{i_{1}}\cap \dots \cap T_{i_{p}}$ for $1\leq i_1<\dots<i_p\leq k$
      is  either empty or contractible. 
      \end{description}
      Let $B$ be a closed $m$-dimensional ball with a given point $p \in \p B$,
         $S_j:T_j\times B\to T_j\times U$, 
        $S_j(\tau,x)=(\tau,s_j(\tau,x))$, and $S_\pm:T\times B\to T\times U$, $S_\pm(\tau,x)=(\tau,s_\pm(\tau,x))$, be pairwise disjoint fiberwise smooth embeddings.
         Then there exists a fiberwise embedding $S:T\times[-1,1]\to T\times U$  such that
     \begin{enumerate}
    \item  $S(\tau,\pm 1)=S_\pm(\tau,p)$, $\tau\in T$;
      \item $S (T\times[-1,1])\cap\bigcup\limits_1^kS_j(T_j\times B)=\varnothing$;
      \item $S (T\times(-1,1))\cap\left(S_-(T\times B)\cup S_+(T\times B)\right)=\varnothing$.
\end{enumerate}
      \end{lemma}
     
    \begin{proof}

     We will  prove the statement by induction in $k$.
     When $k=0$ the statement follows from the fact that  the space of  maps of the contractible set $T$ into the space of pairs of  disjoint embeddings 
     of $B$ into $U$ is connected, and hence by a fiberwise isotopy we can assume that the embeddings $s_\pm(\tau,\cdot):B\to U$ are independent of $\tau$, i.e. $s_\pm(\tau,x)=\wt s_\pm(x)$ for all $(\tau,x)\in T\times B$.
     Then to construct the required embedding it is sufficient to connect the points $\wt s_\pm(p)$ by an embedded arc in $U$ which does not intersect the ball $\wt s_\pm(B)$ in its interior points.
       
    Suppose that  the statement is already proven for    $k=j$ (and any $U$). Suppose first that  one of the  $k=j+1$ sets $T_1,\dots, T_k$, say $T_k$, coincides with $T$. 
    By a fiberwise isotopy we can make the embedding  $s_k(\tau,\cdot):B\to U$ independent of $\tau$, i.e. $s_k(\tau,x)=\wt s_k(x)$ for all $(\tau,x)\in T\times B$. Therefore, the statement reduces to the case of $k-1=j$ sets
    $T_1,\dots, T_j$ and their embeddings into $\wt U=U\setminus\wt s_k(B)$, which is connected as well.
    
   Consider now    the general case. By an argument as above, we can assume that the embeddings $s_+(\tau, \cdot)$ are independent of $\tau$,
   i.e. $s_+(\tau,x)=\wh s_+(x)$ for all $(\tau,x)\in T\times B$. Denote $\wh U:= U\setminus \wh s_+(B)$.
    Suppose that $T_k$ is a proper subset of $T$.  
    Set $\wh T:= T_k$, $\wh T_i:= T_i\cap T_k$,   $\wh S_i:= S_i|_{\wt T_i\times B}$, $i=1,\dots, k-1$, $\wh S_-:= S_-|_{T_k\times B}$, $\wh S_+:= S_k|_{T_k\times B}$. Note that the sets $\wh T_i$, $i=1,\dots, k-1$, and $\wh T$ satisfy the condition $(\star)$.
    
    Considering $\wh S_i$ as embeddings into $\wh T_i\times\wh U$, and $\wh S_\pm$ as embeddings into $\wh T\times\wh U$ we can apply the inductional hypothesis to construct a fiberwise embedding 
    $\wh S:\wh T\times [-1,1]\to \wh T\times\wh U$ such that
    \begin{description}
       \item{--}  $\wh S(\tau,\pm 1)=S_\pm(\tau,\wh p)$, $\tau\in \wh T$, where $\wh p\in \p B$, is a point different from $p$;
      \item{--} $\wh S (\wh T\times[-1,1])\cap\bigcup\limits_1^{k-1}\wh S_j(\wh T_j\times B)=\varnothing$;
      \item{--} $\wh S (\wh T\times(-1,1))\cap\left(\wh S_-(\wh T\times B)\cup \wh S_+(\wh T\times B)\right)=\varnothing$.
\end{description}

   Using the embedding $\wh S$ we can make a fiberwise connected sum of the   embeddings $\wh S_\pm$ to construct a  fiberwise embedding $\wt S_-:T\times B\to T\times \wh U$ with the following properties:
   \begin{itemize} 
   \item $\wt S_-(T\times B)\cap \bigcup\limits_1^{k-1}S_j(T_j\times B)=\varnothing$;
\item $\wt S_-(T\times B)\supset S_-(T\times B)\cup S_k(T_k\times B)$;
\item the embeddings $\wt S_-$ and $S_-$ coincides   near  $T\times p\in T\times \p B$.
   \end{itemize}
   
Hence, by applying again  the inductional hypothesis to the  embeddings $\wt S_-, S$ and $S_j, j=1,\dots k-1$, we can construct a fiberwise  embedding
$S:T \times [-1,1]\to T\times U$ with the required properties.

        \end{proof}

\begin{proof}[Proof of Theorem  \ref{thm:classif-fol}]
  
Let $T$ be an $m$-ball. We need to prove that any  map  $$ (T,\p T)\to (\cont_\ot(\FF;  h_1,\dots, h_N)), \Cont_\ot(\FF; h_1,\dots, h_N)) $$ is homotopic rel. $\p T$ to a map into
$\Cont_\ot(\FF; h_1,\dots, h_N)$.
In other words, let $\xi_\tau\in\cont_\ot(\FF;  h_1,\dots, h_n)$, $\tau\in T$, be a family of leafwise almost contact structures which are genuine leafwise  contact structures for $\tau\in\p T$.
We will construct a homotopy rel. $\p T$ to a family of  genuine leafwise contact structures $\wt\xi_\tau, \tau\in T$. 

Consider a foliation $\wh\FF$ on $T\times V$ with leaves $\tau\times L$ where $\tau\in T$ and  $L$ a leaf of $\FF$.
Let  $\wh h_j:T\times T_j\times B\to T\times V$ be  the embeddings, given by
$$\wh h_j(\tau, \tau',x)=(\tau,h_j(\tau',x), (\tau,\tau',x)\in T\times T_j\times B, j=1,\dots, N.$$
The family $\xi_\tau,\tau\in T$, can be viewed as a leafwise almost contact structure $\Xi$ from $\cont(\wh \FF;\wh h_1,\dots,\wh h_N)$ which is genuine  on leaves $\tau\times L$ for $ \tau\in\p T$.
Moreover, we can assume that $\Xi$ is a  genuine leafwise contact structure on a neighborhood  $U\supset \p T\times V$ and neighborhoods  $U_j\supset \wh h_j(T\times T_j\times B)$, $j=1,\dots, N$.

There exists a triangulation $\TT$ of $T\times V$ with the following properties:
\begin{itemize}
\item  there are  compact {\em subcomplexes} $\wh U ,  \wh U_j$, $j=1,\dots, N$ of the triangulation $\TT$ such that $\p T\times V\subset\wh U\subset U,\;  \wh h_j(T\times T_j\times B)\subset\wh U_j\subset U_j$, $j=1,\dots, N$;
\item the  restriction $\TT_0$ of the triangulation  $\TT$ to $T\times V\setminus \Int(\wh U\cup\bigcup\limits_1^N\wh U_j)$  is  transverse to  the foliation $\wh\FF$;
\item  For every top-dimensional simplex $\sigma$ of $\TT_0$   there exists  a submersion  $\pi_\sigma:\Int\sigma\to B^{q+m}$   which is a fibration over an open $(q+m)$-ball  with the ball fibers, and    such  that  the pre-images $\pi_\sigma^{-1}(s), s\in  B^{q+m},$ are intersections of the  leaves of $\wt\FF$ with  $\Int\sigma$.\end{itemize}

  Applying   Gromov's parametric $h$-principle  for contact structures on open manifolds (see \cite{Gro69} and Theorem \ref{thm:Gromov-open}) inductively over skeleta of the triangulation we can deform $\Xi$, keeping it fixed on $\wh U\cup\bigcup\limits_1^N\wh U_j$ to make it leafwise genuinely contact in a neighborhood of the codimension $1$ skeleton of the triangulation 
  $\TT_0$.  
  
 Our next goal is to  further deform $\Xi$ on each top-dimensional  simplex $\sigma$ of the triangulation $\TT_0$, keeping it fixed on $\Op\p\sigma$,   to make in leafwise genuine contact structure on $\sigma$.    Let us choose one of such simplices.
  There exists  a compact subset $\overline\sigma\subset\Int\sigma$ such that the leafwise almost contact structure $\Xi$ is genuine on
   $\Op(\sigma\setminus\Int\overline\sigma)$ and 
   $\pi_\sigma|_{\overline\sigma}$ is a fibration over a closed  $(m+q)$-ball $X $ with fibers diffeomorphic to a closed $(2n+1)$-ball.

 Hence, $\Xi|_{\overline\sigma}$ can be viewed as a fibered  over $X$ almost contact structure on $\overline\sigma$, and applying Proposition \ref{prop:to-saucers-param}  we can further deform $\Xi$  keeping it fixed on  $\Op\p\sigma$,    to make it   genuine away from a finite number of  disjoint domains $Z_i$  fibered over    $X_i\subset X$ with piecewise smooth boundary, $i=1,\dots, K$. These domains are not necessarily disjoint but could be
chosen    arbitrary small and in such a way that all non-empty intersections $X_{i_1}\cap\dots\cap X_{i_k}$, $1\leq i_1<\dots<i_k\leq K$, are again balls with piecewise smooth boundaries. Let $Y$,  $Y\subset\overline\sigma$, and $Y_i,\; Y_i\subset Z_i,$ be subfibrations of the fibrations $\overline\sigma\to X$ and $Z_i\to X_i , \; i=1,\dots, K$, formed by boundaries of the corresponding ball-fibers.

Next, we use Lemma \ref{lm:choosing-paths}  to construct for each $X_j$ a fiberwise embedding 
$S_j:X_j\times[0,1]\to\overline  Z_j\setminus \bigcup\limits_{i\neq j}Z_i$ with $S_j(\tau,0)\in   Y_j$ and 
 $S_j(\tau,1)\in   Y$. 
 Recall that  by assumption  every point  $(\tau,x)\in   Y$ 
 can be connected to  a  point on the boundary of one of the 
 overtwisted balls 
 $B_{i,\tau,\tau'}:=h_i(\tau\times\times\tau' B_i), \; i=1,\dots, N,\;\tau\in T,$ by an embedded path  in the corresponding leaf.
  This path  can be chosen  inside an arbitrarily small  neighborhood  of the  codimension $1$ skeleton of the triangulation $\TT_0$.
   Hence,  if the sets $X_j$ are chosen sufficiently small  we can extend each of the  embeddings $S_j$ to a leafwise   embedding $\wt S_j:X_j\times[0,2]\to V$  
    such that 
 \begin{itemize}
 \item  $S_j(\tau,0)\in   Y_j$;
 \item   $S_j(\tau,2)\in h_i(T_i\times\p B) $ for some $i=i(\sigma,j)$.
 \end{itemize}
Moreover,   using  Proposition \ref{prop:neighb-ot} to increase the number of embeddings $h_i$ we can additionally arrange that
 the map $(\sigma,j)\mapsto i(\sigma,j)$  is injective.
    Then,  successively applying  Theorem \ref{thm:main-existence-T} to neighborhoods of
  $Z_j\cup S_j(X_j\times[0,2])\cup \bigcup\limits_{\tau\in X_j}h_{i(\sigma,j)}(S(\tau,2)\times B)$  for all top-dimensional simplices $\sigma$ of the triangulation we deform $\Xi$ to make it leafwise genuinely contact on these neighborhoods.
   \end{proof}



 \section{The overtwisted contact structures. Discussion}\label{sec:discussion}
 Recall the definition of an overtwisted contact structure from Section \ref{sec:ot-disc}:
 a contact structure $\xi$ on a manifold $M$ is called {\em overtwisted} if there is a contact embedding $(D_\ot,\xi_\ot)\to (M,\xi)$, see Section 
\ref{sec:ot-disc}. We note that the disc $D_\ot$ is only {\em piecewise smooth}.  We do not know if it is possible to characterize   overtwisted structures in dimension $>3$ by existence of a {\em smooth} overtwisted disc.  

In the $3$-dimensional case a contact structure which is overtwisted in our sense is also overtwisted in the  sense of \cite{Eli89}. This should be clear from the picture of the characteristic foliation on the disc $D_\ot$, see Figure~\ref{f:Disk}. 
The converse is also true. This can be seen directly by finding a copy of $(D_\ot,\xi_\ot)$ in a neighborhood of the traditional overtwisted disc, or indirectly, from the classification theorem from \cite{Eli89}.
Indeed,   one can first find a contact structure on the ball with standard boundary which contains $(D_\ot,\xi_\ot)$ and which is  in the standard almost contact class. Then, implanting this ball in an overtwisted contact manifold does not change the isotopy class of this structure.

The overtwisting property can be characterized in many other equivalent ways. We believe that the best characterization is yet to be found. Note, however, that {\em all  definitions  for which Theorem \ref{thm:main} holds are equivalent}. We describe below some possible  variations of the definition of the overtwisting property.

\subsubsection*{Changing $\Delta_\cyl$}
The definition of an overtwisted disc depends on a choice of the special contact Hamiltonian $K_\eps:\Delta_\cyl\to \R$, where $\Delta_\cyl=D\times[-1,1]$, where $D$ is the unit ball in $\R^{2n-2}$. 
Suppose $\wt D\subset\R^{2n-2}$ be any other star-shaped domain   with a piecewise smooth boundary. Denote
$$\wt\Delta_\cyl:= \{(x,z)\in\R^{2n-2}\times\R;\; x\in\wt D, |z|\leq 1\},\;\;
 \Delta_\cyl^-:=\Delta_\cyl\cap\{z\leq 0\} $$
 $$\hbox{and}\;\; \wt\Delta_\cyl^-:=\wt\Delta_\cyl\cap\{z\leq 0\}.$$ 

Let $C_+ $ be the space of continuous piecewise smooth functions
$\wt\Delta_\cyl^-\to \R$ which are positive on $\Op \p\wt\Delta_\cyl\cap\wt\Delta_\cyl^-$.

Given  two functions $K_\pm\in C_+ $ such that $K_-<K_+$
we denote
\begin{align*}
	U_{K_-,K_+} &= \{(x,v,t): K_-(x,t)\leq  v \leq  K_+(x,t), z(x)\leq 0\} \\&\subset (\wt\Delta_\cyl ^-\times T^*S^1, \ker(\lambda_{\st} + v\,dt)) \quad
	\mbox{and} \\
	\Sigma_{K_+} &= \{(x,v,t): 0 \leq v \leq K_+(x, t)\,,\,\, x \in \p \Delta, z(x)\leq 0\} \\&
	\subset (\Delta \times \R^2, \ker(\lambda_{\st} + v\,dt)).
\end{align*}
Gluing these pieces together via  the natural identification between their common parts  
we define  $\wh U_{K_-,K_+} :=U_{K_-,K_+}  \cup  \Sigma_{K_+}$.

\begin{lemma}\label{lm:base-change}
For any $K_+\in C_+(\wt\Delta)$ there exists  $K_-\in C_+(\wt\Delta)$, such that $K_-<K_+$ such that  $\wh U_{K_-,K_+}$ is overtwisted.
\end{lemma}
\begin{proof}
 
Choosing a representative $\eta$ of the contact germ along $\Sigma_{K_+}$ let $U$ be a neighborhood of $\p \wt\Delta_\cyl $ such that $K|_{U\cap\wt\Delta_\cyl^-}>0$ and the contact structure $\eta$ is defined on $\{(x,v,t); x\in U\cap\{z\geq 0\}, v\leq K_+(x)\}$.
There is a contact embedding
$\Phi:  \Delta_\cyl^-:= \wt\Delta_\cyl^-$ such that
$\Phi(\p\Delta_\cyl\cap \Delta_\cyl^-)\subset U$ and $\Phi(\Delta_\cyl^-\cap\{z=0\})\subset
\{z=0\}$.
Indeed,   the contact vector field $Z=L+z\frac{\p}{\p z}$, $L=\sum\limits_1^{n-1}\limits u_i\frac{\p}{\p u_i}$ is given by the contact Hamiltonian $z$
with respect to the standard contact form $\lambda_\st=dz+\sum\limits_1^{n-1}u_id\phi_i$.
Consider a cut-off function  $\sigma:\R^{2n-1}\to\R_+$ which is equal to $1$
on $\wt\Delta_\cyl\setminus U$ and supported in $\Int\wt\Delta_\cyl$, and let $\wt Z$ denote the contact vector field defined by the contact Hamiltonian $K:=z\sigma$.
Let us observe that $\wt Z$ is tangent to the hyperplane  $\{z=0\}$ because $K$ vanishes on this hyperplane. Let $Z^t$ and $\wt Z^t$ be the contact flows generated by $Z$ and $\wt Z$.
Then  the formula
$\Phi:= \wt Z^{\wt C}\circ Z^{-C}$ is the required  contact embedding
for appropriately chosen positive constants $C,\wt C$.

For   an appropriate choice of a  special Hamiltonian $K<K_\univ$ we have $\Phi_*K< K_+$.
On the other hand, there exists $K_-\in C_+$ such that $\Phi_*K>K_-$.
Hence, the overtwisted disc $D_\ot=D_K$ embeds into an arbitrarily small neighborhood of  $\wh U_{K_-,K_+}$, i.e.  $\wh U_{K_-,K_+}$ is overtwisted.
\end{proof}
\subsubsection*{Wrinkles and overtwisting}

Consider  the standard contact   $(\R^{2n+1},\xi_\st=\{dz+\sum\limits_1^{n-1} u_id\varphi_i -y_ndx_n=0\})$.
Let $B$ denote the unit ball in $\R^{2n+1}$ and 
  $w:B\to\R^{2n+1}$  be the {\it standard wrinkle} (see \cite{EliMi97}), i.e.\ a map given by the formula
$$(v, y_n,)\mapsto(v, y_n^3-3\alpha(r)y_n),\; v\in\R^{2n},$$
where $r:=||v||^2$, and
 $\alpha:[0,1]\to\R$  is a $C^{\infty}$-function which is positive on $(\frac14,\frac34)$, negative on $(\frac34, 1]$, constant   near $0$,  has a negative derivative at $\frac34$, and satisfies the inequality
$\alpha(r)\leq 1-r^2$. Denote $W:\{y_n^2\leq\alpha(r)\}$.

Note that Corollary \ref{cor:isocontact} allows us to construct a contact  embedding of  $(U\setminus W,(w^*\zeta)|_{U\setminus W})$  into  any overtwisted contact manifold of the same dimension. 
One can  also show, though we do not know a simple proof of this fact, that $(U\setminus W,(w^*\zeta)|_{U\setminus W})$ contains an overtwisted disc. Hence, {\em a contact structure $\xi$ on a manifold $M$   is overtwisted if and only  
if there exists a neighborhood $U\supset W$ in $B$, and  a contact embedding $(U\setminus W,(w^*\zeta)|_{U\setminus W})\to (M,\xi)$. }

  \subsubsection*{Stabilization of overtwisted contact manifolds}
Given a contact manifold $(Y,\xi)$ with a fixed contact form $\lambda$ its {\em $k$-stabilization}
is the contact manifold $Y_k^\stab:=Y\times\R^{2k}$ endowed with the contact structure $
\xi_k^\stab:=\{\lambda+\sum\limits_1^k v_id\phi_i=0\}$. It is straightforward to check  that up to a canonical contactomorphism the contact
manifold $(Y_k^\stab,\xi_k^\stab)$ is independent of the choice of the contact form $\lambda$. 

After the first version of the current paper was posted on arXiv, 
R.~Casals and F.~Presas observed that the  $k$-stabilization for any $k\geq 0$ preserves the overtwisting property, i.e.
\begin{theorem}[\cite{CMP}]
\label{thm:stab}
The $k$-stabilization  $(Y_k^\stab,\xi_k^\stab)$
of  any overtwisted contact manifold  $(Y,\xi)$  is overtwisted for any $k\geq 0$.
\end{theorem}

 In particular,  this implies that an overtwisted contact manifold of dimension $2n+1$ can be equivalently defined as a  contact manifold containing the   $(2n-2)$-stabilization of a neighborhood of the standard $3$-dimensional overtwisted disc.

 Note that  Theorem \ref{thm:stab} also implies
  \begin{corollary}\label{cor:stab2}
For  any overtwisted contact manifold $(M,\xi=\{\lambda=0\})$ the contact  manifold
 $(M\times T^*S^1, \{\lambda+vdt=0\}$ is  overtwisted. Moreover,
 $M\times T_+^*S^1:=(M\times T^*S^1)\cap\{v>0\}$ is overtwisted as well.
 \end{corollary}
 \begin{proof}     $(M\times T_+^*S^1, \lambda+vdt)$ is contactomorphic to
 $(M\times \R^2\setminus 0, \lambda+xdy-ydx$.  On the other hand,
 there exists a contact embedding $(M\times D^2_R,\lambda +xdy-ydx)\to (M\times \R^2\setminus 0, \lambda+xdy-ydx)$. It can be defined, for instance, by the formula $(w,x,y)\mapsto (\Reeb^{-2Ry}(w), x+2R,y)$, where $\Reeb^t$ is the Reeb flow of the contact form $\lambda$.
 But according to Theorem \ref{thm:stab}
 the product  $(M\times D^2_R, \{\lambda+xdy-ydx=0\})$ is overtwisted if the radius of the $2$-disc $D^2_R$ is sufficiently large, and the claim follows.
 \end{proof}

 \subsubsection*{Overtwisting and (non)-orderability}
In \cite{EliPolt00} there was introduced a relation $\leq$ on the universal cover
$\wt{\Cont}(Y,\xi)$ of the identity component of the group of contactomorphisms of $(Y,\xi)$. Namely, $f\leq g$  for $f,g\in  \wt{\Cont}(Y,\xi)$  if
there is a path in $\wt{\Cont}(Y,\xi)$ connecting $f$ to $g$ which is generated by a non-negative contact Hamiltonian.
This relation is either trivial  (e.g. in the case of the standard contact sphere of dimension $>1$, see \cite{EKP06}, and in this case the contact manifold $(Y,\xi)$ is called {\em non-orderable}, or it is a genuine partial order, e.g. in the case of $\R P^{2n-1}$ (see \cite{Giv}) or the unit cotangent bundle 
$UT^*(M)$ of a closed manifold $M$, see \cite{EKP06,  ChNem10}.
It is an open longstanding question whether a $3$-dimensional  closed overtwisted contact manifold  is orderable or not (see \cite{{CaPrSa10}, CaPr14} for partial results  in this direction). 

We define below   a possibly weaker notion of the above order  relation, for which all contact manifolds which  are  known to be orderable in the sense of $\leq$ are orderable as well, while overtwisted contact manifolds are unorderable.

Let us recall a property   of the relation
$\leq$ from \cite{EliPolt00}.
Given a contact manifold $(Y,\xi)$ with a fixed contact for $\lambda$ consider the contact manifold $(Y\times T^*S^1, \{\lambda+vdt=0\})$, where $S^1=\R/\Z$.    Let  $f\in\wt\Cont(Y,\xi)$  be generated by a time dependent contact Hamiltonian
$K_t:Y\to\R$ which can be assumed $1$-periodic in $t$.
We consider  the domain $V^+(f)=\{v+K_t(x)\geq 0, x\in Y\}\subset Y\times T^*S^1.$ 
If $f\leq g$ then there exists a contact diffeotopy 
  \begin{equation}\label{eq:order2}
  h_t:Y\times T^*S^1\to Y\times T^*S^1,\quad \hbox{such that} \quad h_0=\Id \quad\hbox{ and}\quad  h_1(V^+(f))\subset V^+(g).
  \end{equation}
 
 However it is not known whether the converse is true.
 Thus, it
 seems  natural  to introduce a weaker relation: {\em we say that $f\lessapprox g$ if there exists an isotopy $h_t$ as in \eqref{eq:order2}. }
If $f\leq g$ then $f\lessapprox g$ but we do not know whether these relations are equivalent or not. However, as we already stated above,  in all known to us cases of contact orderability in the sense of  $\leq $
 one can also prove {\em weak orderability} in the sense of $\lessapprox $. On the other hand,
 
 \begin{theorem}\label{thm:non-orderability}
 Any closed overtwisted contact manifold is not weakly orderable.
 \end{theorem}
 \begin{proof} According to Corollary \ref{cor:stab2} for a sufficiently large contact Hamiltonian $K>0$ the domain $V^+(K)$ is overtwisted. Hence, Corollary \ref{cor:isocontact} from   allows us to construct an isotopy of $V^+({2K})$ into $V^+(K)$ inside $Y\times T^*S^1$. This isotopy extends to a global diffeotopy, and hence   $g_{2K}\lessapprox g_K$, where $g_H$  denotes a time 1 map of the contact Hamiltonian $H$.
 On the other hand,  we clearly have $g_{K}\lessapprox g_{2K}$, i.e. the order $\lessapprox $ is trivial.
 
 \end{proof}

\subsubsection*{Classification of overtwisted contact structures on spheres}
 We will finish the paper by discussing the classification of overtwisted contact structures on $S^{2n+1}$ explicitly. Almost contact structures on the sphere $S^{2n+1}$ are classified by the homotopy group $\pi_{2n+1}(SO(2n+2)/U(n+1))$.
 The following lemma  computes  this group.\footnote{We thank Soren Galatius for providing this reference.}
 \begin{lemma}[Bruno Harris, \cite{Harris}]
 \label{lm:comput-sphere}  
 $$\pi_{2n+1}( SO(2n+2)/U(n+1))=\begin{cases}
 \Z/n!\Z,& n=4k;\cr
\Z,& n=4k+1;\cr
\Z/\frac{n!}2\Z,& n=4k+2.\cr
 \Z\oplus\Z/2\Z, & n=4k+3;\cr
 \end{cases}
 $$
  \end{lemma}

Thus,  Corollary \ref{cor:main}   implies that
  on spheres $S^{8k+1}$, $k>0$, there are exactly  $(4k)!$  different  overtwisted contact structures, on spheres $S^{8k+5}$, $k\geq 0$, there are $\frac{(4k+2)!}2$ different   overtwisted contact structures, while on all other spheres there are infinitely many. In  particular, there is a unique overtwisted contact structure on $S^5$. 
  
  It is interesting to note that
  $S^5$ has  infinitely many  
 {\it tight}, i.e.\ non-overtwisted contact structures. Besides  the standard contact structure,  these  are   examples  given  by  Brieskorn spheres (see \cite{Ust}).
 The full classification of tight contact structures on  any manifold of dimension $>3$  is an open problem.

  \end{document}